\documentclass[11pt,english,reqno]{amsart}
\usepackage[all]{xy}
\SelectTips{cm}{}
\usepackage[all]{xy}

\usepackage[top=1in, bottom=1.25in, left=1.25in, right=1.25in]{geometry}
\usepackage{mathtools}
\usepackage{subcaption}
\usepackage{bm}
\usepackage{xcolor}
\definecolor{link}{RGB}{0,100,80}
\definecolor{cite}{RGB}{169,40,55}
\usepackage{babel}
\usepackage{verbatim}
\usepackage{amsthm}
\usepackage{amstext}
\usepackage{amssymb}
\usepackage{amsbsy}
\usepackage{ifpdf}
\usepackage[nobysame,abbrev,alphabetic]{amsrefs}
\usepackage{enumerate}
\usepackage{amsmath}
\usepackage{amscd}
\usepackage{amsfonts}
\usepackage{graphicx}
\usepackage[all]{xy}
\usepackage{verbatim}
\usepackage{gensymb}
\usepackage{mathrsfs}
\usepackage[colorlinks]{hyperref}
\hypersetup{
linkcolor=link,          % color of internal links (change box color with linkbordercolor)
citecolor=cite,        % color of links to bibliography
}
\newtheorem{theorem}{Theorem}[section]
\newtheorem{conjecture}[theorem]{Conjecture}

\newtheorem{proposition}[theorem]{Proposition}

\newtheorem{problem}[theorem]{Problem}
\newtheorem{claim}[theorem]{Claim}
\newtheorem*{claim*}{Claim}
\newtheorem{lemma}[theorem]{Lemma}
\newtheorem{corollary}[theorem]{Corollary}
\newtheorem{definition}[theorem]{Definition}
\newtheorem*{definition*}{Definition}

\newtheorem{remark}[theorem]{Remark}
\theoremstyle{definition}\newtheorem*{acknowledgments}{Acknowledgments}
\numberwithin{equation}{section}
\newcommand{\al}{\alpha}
\newcommand{\be}{\beta}

\newcommand{\Ga}{\Gamma}
\newcommand{\del}{\delta}
\newcommand{\Del}{\Delta}
\newcommand{\lam}{\lambda}
\newcommand{\Lam}{\Lambda}
\newcommand{\eps}{\epsilon}

\newcommand{\ka}{\kappa}
\newcommand{\sig}{\sigma}
\newcommand{\Sig}{\Sigma}
\newcommand{\om}{\omega}

\newcommand{\vphi}{\varphi}
\newcommand{\cA}{\mathcal{A}}
\newcommand{\cB}{\mathcal{B}}
\newcommand{\cC}{\mathcal{C}}
\newcommand{\cD}{\mathcal{D}}

\newcommand{\cF}{\mathcal{F}}

\newcommand{\cH}{\mathcal{H}}

\newcommand{\cL}{\mathcal{L}}
\newcommand{\cM}{\mathcal{M}}
\newcommand{\cN}{\mathcal{N}}
\newcommand{\cO}{\mathcal{O}}
\newcommand{\cP}{\mathcal{P}}
\newcommand{\cQ}{\mathcal{Q}}

\newcommand{\cU}{\mathcal{U}}
\newcommand{\cV}{\mathcal{V}}

\newcommand{\cY}{\mathcal{Y}}

\newcommand{\bE}{\mathbb{E}}

\newcommand{\bP}{\mathbb{P}}
\newcommand{\bR}{\mathbb{R}}
\newcommand{\bZ}{\mathbb{Z}}
\newcommand{\bQ}{\mathbb{Q}}
\newcommand{\bF}{\mathbb{F}}
\newcommand{\bK}{\mathbb{K}}
\newcommand{\bN}{\mathbb{N}}
\newcommand{\bH}{\mathbb{H}}
\newcommand{\bT}{\mathbb{T}}

\newcommand{\gog}{\mathfrak{g}}

\newcommand{\gol}{\mathfrak{l}}
\newcommand{\gom}{\mathfrak{m}}

\newcommand{\gos}{\mathfrak{s}}

\newcommand{\gou}{\mathfrak{u}}
\newcommand{\gor}{\mathfrak{r}}
\newcommand{\goz}{\mathfrak{z}}

\newcommand{\goD}{\mathfrak{D}}
\newcommand{\gow}{\mathfrak{w}}

\newcommand{\OO}{\operatorname{O}}
\newcommand{\SL}{\operatorname{SL}}
\newcommand{\SO}{\operatorname{SO}}

\newcommand{\PGL}{\operatorname{PGL}}

\newcommand{\GL}{\operatorname{GL}}
\newcommand{\PSL}{\operatorname{PSL}}
\newcommand{\Gr}{\on{Gr}}

\newcommand{\Stab}{\on{Stab}}
\DeclarePairedDelimiter\av{\lvert}{\rvert}
\DeclarePairedDelimiter\norm{\lVert}{\rVert}
\DeclarePairedDelimiter\set{\{}{\}}
\DeclarePairedDelimiter\mset{(\{}{\})}
\DeclarePairedDelimiter\pa{(}{)}
\DeclarePairedDelimiter\idist{\langle}{\rangle}
\DeclarePairedDelimiter\br{[}{]}
\newcommand{\defi}{:=}
\newcommand{\comp}{\textrm{{\tiny $\circ$}}}

\newcommand\wt[1]{\widetilde{#1}}
\newcommand\wh[1]{\widehat{#1}}
\newcommand\on[1]{\operatorname{#1}}
\newcommand\diag[1]{\operatorname{diag}\left(#1\right)}

\newcommand\mb[1]{\mathbf{#1}}
\newcommand{\qfa}{\quad\textrm{for all}\ }

\newcommand\tb[1]{\textbf{#1}}
\newcommand\mat[1]{\pa*{\begin{matrix}#1\end{matrix}}}

\newcommand\crly[1]{\mathscr{#1}}

\newcommand{\Mat}{\operatorname{Mat}}

\newcommand{\spa}{\on{span}}
\newcommand{\supp}{\on{supp}}

\newcommand{\X}{X}
\newcommand{\Xbar}{\Gr_2(\bR^3)}
\newcommand{\Ad}{\on{Ad}}
\newcommand{\hT}{\hat{T}}
\newcommand{\hsig}{\hat{\sig}}
\newcommand{\hB}{B^{X,Z}}
\newcommand{\hBU}{B^{X,U}}
\newcommand{\hphi}{\hat{\Phi}}
\newcommand{\hp}{\mb{s}}
\newcommand{\hs}{\mb{s}}
\newcommand\bd[1]{\mb{#1}}

\newcommand{\hbeta}{\beta^{X,Z}}
\newcommand{\hE}{\hat{E}}

\newcommand{\mbe}{\on{E}}

\newcommand{\hmu}{\mu^{\otimes n}_{\bd s,U}}

\newcommand{\omrt}{\om_{\bR^3}}
\newcommand{\omrd}{\om_{\wedge^2\bR^3}}
\newcommand{\oml}{\om_{\gol_0}}
\newcommand{\omr}{\om_{\gor_0}}

\newcommand{\dx}{\on{d}_X}
\newcommand{\oxy}{o_{x,y}}
\newcommand{\ogy}{o_{gx,gy}}
\newcommand{\hbe}{\be^X}

\global\long\def\dv#1#2{\mathrm{d}{#1}\left({#2}\right)}

\newcommand\lwm[1]{(\be^{X,U})_{#1}^\Phi}
\newcommand\lwmz[1]{(\be^{X,Z})_{#1}^\Phi}

\global\long\def\Vang#1{{#1}_{\sphericalangle \al}[\om_{\gol_0}]}
\newcommand\spm[2]{\cP_{#1}({#2})}
\newcommand{\ph}{-}
\newcommand{\bxz}{B^{X,Z}}
\newcommand{\bexz}{\be^{X,Z}}
\newcommand{\bcxz}{\cB^{X,Z}}
\newcommand{\bxu}{B^{X,U}}
\newcommand{\bexu}{\be^{X,U}}

\newcommand{\wT}{\widehat{T}}
\newcommand{\mz}{\smallsetminus\set 0}

\newcommand{\onto}{\xymatrix{\ar@{>>}[r]&}}
\newcommand{\usadd}[1]{{\color{cyan}{\tiny [US]} #1}}

\begin{document}
\title[]{Dynamics on the space of 2-lattices in 3-space}
\author[]{Oliver Sargent}
\author[]{Uri Shapira}
\dedicatory{Dedicated to Ponyo}
\begin{abstract}
    %Building on the seminal work of Benoist and Quint we classify stationary measures on the space of 2-lattices
    %in 3-space under some assumptions on the acting measure.
    %In the case where the acting measure is Zariski dense in the special linear group the classification boils down
    %to the uniqueness of a stationary measure.
    %%%%
    We study the dynamics of $\SL_3(\bR)$ and its subgroups
    on the homogeneous space $\X$ consisting of
    homothety classes of rank-2 discrete subgroups of $\bR^3$. We focus on the case where the acting group is
    Zariski dense in either $\SL_3(\bR)$ or $\SO(2,1)(\bR)$.
    Using techniques of Benoist and Quint we prove that for a compactly
    supported probability measure $\mu$ 
    on $\SL_3(\bR)$ whose support generates a group which is Zariski dense in $\SL_3(\bR)$,
    there exists a unique $\mu$-stationary probability measure on $\X$. When the Zariski closure is $\SO(2,1)(\bR)$
    we establish a certain dichotomy regarding stationary measures and discover a surprising phenomenon:
    The Poisson boundary can be embedded in $X$. The embedding is of algebraic nature and
    raises many natural open problems.  Furthermore, motivating
    applications to questions in the geometry of numbers are discussed.
\end{abstract}
%%
%% Authors emails, addresses and acknowledgements
\address{Department of Mathematics\\
    Technion \\
    Haifa \\
Israel }
\email{ushapira@tx.technion.ac.il}
\email{o.g.sargent@gmail.com}
\thanks{}
\maketitle

%<---
%%\tableofcontents
%---> motivation
\section{Introduction}\label{sec:introduction}
\subsection{A motivating conjecture}
We begin by stating the conjecture which motivated this paper and remains unsolved.
Let $X_2$ be the space of lattices in $\bR^2$ identified up to scaling. The quotient $\OO_2(\bR)\backslash X_2$ of $X_2$
by the action of the orthogonal group is thought of
as the space of \textit{shapes of 2-lattices}.
Given a rank-2 discrete subgroup $\Lam\subset\bR^3$ -- hereafter known as a 2-lattice -- we define its shape $\mb{s}(\Lam)$ to be the
%$\OO_2(\bR)$-orbit
point of $\OO_2(\bR)\backslash X_2$ corresponding to
an image of $\Lam$ in $X_2$
obtained by choosing an arbitrary isometry between the plane spanned by $\Lam$ and $\bR^2$.
\begin{conjecture}\label{conj:1}
    Consider the signature (2,1) quadratic form $Q(v_1,v_2,v_3)\defi 2v_1v_3-v_2^2$  and the variety
    $V_Q^1 \defi \set{v\in \bR^3: Q(v) = 1}$.
    Let $V_Q^1(\bZ)\defi V_Q^1\cap \bZ^3$ denote the collection of integer points on $V_Q^1$.
    Then, the collection of \emph{orthogonal shapes}
    $$\set{\mb{s}(\bZ^3\cap v^\perp): v\in V_Q^1(\bZ)}$$
    is dense in $\OO_2(\bR)\backslash X_2$.
\end{conjecture}

Currently it is not  even known that the above set is unbounded.
Conjecture~\ref{conj:1} was motivated by a conjecture of Furstenberg which is related to 
a conjecture about cubic irrationals discussed in Appendix~\ref{appendix} (Conjecture~\ref{conj:2}). 
%which is related to Diophantine properties of cubic irrationals and will be discussed in Appendix~\ref{appendix}.
%A long standing open problem in the geometry of numbers is to decide if cubic irrationals have unbounded coefficients in thier continued fraction expansions. Conjecture~\ref{conj:2} provides a recasts the belief that cubic numbers have unbounded coefficients in their continued fraction expansions.
%Conjecture~\ref{conj:2} is similar to Conjecture~\ref{conj:1} but the surface $V_Q^1$ is replaced by the surface $\set{v\in\bR^3: N(v) =1}$ where $N$ is an integer polynomial
%of degree 3, not representing zero, obtained as a product of three linearly independent linear forms. \usadd{Currently
%in the appendix I use directional lattices and here orthogonal lattices and we need to decide if we want to make the discussions
%compatible of just say that they are basically the same}
Using duality it is easy to see that Conjecture~\ref{conj:1} would follow from the
density of the collection $\set{\mb{s}(g\Lam_{v_1}):g\in \SO(Q)(\bZ)}$, where $\Lam_{v_1} = \bZ^3 \cap v_1^\perp$ for
$v_1 = (1,1,1)\in V_Q^1(\bZ)$. See Figure~\ref{fig:1} for compelling evidence towards Conjecture~\ref{conj:1}.
In our figures we plot some numerical experiments. Since the more familiar space $\on{PSO}_2(\bR)\backslash X_2$
is a double cover of $\OO_2\backslash X_2$, we lift the plots to this space.

\begin{figure}[h]
    \begin{center}
        %\includesvg[height=21em,pretex=\relscale{0.7}]{nicepics/nicepic5-half}
        \includegraphics[width=5in]{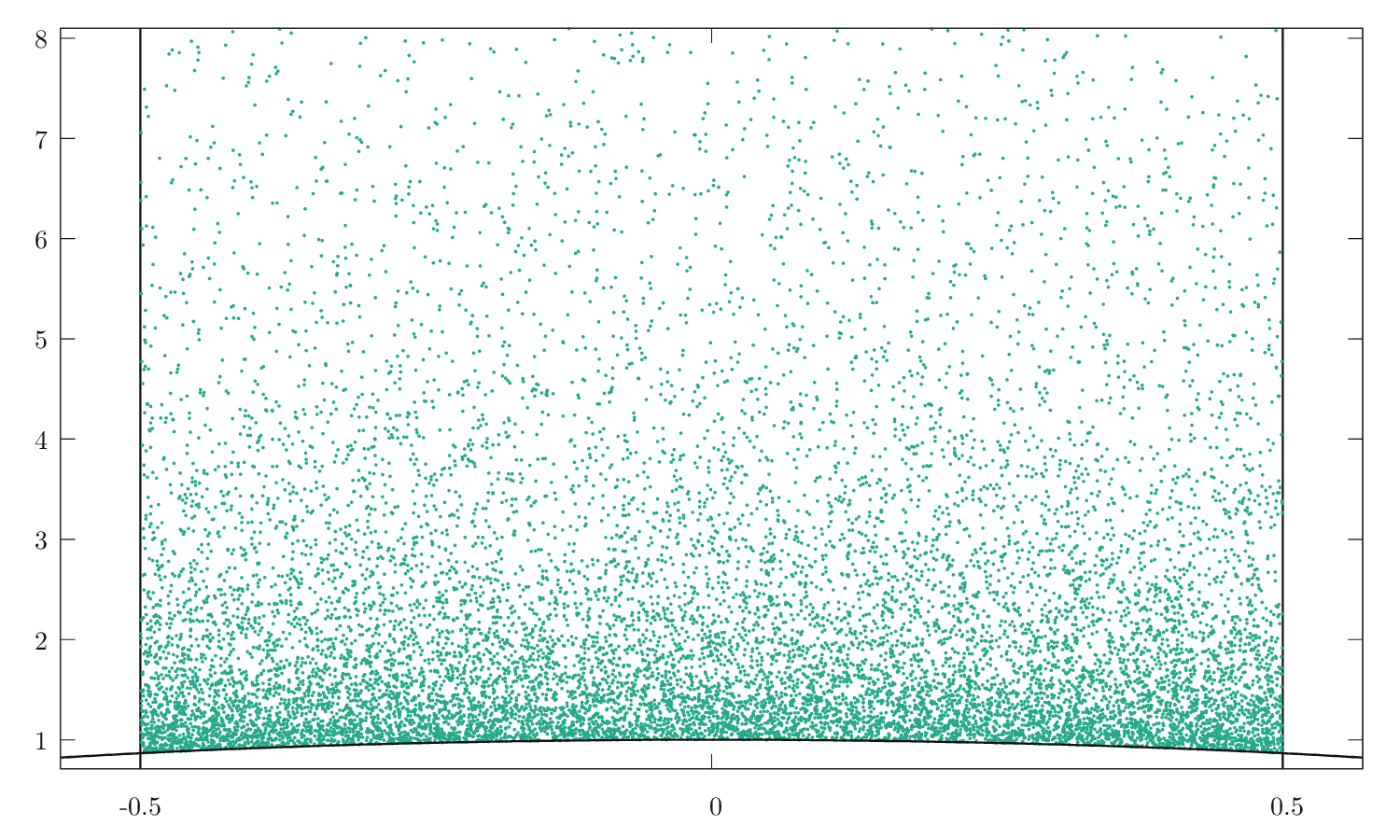}
        \caption{Plot of $\approx 15,000$ points in 
            $\PSL_2(\bZ)\backslash \bH\simeq \on{PSO}_2(\bR)\backslash X_2$
            %(which is a double cover of $\OO_2(\bR)\backslash X_2$)
            corresponding to the
            %(pre-images in the double cover of the)
            shapes $\mb{s}(g\Lam_{v_1})$
        where $g\in \SO(Q)(\bZ)$ is chosen `randomly'.}\label{fig:1}
    \end{center}
\end{figure}
%<---
%%
%---> first result

Motivated by the above discussion, we can now present a corollary of one of our main results. We consider the case where $\SO(Q)(\bZ)$ is replaced by a Zariski dense subgroup of $\SL_3(\bR)$.
\begin{theorem}\label{thm:example}
    Let $\Ga<\SL_3(\bR)$ be a compactly generated Zariski dense
    subgroup and let $\Lam<\bR^3$ be a rank-2 discrete subgroup. Then the collection of shapes
    $\set{\mb{s}(g\Lam):g\in \Ga}$ is dense in $\OO_2(\bR)\backslash X_2$.
\end{theorem}
\begin{remark}
    A much stronger statement holds in the setting of Theorem~\ref{thm:example}. Let $\mu$ be a compactly supported
    probability measure on
    $\SL_3(\bR)$ such that the group generated by the support of $\mu$ is $\Ga$.  Then for $\mu^{\otimes\bN}$-almost every $(g_1,g_2,\dots)\in \SL_3(\bR)^{\bN}$ the sequence $\mb{s}(g_n\cdots g_1\Lam)$ is equidistributed in $\OO_2(\bR)
    \backslash X_2$ with respect to the uniform measure on $\OO_2(\bR)\backslash X_2$.
    %obtained as the image of the unique $\SL_2(\bR)$-invariant probability measure on $X_2$.
\end{remark}

%%%%%%%%%%COMMETNING ends
%<---

Our attempt towards proving Conjecture~\ref{conj:1} involves studying random walks on the space of 2-lattices.
We build heavily on results and ideas from the seminal series of papers of Benoist and  Quint~\cite{BQAnnals, BQInventiones, BQAnnals2, BQJams} and
prove
two classification results regarding stationary measures on this space under assumptions on the acting group. Theorem~\ref{thm:example} is an immediate consequence of Theorem~\ref{thm:case1} which is a strong classification theorem stating the uniqueness of a
stationary probability measure  -- which we refer to as the \textit{natural lift} -- under the assumption that the acting measure generates a group which is Zariski dense in $\SL_3(\bR)$. The analogous
classification for the case when the Zariski closure is $\SO(Q)(\bR)$ is weaker in the sense that sometimes there are
stationary probability measures other than the natural lift. This is the reason we could not establish Conjecture~\ref{conj:1},
but it is not unlikely that further investigations of the structure of the space of ergodic $\mu$-stationary probability measures
will lead to the resolution of Conjecture~\ref{conj:1}. See Problem~\ref{prob:unique lifts}.

%<---

%<---
%%
%--->intro-statements
%%
%--->warmup
\subsection{Statements of results}\label{ssec:statements}
%The group $SL_3(\bR)$ acts naturally on $\X$ and given $\mu$ a probability measure on $\SL_3(\bR)$ we aim to classify $\mu$-stationary probability measures on $X$.
%Our first main result deals with the case when $\mu$ is a probability measure on $\SL_3(\bR)$ such that $\idist{\supp \mu}$ is Zariski dense in $\SL_3(\bR)$.
%\begin{theorem}\label{thm:case1}
%    Let $\mu$ be a compactly supported probability measure such that $\idist{\supp \mu}$ is Zariski dense in $\SL_3(\bR)$. Then there is a unique $\mu$-stationary probability measure on $X$.
%Furthermore, for any $x\in X$ we have that
%\begin{enumerate}
%\item The sequence $\frac{1}{n}\sum_{k=1}^n \mu^{*k}*\del_x$ converges to $\nu$.
%\item If $g_i\in G$ is a sequence chosen randomly and independently with law $\mu$ then with probability
%1 the sequence $\frac{1}{n}\sum_{k=1}^n\del_{g_k\cdots g_1 x}$ converges to $\nu$.
%\end{enumerate}
%\end{theorem}
For a topological space $Y$ we let $\cP(Y)$ denote the space of Borel probability measures on $Y$.
For $G\curvearrowright Y$ a continuous action of a topological group $G$ and $\mu\in\cP(G)$
we let $\cP_\mu(Y)$ be the subset of $\cP(Y)$ consisting of $\mu$-stationary measures.
%\osadd{I'm not sure whether we should use this notation or not... can't make up my mind - so far it is only in this section}
%if
%$$\mu * \nu\coloneqq\int_G g\nu\dv$$
%induces an action $G\curvearrowright \cP(Y)$ which in turn induces
%a \textit{convolution} map $\cP(G)\times \cP(Y)\to \cP(Y)$, $(\nu,\mu)\mapsto \nu*\mu$, where
%$\nu*\mu$ is the push-forward of $\mu\otimes \nu\in \cP(G\times Y)$ by the action map $(g,y)\mapsto gy$.
%Given $\mu\in\cP(G)$ we say that $\nu\in \cP(Y)$ is \textit{$\mu$-stationary} if $\mu*\nu = \nu$.

Henceforth we set $$G\defi \SL_3(\bR)$$
and for $\mu\in \cP(G)$
$$\Ga_\mu \defi \idist{\supp\mu}$$ will be the group generated by the support of $\mu$.
A measure $\nu\in\spm\mu X$ is said to be $\mu$-ergodic if the action of $\Ga_\mu$ on $(X,\nu)$ is ergodic.
It is a classical result of Furstenberg~\cite{MR0146298} (see~\cite[Chapter 4]{MR3560700} for a modern exposition) that if $\Ga_\mu$ acts
strongly irreducibly and proximally on $\bR^3$, then $\cP_\mu(\Xbar)$ consists of a single element.
We will refer to it as the \textit{Furstenberg measure} of $\mu$ on $\Xbar$ and denote it by $\bar{\nu}_{\Xbar}$.
\begin{remark}\label{rem:non-atomic}
    It will be important for us that the Furstenberg measure is atom free. This is ensured by the strong irreducibility assumption, since if there was
    an atom of $\bar{\nu}_{\Xbar}$ then the set of atoms with maximal weight is a finite $\Ga_\mu$-invariant set.
\end{remark}

We fix $\set{e_1,e_2,e_3}$ the standard orthonormal basis of unit vectors in $\bR^3$.
For $v\in\bR^3$ and $1\leq i \leq 3$ we will write $v_i\defi\idist{v,e_i}$.
As before we consider the indefinite quadratic form $Q:\bR^3\to\bR$ defined by
\begin{equation}\label{eq:0923}
    Q(v)\defi 2v_1v_3 - v_2^2.
\end{equation}

Let $H_\mu$ denote the Zariski closure of $\Ga_\mu$.
In what follows we will concentrate on two cases which will be referred to as~\ref{case1} and~\ref{case2} as follows:
\begin{align}
    H_\mu &= \SL_3(\bR) \tag{Case I}\label{case1}\\
    H_\mu &=\SO(Q)(\bR)\tag{Case II}\label{case2}.
\end{align}
In both of these cases it follows from a theorem of Gol'dsheid and Margulis (see~\cite[Theorem 5.1]{MR2405998} or~\cite{MR1040268}) that $\Ga_\mu$ acts strongly irreducibly and proximally on $\bR^3$.
For the rest of the paper $\X$ will be the space of rank-2 discrete subgroups in $\bR^3$ identified up to scaling.
The linear $G$-action on $\bR^3$ induces a transitive $G$-action on $X$ endowing it with the structure of a homogeneous space.
There is a natural projection $$\pi:X\to\Xbar$$ which sends an equivalence class of a 2-lattice to the plane it spans. We note that $\pi$ is $G$-equivariant.

Given a rank-2 discrete subgroup $\Lam\subset\bR^3$ we denote its equivalence class modulo scaling by $\br{\Lam}$. Abusing terminology we refer to both $\Lam$ and $\br{\Lam}$ as a 2-lattice and to $\X$ as the space of 2-lattices in $\bR^3$.
For each plane $p\in \Xbar$ the fibre $\pi^{-1}(p)\cong \SL_2(\bR)/\SL_2(\bZ)$.
This identification is not canonical and depends on choosing a linear isomorphism between $p$ and $\bR^2$. Still,
the unique $\SL_2(\bR)$-invariant measure on $\SL_2(\bR)/\SL_2(\bZ)$ translates to a well defined probability measure $m_p\in \cP(\pi^{-1}(p))$.
\begin{definition}\label{def:naturallift}
    Given a measure $\nu\in\cP(X)$ we can disintegrate $\nu$ with respect to the map $\pi$.
    The result is a collection of measures $\set{\nu_p}_{p\in\Xbar}\subset\cP(X)$ and a measure
    $\bar\nu\defi\pi_*\nu\in\cP(\Xbar)$ such that for $\bar{\nu}$-almost any $p\in\Xbar$, $\nu_p\in\cP(\pi^{-1}(p))$ and 
    $$\nu =\int_{\Xbar} \nu_p \dv{\bar{\nu}}p\in\cP(X).$$
    \begin{enumerate}[$\bullet$]
        \item{When $\nu_p=m_p$ for $\bar{\nu}$ almost any $p\in\Xbar$ we say that $\nu$ is the \textit{natural lift} of $\bar{\nu}$.}

        \item{In contrast, if there exists $k\in \bN$ such that $\nu_p$ is a uniform measure
                supported on a set of size
            $k$ for all $p\in \Xbar$, then we say that $\nu$ is a \textit{$k$-extension} of $\bar{\nu}$.}

        \item{We will also say that $\nu$ is a \textit{finite extension} of $\bar\nu$ if it is a $k$-extension of $\bar\nu$
            for some $k\in\bN$ which we do not specify.}
        \item{We also recall that given $\mu\in\cP(G)$,  $\nu$ is said to be a \textit{measure preserving extension} of $\bar\nu$ if $g\nu_p=\nu_{gp}$ for $\mu$-almost
            every $g\in G$ and $\bar\nu$-almost every $p\in\Xbar$.}
    \end{enumerate}
\end{definition}
Since $\pi$ is $G$-equivariant, given $\mu\in\cP(G)$ and $\nu\in\cP_\mu(\X)$, the push-forward $\pi_*\nu$ belongs to $\cP_\mu(\Xbar)$.
As noted earlier, we will only consider cases when $\Ga_\mu$ acts strongly irreducibly and proximally on $\bR^3$ so we can conclude that
$\pi_*\nu = \bar{\nu}_{\Xbar}$ is the Furstenberg measure. Our main result regarding~\ref{case1} is the following.
%<---
%%
%---> theorems
\begin{theorem}\label{thm:case1}
    Let $\mu$ be a compactly supported measure whose support generates a Zariski dense subgroup of $G$. Then
    the natural lift of the Furstenberg measure on $\Xbar$ is the unique
    $\mu$-stationary measure on $\X$.
    Furthermore, for any $x\in X$ we have that:
    \begin{enumerate}[(1)]
        \item The sequence $\frac{1}{n}\sum_{k=1}^n \mu^{*k}*\del_x$ converges to the natural lift.
        \item For $\mu^{\otimes\bN}$-almost every $(g_1,g_2,\dots)\in G^\bN$ the sequence $\frac{1}{n}\sum_{k=1}^n\del_{g_k\cdots g_1 x}$ converges to the natural lift.
    \end{enumerate}
\end{theorem}
The second part of Theorem~\ref{thm:case1} has the following immediate corollary.
\begin{corollary}
    Let $\Ga$ be a finitely generated discrete Zariski dense subgroup of $G$. Then the pre-image
    $\pi^{-1}(\supp\bar{\nu}_{\Xbar})$ is the unique $\Ga$-minimal subset in $X$.
\end{corollary}
Note that a non-discrete Zariski dense subgroup of $G$ is automatically dense in $G$ and thus the corollary is trivial for
such groups because $G$ acts transitively on $X$.

Theorem~\ref{thm:case1} should be compared with the main result of~\cite{BQAnnals} which is an analogous  
statement. The reason that the results of Benoist and
Quint fall short of being applicable to our discussion is that $X$ is not obtained as a quotient of a Lie group by a lattice but rather by
a closed group with non-trivial connected component.

In~\ref{case2} we also have the following result.
%Recall that if $I_{\bar{\nu}}(\set{\nu_p})$ is constructed as in Definition~\ref{def:naturallift}, then $I_{\bar{\nu}}(\set{\nu_p})$ is said to be a \textit{measure preserving extension} of $\bar{\nu}$ with respect to the $\mu$-action if
%for $\mu$-almost every $g\in G$ and for $\bar{\nu}$-almost every $p\in \Xbar$ we have $g_*\nu_p = \nu_{g p}$.
%To state our result we need to define a notion which is in some sense the extreme opposite of a natural lift. Given any $\nu\in \cP(\X)$
%one
%can disintegrate it with respect to $\pi$ and present it as $\nu = \int_{\Xbar} \nu_p d\bar{\nu}(p)$, where $\bar{\nu}=\pi_*\nu$.
%\begin{definition}
%We say that \textit{ $\nu$ is a $k$-cover of $\bar{\nu}$} if the $\nu_p$ is the equiprobability on a
%set of cardinality $k$
%for $\bar{\nu}$-almost every $p$ and we say that $\nu$ is a finite cover of $\nu$ if it is a $k$-cover for some integer $k$.
%We say that $\nu$ is a \textit{measure preserving extension of $\bar{\nu}$ with respect to the $\mu$-action} if
%for $\mu$-almost every $g$ and for $\bar{\nu}$-almost every $p$ we have $g_*\nu_p = \nu_{g p}$.
%\end{definition}
\begin{theorem}\label{thm:case2}
    Let $\mu$ be a compactly supported probability measure on $\SO(Q)(\bR)$ satisfying either one of the following:
    \begin{enumerate}[(a)]
        \item\label{case:zd}{The group generated by the support of $\mu$ is discrete and Zariski dense in $\SO(Q)(\bR)$.}
        \item\label{case:ac}{The measure $\mu$ is absolutely continuous with respect to the 
                Haar measure on $\SO(Q)(\bR)$ and contains the
            identity in the interior of its support.}
            %\item{A compactly supported probability measure such that the support of $\mu$ generates a discrete Zariski dense subgroup of $\SO(Q)(\bR)$.}
            %\item{A probability measure absolutely continuous with respect to the Haar measure on $\SO(Q)(\bR)$.}
    \end{enumerate}
    Then if $\nu$ is a
    $\mu$-ergodic $\mu$-stationary measure on $\X$ then either it is the natural lift or it is a measure preserving
    finite extension of the Furstenberg
    measure on $\Xbar$.

    \noindent Furthermore, for any $x\in X$ we have that:
    \begin{enumerate}[(1)]
        \item Any weak-* accumulation point of the sequence $\frac{1}{n}\sum_{k=1}^n \mu^{*k}*\del_x$ is a $\mu$-stationary probability measure
            on $X$.
        \item For $\mu^{\otimes\bN}$-almost every $(g_1,g_2,\dots)\in G^\bN$ any weak-* accumulation point of
            the sequence $\frac{1}{n}\sum_{k=1}^n\del_{g_k\cdots g_1 x}$ is a $\mu$-stationary probability measure.
    \end{enumerate}
\end{theorem}
\begin{remark}\label{rem:amplification}
    In fact,  in the proof of Theorem~\ref{thm:case2} 
    we will see that in the case $\mu$ satisfies assumption~\eqref{case:ac} the existence of 
    a finite extension is excluded and the natural lift is the unique $\mu$-stationary measure. See the last paragraph
    of \S\ref{sec:cover case}. 
    This implies that the second part of Theorem~\ref{thm:case2} yields a statement similar to Theorem~\ref{thm:case1}. 
\end{remark}
In Theorem~\ref{thm:case2} the assumptions about the measure~\eqref{case:zd} and~\eqref{case:ac} are there to ensure
that
%\osadd{discuss how to state this - looking at the proof you see that in case b the Poisson boundary is irrelevant}
$(\Xbar,\bar{\nu}_{\Xbar})$ is the Poisson boundary of $(\Ga_\mu,\mu)$ which is the actual assumption needed for the
part of the proof  appearing in \S\ref{sec:cover case}. See~\cite[Theorem 2.17, Theorem 2.21]{FurmanHandbook} and 
also~\cite[Theorem 5.3]{MR0146298}.

%group $\Ga_\mu$ is assumed to be discrete.
%This implies that
%$(\Xbar,\bar{\nu}_{\Xbar})$ is the Poisson boundary of $(\Ga_\mu,\mu)$ which is the actual assumption needed for the
%part of the proof  appearing in \S\ref{sec:cover case}.
%In particular, if $\mu$ is absolutly continuous with respect to the Haar measure on
%$\SO(2,1)(\bR)$ (see~\cite{} for details \usadd{we need to find a reference for this notion and for the statement saying that
%in such a case $H/P$ is the Poisson boundary}).

The existence of finite extensions is analogous to the existence of atomic stationary measures in the work of Benoist and Quint.
It seems to us that in many cases the existence of finite extensions can be excluded due to algebraic reasons.
%In a case
%the statement of Theorem~\ref{thm:case2}
%gets upgraded to a unique ergodicity statement similar to that of Theorem~\ref{thm:case1}.
%In fact, it seems plausible to expect that finite covers occur only as an arithmetic phenomenon (cf.\ Theorem~\ref{thm:cover} and Figure~\ref{fig:2}).

The lack of uniqueness in the classification part of Theorem~\ref{thm:case2} is what makes the conclusion regarding distributional
properties of individual orbits weaker than that in Theorem~\ref{thm:case1}. It is not clear to us if one should expect individual orbits
to equidistribute with respect to a single ergodic stationary measure or not (see Problem~\ref{prob:unique lifts}).
%<---
%<---
%%
%---> poisson
\subsection{Embedding of the Poisson boundary in $X$}\label{ssec:pboundary}
In this subsection we work under the assumption that we are in~\ref{case2} the quadratic form $Q$ is as in equation~\eqref{eq:0923} and we set $H=\SO(Q)(\bR)$.
For a long time we thought we could prove that the natural lift of the Furstenberg measure is the unique $\mu$-stationary measure in the setting of Theorem~\ref{thm:case2}. As Conjecture~\ref{conj:1} follows from such a statement, we announced
Conjecture~\ref{conj:1} as a theorem
in several talks and research proposals. A gap in the proof was pointed out to us by  Lindenstrauss and
after several failed attempts to close it we ran a computer experiment and immediately found an example of a
1-extension (see Theorem~\ref{thm:cover}). All the other examples that we can find are obtained from this example by means of finite index and we do not understand to what extent these objects are rare and what kind of structure they possess. See Problem~\ref{prob:unique lifts} and Remark~\ref{rem:Uri Bader's insight}. We now describe this simple example and urge the reader to ponder it as we find
it mind boggling.

In the following discussion and in Figure~\ref{fig:3} we will use the notation:
$$u^+(t)\defi \mat{1&t&t^2/2\\0&1&t\\0&0&1},\;\;
u^-(t) \defi \mat{1&0&0\\t&1&0\\t^2/2&t&1}\;\textrm{and }\;
k \defi \mat{0&0&1\\ 0 &-1&0\\ 1&0&0}.$$
Also, let $u_\pm\defi u^\pm(2)$ and
%$$u_+ =\mat{1&2&2\\ 0&1&2 \\ 0&0&1}\textrm{ and } u_- = \mat{1&0&0\\ 2&1&0\\ 2&2&1}.$$
note that $\Ga_0\defi\idist{u_+,u_-}$ is a finite index subgroup in the arithmetic group  $\SO(Q)(\bZ)$.
Hence it follows from the Borel Harish-Chandra Theorem~\cite{MR0147566} and the fact that $Q$ is defined over $\bQ$ 
that $\Ga_0$ is a lattice in $H$.
Let us denote also by $\cC\subset \Xbar$ the \textit{circle of isotropic planes}; that is, the set of planes $p\in\Xbar$ such that the 
restriction of $Q$ to $p$ is degenerate.
%there exists $v\in p\smallsetminus\set 0$ such that $Q(v)=0$. 
Note that $\cC$ is the unique $H$-invariant closed 
minimal subset of $\Xbar$
and $\cC = Hp_0$ where $p_0\defi\spa_\bR(\set{e_1,e_2})$. Since $\Stab_H(p_0)\defi P$ is a minimal parabolic subgroup of $H$, one can also think of $\cC$ as the full flag variety of $H$.
If $\mu\in\cP(H)$ is such that $\Ga_\mu$ is Zariski dense in $H$ then its Furstenberg measure is supported on the circle of isotropic 
planes~\cite[\S4]{MR3560700}.
\begin{theorem}\label{thm:cover}
    There exists a continuous $\Ga_0$-equivariant section
    %section of $\pi$,
    $\zeta:\cC\to \X$ (i.e.\ $\pi(\zeta(p))=p$ for all $p\in\cC$).
    In particular, if $\mu\in \cP(G)$ satisfies $\Ga_\mu = \Ga_0$, then $\zeta_*\bar{\nu}_{\Xbar}$
    is a $\mu$-stationary
    1-extension of the Furstenberg measure of $\mu$ on $\Xbar$.
\end{theorem}
\begin{proof}
    %Let $\Lam_0\coloneqq\spa_\bZ(\set{e_1,e_2})$.
    For $t\in\bR$ we define
    $$\Lam_t\defi\spa_\bZ(\set{e_1+te_2,e_2+2te_3})\quad\textrm{and}\quad\Lam_\infty\defi\spa_\bZ(\set{e_2,2e_3}).$$
    %$$\Lam_t = \smallmat{1&0&0\\t&1&0\\0&2t&0}\Lam_0,\quad
    %\Lam_\infty =\smallmat{0&0&0\\1&0&0\\0&2&0}\Lam_0.$$
    %Furthermore, let $x_t = \br{\Lam_t}, x_\infty = \br{\Lam_\infty}\in X$ be the corresponding homothety class.
    Note that $\lim_{t\to\pm\infty}\br{\Lam_t}=\br{\Lam_\infty}\in X$.
    %We can identify $\bP\bR^2$ with $\set{(1,t):t\in\bR}\cup\set{(0,1)}$
    Consider the map $\psi:\bR\cup\set{\infty}\to X$ given by $\psi(t) = \br{\Lam_t}$.
    %$$\xymatrix{
    %    & X\ar[d]^{\pi}\\
    %    \bP(\bR^2)  \ar[ur]^{\psi}\ar[r]_{\pi\comp\psi} & \Xbar
    %}$$
    Let $$g_1\defi\mat{1&0\\2&1}\quad\textrm{and}\quad g_2\defi\mat{1&1\\0&1}.$$
    There is an action $\SL_2(\bR)\curvearrowright\bR\cup\set\infty$ by fractional linear transformations, obtained by identifying
    $\bP\bR^2$ with $\bR\cup\set\infty$.
    We claim that
    \begin{equation}\label{eq:1invcirc}
        \psi\comp g_1=u_+\comp\psi\quad\textrm{and}\quad\psi\comp g_2=u_-\comp\psi.
    \end{equation}
    To show this we compute
    \begin{align*}
        u_+\br{\Lam_t}&= \br{\spa_\bZ(\set{(1+2t)e_1+te_2,(2+4t)e_1+(1+4t)e_2+2te_3})}\\
                      &= \br{\spa_\bZ(\set{(1+2t)e_1+te_2,(1+2t)e_2+2te_3})}\\
        %&= \br{\Lam_{t/(2t+1)}}\\
                      &= \br{\Lam_{g_1t}}
    \end{align*}
    and
    \begin{align*}
        u_-\br{\Lam_t}&= \br{\spa_\bZ(\set{e_1+(2+t)e_2+(2+2t)e_3,e_2+(2+2t)e_3})}\\
                      &= \br{\spa_\bZ(\set{e_1+(1+t)e_2,e_2+(2+2t)e_3})}\\
        %&= \br{\Lam_{1+t}}\\
                      &= \br{\Lam_{g_2t}}
    \end{align*}
    as required. Similar calculations also show that the above equalities hold true when $t=\infty$ and so~\eqref{eq:1invcirc} is verified.
    %\begin{align*}
    %    u_+x_t &= u_+\smallmat{1&0&0\\t&1&0\\0&2t&0}x_0= \smallmat{2t+1&4t+2&0 \\ t&4t+1&0\\0&2t&0}x_0
    %    =  \smallmat{2t+1&4t+2&0 \\ t&4t+1&0\\0&2t&0} \smallmat{1&-2&0\\0&1&0\\0&0&0} x_0\\
    %                     &=
    %    \smallmat{2t+1&0&0 \\ t&2t+1&0\\0&2t&0}x_0=x_{t/(2t+1)}
    %\end{align*}
    %where the last equality holds true even if $2t+1 = 0$. Similarly $u_+x_\infty = x_{1/2}$ which establishes the fact that
    %$\psi$ intertwines $\smallmat{1&0\\ 2&1}$ with $u_+$. Furthermore,
    %%\smallmat{1&0&0\\0&1&0\\0&0&1} + \Ad_{u_+}\smallmat{0&0&0\\t&0&0\\0&2t&0}
    %%=\smallmat{1&0&0\\0&1&0\\0&0&1} +
    %%\smallmat{2t+1&0&-4t\\t&2t+1&-6t\\0&2t&-4t+1} x_t = x_{t/(2t+1)}. $$
    %\begin{align*}
    %    u_-x_t&= u_-\smallmat{1&0&0\\t&1&0\\0&2t&0}x_0 =\smallmat{1&0&0\\2+t&1&0\\2+2t&2+2t&0} x_0
    %    = \smallmat{1&0&0\\2+t&1&0\\2+2t&2+2t&0}\smallmat{1 & 0 &0\\ -1&1&0\\0&0&0} x_0\\
    %                 & = \smallmat{1&0&0\\1+t&1&0\\ 0 &2+2t&0}x_0 = x_{t+1}.
    %\end{align*}
    %Similarly, $u_-x_\infty = x_\infty$, which establishes the fact that $\psi$ intertwines the action of $\smallmat{1&1\\0&1}$ with
    %$u_-$.
    %%=\smallmat{1&0&-4t\\0&1&-6t\\0&0&1}

    Since $\idist{g_1,g_2}$ is a lattice in $\SL_2(\bR)$ its action on $\bP\bR^2$ is
    minimal\footnote{This follows (for instance) from that fact that any parabolic subgroup of $\SL_2(\bR)$ acts
        minimally on $\SL_2(\bR)/\Gamma$, for every lattice $\Gamma$ in $\SL_2(\bR)$. 
    See for example~\cite[Proposition 1.5]{DaniRaghavan}.}
    It follows from the equivariance of $\pi$ and~\eqref{eq:1invcirc} that $\pi\comp\psi(\bP\bR^2)$ is a closed minimal $\Ga_0$-invariant set in $\Xbar$.
    It is thus equal to $\cC$ since the latter is the unique such set.
    Moreover, it is straightforward to check that $\pi\comp\psi$ is 1-1 which shows that there exists a continuous inverse
    $(\pi\comp \psi)^{-1}:\cC\to \bP\bR^2$.
    We then define
    $\zeta \defi \psi \comp (\pi\comp \psi)^{-1} : \cC\to X$ and note that from what we established so far it is clear that $\zeta$ is a
    $\Ga_0$-equivariant.
\end{proof}
\begin{remark}
    The remarkable feature of the section $\zeta$ from Theorem~\ref{thm:cover} is that it is $\Ga_0$-equivariant and not $H$-equivariant.
    Its image $\wt{\cC}\defi \zeta(\cC)$ is a $\Ga_0$-invariant circle  which intersects each fibre above the circle of isotropic planes in a single 2-lattice.
    See Figure~\ref{fig:2} for an illustration of the (lift of the) projection of $\wt{\cC}$ to $\on{PSO}_2(\bR)\backslash X_2$.
    Since $H$ acts minimally on $\pi^{-1}(\cC)$, $\wt{\cC}$ is not $H$-invariant.
    This minimality is one of the reasons we did not expect the existence of the section $\zeta$.
\end{remark}
\begin{remark}\label{rem:Uri Bader's insight} 
After presenting the above example to Uri Bader, he managed to explain it in a conceptual manner. It seems likely that
his insights could be used to resolve some of the problems presented in this paper. We expect this to be the subject of 
future work.
\end{remark}

\begin{figure}[h]
    \begin{center}
        %\includesvg[height=21em,pretex=\relscale{0.7}]{nicepics/circle-better}
        \includegraphics[width=2.in]{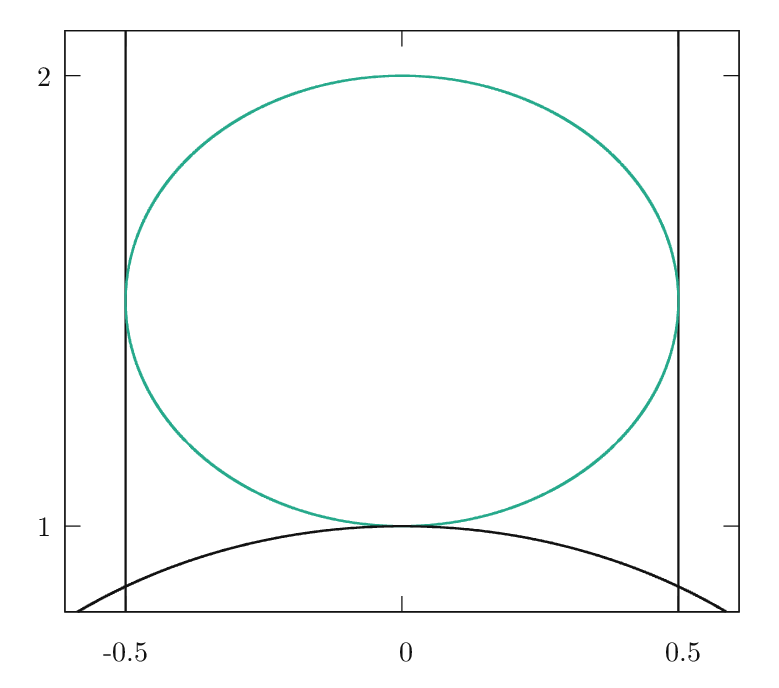}
        \caption{Plot of the (lift of the) projection of the $\Ga_0$-invariant set $\wt{\cC}=\zeta(\cC)\subset X$
        in $\PSL_2(\bZ)\backslash \bH\simeq \on{PSO}_2(\bR)\backslash X_2$.}\label{fig:2}
    \end{center}
\end{figure}
As this introduction is quite long, we do not dwell on the comparison between the results here and similar
classification results of stationary measures on homogeneous spaces.
Nevertheless, this comparison is essential if one
wishes to shape a reasonable set of expectations regarding stationary measures and closed invariant sets of 
semisimple groups acting on spaces such as $X$.
In particular, in the case where the acting measure generates a Zariski dense subgroup in a semisimple group, 
one should compare our results with the seminal works of
Benoist and Quint~\cite{BQAnnals, BQJams, BQInventiones, BQAnnals2}, Eskin and Margulis~\cite{MR2087794},
Bourgain-Furman-Lindenstrauss-Mozes~\cite{BFLM} and~\cite{EskinMirzahani}. 
The recent preprint~\cite{EskinLindenstrauss} is also highly relevant and as it provides an alternative proof for some of the
results of Benoist and Quint, potentially, the techniques 
introduced there can simplify the analysis carried out in this paper.
See also 
Simmons and Weiss~\cite{2016arXiv161105899S} for results pertaining to non-semisimple Zariski closures.
Compare also, the more classical results
regarding measures on projective spaces originating from the seminal work of Furstenberg~\cite{MR0163345,MR0284569}, Furstenberg-Kesten
~\cite{MR0121828} and Furstenberg-Kifer~\cite{MR727020}.
For potential applications of such classification results see~\cite{2016arXiv161105899S}.
%\usadd{Find a place to cite also Amos?}
In the opposite case when the acting measure has certain smoothness properties one can juxtapose our results with those of 
Nevo and Zimmer~\cite{MR1933077,MR1919409}. 

We wish to stress, as this cannot be stressed enough,
that we follow closely the exposition and methods developed in~\cite{BQJams}. Our main work was
to overcome technical difficulties arising from the fact that $X$ is obtained as a quotient by a group
with a non-trivial connected component. Other than that we mainly needed
to downgrade the generality of their discussion and hopefully maintain the quality of presentation.

In future work we plan to generalise the results of this paper and analyse actions of 
discrete groups on spaces 
with features similar to $X$. 
These include the space of homothety classes of lattices in $k$-planes in $\bR^n$ but more
generally bundles over projective spaces with fibres obtained as quotients of a Lie group by a lattice.

We conclude this introduction by stating some natural open problems and presenting
figures pertaining to~\ref{case2}.
\begin{problem}\label{prob:unique lifts}
    Let $\mu\in\cP(H)$ be a finitely supported measure
    such that $\Ga_\mu$ is Zariski dense in $H = \SO(Q)(\bR)$.
    \begin{enumerate}[(1)]
        \item Is it true that if $\Ga_\mu$ is dense in $H$, or if $\Ga_\mu$ is cocompact in $H$, then the natural lift
            of the Furstenberg measure is
            the unique $\mu$-stationary measure on $\X$?
        \item If $k_i$ is a sequence of natural numbers such that $k_i\to\infty$ and $\nu_i\in \cP(\X)$ is a $\mu$-stationary $k_i$-extension of the Furstenberg measure on $\Xbar$ is it true that $\nu_i$ converges
            to the natural lift of the Furstenberg measure?
        \item For $x\in X$, does the set of accumulation points  $\overline{\Ga_\mu x}\smallsetminus \Ga_\mu x$ of the orbit $\Ga_\mu x$ support a $\mu$-ergodic $\mu$-stationary probability measure?
        \item Is it true that for any $x\in X$ the sequence $\frac{1}{n}\sum_{k=1}^n\mu^{*k}*\del_x$ converges to a $\mu$-ergodic $\mu$-stationary probability measure?
        \item Is it true that for any $x\in X$ and $\mu^{\otimes\bN}$-almost every $(g_1,g_2,\dots)\in G^\bN$ that the sequence
            $\frac{1}{n}\sum_{k=1}^n \del_{g_k\cdots g_1x}$ converges to a $\mu$-ergodic $\mu$-stationary probability measure?
        \item Is it true that if $\nu$ is a $\mu$-ergodic $\mu$-stationary probability measure on $X$ which is a $k$-extension of
            $\bar{\nu}_{\Xbar}$, then there exists a copy of the circle $\wt{\cC}\subset  X$ such that $\pi:\wt{\cC}\to \cC$ is
            a covering map of degree $k$ and $\nu(\wt{\cC}) = 1$?
    \end{enumerate}
\end{problem}
%%%%%%%%%%%

%%%%%%%%%%%
%%%%%%%%%%
%\begin{comment}
\begin{figure}[h]
    \begin{subfigure}{0.45\textwidth}
        \centering
        %\includesvg[height=13em,pretex=\relscale{0.5}]{nicepics/nicepic2-small}
        \includegraphics[width=2.1in]{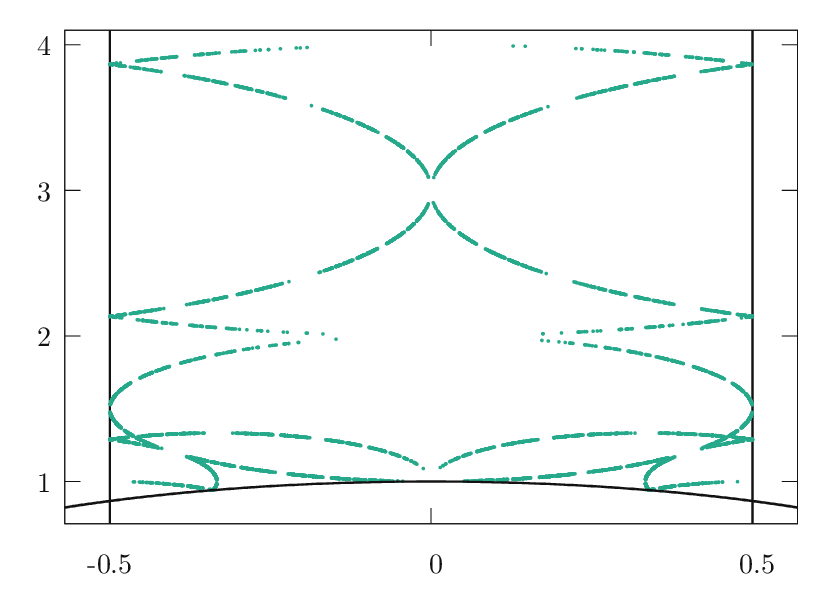}
        \caption{$\Ga=\idist{u^+(2), u^{-}(1)}$}
        \label{fig:subim1}
    \end{subfigure}
    \begin{subfigure}{0.45\textwidth}
        \centering
        %\includesvg[height=13em,pretex=\relscale{0.5}]{nicepics/nicepic7-small}
        \includegraphics[width=2.1in]{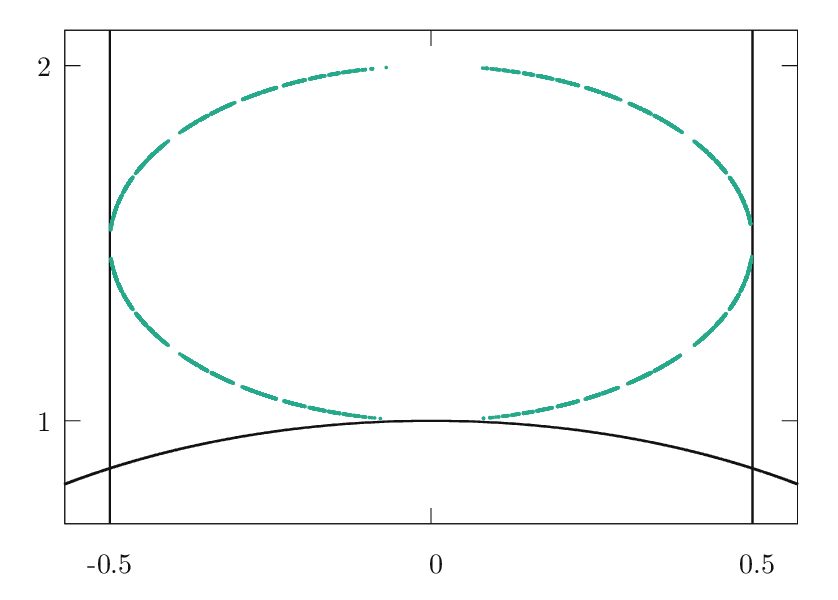}
        \caption{$\Ga=\idist{u^+(1), u^-(2)}$}
        \label{fig:subim4}
    \end{subfigure}
    \begin{subfigure}{0.45\textwidth}
        \centering
        %\includesvg[height=13em,pretex=\relscale{0.5}]{nicepics/nicepic3-small}
        \includegraphics[width=2.1in]{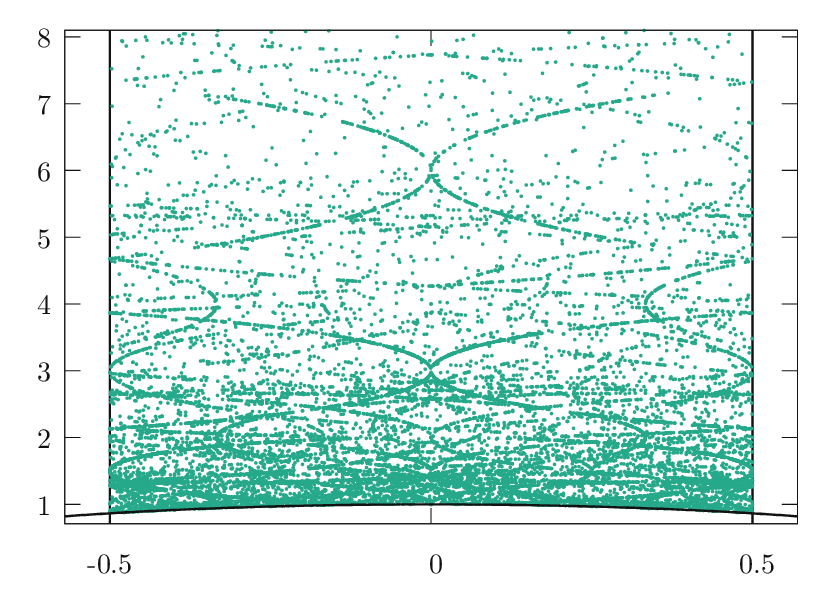}
        \caption{$\Ga=\idist{u^{+}(2), u^{-}(1), k}$}
        \label{fig:subim2}
    \end{subfigure}
    \begin{subfigure}{0.45\textwidth}
        \centering
        %\includesvg[height=13em,pretex=\relscale{0.5}]{nicepics/nicepic6-small}
        \includegraphics[width=2.1in]{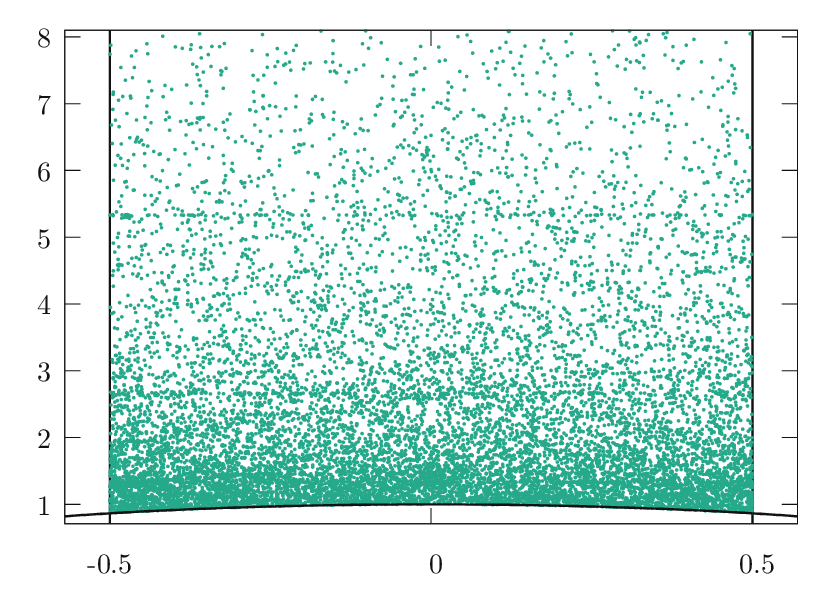}
        \caption{$\Ga=\idist{u^+(1), u^-(1)}$}
        \label{fig:subim3}
    \end{subfigure}
    \caption{Plots in $\PSL_2(\bZ)\backslash \bH$ of the projections of random points in the $\Ga$-orbit
        of $\br{\Lam_0}$ for various choices of $\Ga$.
    }
    \label{fig:3}
\end{figure}
%\end{comment}
%<---
%%
%<---
%%
%--->S2
%%
%--->general setup

\begin{acknowledgments}
    We would like to express our gratitude to Elon Lindenstrauss for correcting a mistake in an earlier draft.
    We would also like to thank Uri Bader, Yves Benoist, Alex Eskin, Alex Furman,  Elon Lindenstrauss,
    Amos Nevo, Jean-Fran\c cois Quint, Ron Rosenthal, Nicolas de Saxc\'e, Barak Weiss and Cheng Zheng 
    for their support encouragement and  
    assistance. We acknowledge the support of ISF grant 357/13 and the warm hospitality and splendid environment provided
    by MSRI where some of the research was conducted during the special semester Geometric and Arithmetic Aspects of 
    Homogeneous Dynamics held on 2015.
\end{acknowledgments}
\section{Generalities}\label{sec:generalities}
Throughout the paper $\mu\in \cP(G)$ is compactly suppoered, $\Ga\defi\idist{\supp \mu}$ is the group generated by its support
and $H$ is the Zariski closure of $\Ga$. Furthermore, we assume we are either in~\ref{case1} or~\ref{case2}.
Given $k,l\in\bN$ with $k<l$ and elements $b_k,\dots, b_l \in G$ we use
the following notation to denote products 
$$b_k^l\defi b_k\cdots b_l\quad\textrm{and}\quad b_l^k\defi b_l\cdots b_k.$$
%Similarly, we use the notation $(b_k^\ell)^{-1}$ (resp.\ $(b_\ell^k)^{-1}$)
%to denote either the inverse of the product $b_k^\ell$ (resp.\ $b_\ell^k$), or the reversed sequence of inverses
%$b_\ell^{-1},\dots ,b_k^{-1}$ (resp.\ $b_\ell^{-1},\dots, b_k^{-1}$). The reader should understand from the context if products or
%sequences are considered. For infinite sequences $b_1,b_2,\dots$ we shall use the notation $b_1^\infty$ and most often
%denote $b = b_1^\infty$. Given an infinite sequence
%$b=b_1^\infty$ and a finite sequence $a_1^n$ we denote $a_1^nb$ the corresponding infinite sequence obtained by concatenation.

\subsection{A restatement and the structure of the paper}
For convenience of reference we aim to state a unified theorem whose statement captures both Theorem~\ref{thm:case1} and Theorem~\ref{thm:case2}. In order to do so we need to define some more objects.
%Introduce the limit measures and the backwards dynamical system (notation and references mainly).
%
For details regarding the following facts we refer the reader to~\cite[\S2.5]{MR3560700}.
Let $A\defi \supp \mu$ and $B \defi A^\bN$ be the space of infinite sequences indexed by the positive integers. Let
$\be\defi \mu^{\otimes\bN}$ be the Bernoulli measure and $S:B\to B$ be the shift map $Sb = (b_2,b_3,\dots)$, where 
$b=(b_1,b_2,\dots)$.
Given $\nu\in \spm\mu X$ it is well known that for $\be$-almost every $b\in B$
the sequence $b_1^n\nu$ converges to a probability measure denoted $\nu_b$ known as the \textit{limit measure} of $\nu$ with respect to $b$.
Hence, the map $b\mapsto \nu_b$ is almost surely well defined and equivariant in the sense that
$\nu_b = b_1\nu_{Sb}$ for $\be$-almost every $b\in B$. Moreover, one can recover the measure $\nu$ by integrating $$\nu = \int_B \nu_b \dv\be b.$$
The following theorem is our unified statement and the reader can readily check that Theorems~\ref{thm:case1} and~\ref{thm:case2} follow from it.
\begin{theorem}\label{thm:full}
    Let $\mu\in\cP(G)$ be a compactly supported measure and suppose we are in~\ref{case1} or~\ref{case2}.
    Let $\nu\in \spm\mu X$ be $\mu$-ergodic.
    \begin{enumerate}[(a)]
        \item\label{main1} If for $\be$-almost every $b\in B$ the limit measure
            $\nu_b$ is non-atomic, then $\nu$ is the natural lift of the Furstenberg measure of $\mu$ on $\Xbar$.
        \item\label{main2} In~\ref{case1} it holds that for $\be$-almost every $b\in B$ the limit measure $\nu_b$ is non-atomic.
        \item\label{main3} In~\ref{case2}, if it does not hold that  
            for $\be$-almost every $b\in B$ the measure $\nu_b$ is non-atomic and if
            $\Ga$ is discrete or if $\mu$ is absolutely continuous with respect to the Haar measure on $H$ and contains 
            the identity in the interior of its support, then
            $\nu$ is a measure preserving finite extension of the Furstenberg measure of $\mu$ on $\Xbar$.
        \item\label{main4} In both~\ref{case1} and~\ref{case2} for any $x\in X$, any weak-* limit point of the sequence 
            $\frac{1}{n}\sum_{k=1}^n \mu^{*k}*\del_x$ is an element of $\spm \mu X$.
            Moreover, for $\be$-almost every $b\in B$ any weak-* accumulation point of
            the sequence $\frac{1}{n}\sum_{k=1}^n\del_{b_k^1x}$ is an element of $\spm \mu X$.
    \end{enumerate}
\end{theorem}
The rest of the paper is devoted to the proof of Theorem~\ref{thm:full}.
%The reader can consult the following diagram
%illustrating
%the logical dependencies between the sections.
%$$\xymatrix{
%    &  \S\ref{sec:generalities} \ar[dl] \ar[dr] & \\
%    \S\ref{sec:drift}\ar[d] \ar@{.>}[rr] \ar@{.>}[rrd]&  &\S\ref{sec:Non-escape-of-mass/contracted}\ar[d]\\
%    \S\ref{sec:cover case}& & \S\ref{sec:non atomicity}
%}$$
In the rest of \S\ref{sec:generalities} we collect notation and results
needed for the rest of the paper.

We establish part~\eqref{main1} of Theorem~\ref{thm:full}
in \S\ref{sec:drift} by means of the so called \textit{exponential drift} argument of Benoist and Quint.

We establish part~\eqref{main3}
of Theorem~\ref{thm:full} in \S\ref{sec:cover case}.
To do this we will use a result of Ledrappier~\cite{LedrappierIMJ}
%\usnote{possibly need to refer to another result saying that extensions of the Poisson boundary are measure preserving}showing
that in this case the measure space $( \Xbar,\bar\nu_{\Xbar} ) $ is the Poisson boundary of $(\Ga,\mu)$.
We note that part~\eqref{main3} of the theorem must be taken into account in conjunction with Theorem~\ref{thm:cover}
which says that this possibility is not vacuous.

We establish part~\eqref{main4} of Theorem~\ref{thm:full} at the end of
\S\ref{sec:Non-escape-of-mass/contracted}.
Given the analysis of Benoist and Quint~\cite{BQAnnals2} the proof boils down to
a non-escape of mass result which is proved in
\S\ref{sec:Non-escape-of-mass/contracted}.
The aim is to show that a certain function on $X$, which can be thought of as a height function,
tends to be contracted by the random walk. This will enable us to prove that
the `cusp' in $X$ is `unstable' with respect to the action induced by $\mu$.

Finally we establish part~\eqref{main2} of Theorem~\ref{thm:full} in \S\ref{sec:non atomicity} using an argument which was developed by Benoist and Quint in~\cite{BQJams}.
The main point is to show that the diagonal in $\X \times \X$ is `unstable'.
%The dotted arrows in the above diagram are there to indicate that although \S\ref{sec:Non-escape-of-mass/contracted},\ref{sec:non atomicity} do not
%rely on \S\ref{sec:drift}, one derives the actual classification of stationary measures only
%by combining the results of these three sections.
% \S\ref{sec:drift} with those of sections
%\ref{sec:Non-escape-of-mass/contracted},\ref{sec:non atomicity}.
%<---
%%
%--->boundary maps
\subsection{The boundary map and other equivariant maps}\label{ssec:boundary maps}
When studying $\mu$-stationary probability measures one is naturally led to consider
equivariant maps $\zeta:B\to Y$ for various spaces $Y$ on which $\Ga$ acts. Here equivariant
means that for $\be$-almost every $b\in B$ one has $\zeta(b) = b_1\zeta(Sb)$. The reason for this
is that given such an equivariant map, the measure $\nu = \zeta_*\be$ belongs to $\spm \mu Y$ and the limit
measures $\nu_b$ are equal to $\del_{\zeta(b)}$ for $\be$-almost every $b\in B$.

In order to proceed we must choose a minimal parabolic subgroup of $H$.
In both~\ref{case1} and~\ref{case2} the subgroup of $H$ consisting of upper triangular elements is a minimal parabolic subgroup of $H$.
We will denote this subgroup by $P$. 
$$\textrm{Explicitly, in \ref{case2}: } P=\left\{\mat{e^s&0&0\\ 0&1&0\\0&0&e^{-s}}\mat{1&t&t^2/2\\0&1&t\\0&0&1}:t,s\in \bR\right\}.$$
%\osadd{Alternatively, we could have defined $P$ by noting that in~\ref{case1} we have $P=\Stab_H(F_0)$ and in~\ref{case2} we have $P=\Stab_H(F_0')$ where
%    $$F_0\coloneqq 0\subset\spa_\bR(\set{e_1})\subset\spa_\bR(\set{e_1,e_2})\subset\bR^3$$ is a full flag and
%    $$F_0'\coloneqq 0\subset\spa_\bR(\set{e_1})\subset\bR^3$$ is a full flag of totally isotropic subspaces for the quadratic form $Q$.
%}
By~\cite[Proposition 10.1]{MR3560700} the set $\spm \mu {H/P}$ consists of a single measure and it is $\mu$-proximal. 
This implies that there is a unique measurable equivariant map $$\xi:B\to H/P$$ which is referred to as the \textit{boundary map}.
%As previously noted, it follows that $\xi_*\be$ is the unique $\mu$-stationary measure on $H/P$.

%    The following definition will be convenient.
%
%    \osadd{I'm not sure exactly why we need this def yet so I'm leaving it - needs editing though}
%    \begin{definition}\label{def:B0}
%        We choose once and for all a measurable $S$-invariant set of full $\be$-measure $B_0\subset B$ such that
%        For any $b\in B_0$ the flag $\xi(b)$ is well defined and the equality $\xi(b) = b_1\xi(Sb)$ holds. Sometimes, when considering
%        a $\mu$-stationary measure $\nu$, we will assume further that or any $b\in B_0$ the measure $\nu_b$ is well defined and the equality $\nu_b = b_1\nu_{Sb}$ holds.
%    \end{definition}
%    All other equivariant maps we will consider will be obtained as compositions
%    of $\xi$ with $H$-equivariant maps from the flag variety $H/P$ to other spaces on which $H$ acts.
%    In particular, they will be defined on $B_0$ and satisfy the equivariance relation there.
%
The mechanism giving rise to the equivariant maps we will consider is as follows:
Let $V$ be
a representation of $H$.
If $W_0\subseteq V$ is a subspace of dimension $d$ which is
$P$-invariant then there is a well defined $H$-equivariant map $H/P\to \on{Gr}_d(V)$ defined by $$hP \mapsto hW_0.$$
For $\eta\in H/P$ we then denote the image of $\eta$ by $W_\eta$.
In turn,  the composition of this map with $\xi$ gives rise to an equivariant map $B\to \on{Gr}_d(V)$ given by $$b\mapsto W_{\xi(b)}$$
for $\be$-almost every $b\in B$.
Hence, for $\be$-almost every $b\in B$ we define $W_b\defi W_{\xi(b)}$.

%    We make the following concrete choices of the
%    minimal parabolic $P$ according to the case at hand:
%    In~\ref{case1}, $H=G$ and we choose
%    \begin{align}
%        \label{eq:pc1} P = \set{\smallmat{*&*&* \\0&*&*\\0&0&*}\in H}.
%    \end{align}
%    In~\ref{case2}, $H=\SO(Q)$ (where $Q$ is as in~\eqref{eq:0923}) and we choose
%    \begin{align}
%        \label{eq:pc2} P = \set{
%        \smallmat{\al &\al t&\frac{\al t^2}{2} \\0&1&t\\0&0&\al^{-1}}:\al\in\bR^\times, t\in\bR}.
%    \end{align}
%
For example, consider the representation of $H$ on $\bR^3$.
As in \S\ref{ssec:pboundary}, let $p_0\coloneqq\spa_\bR(\set{e_1,e_2})\in\Xbar$.
In both~\ref{case1} and~\ref{case2} $p_0$ is $P$-invariant and therefore one obtains
an equivariant map $H/P\to \Xbar$ given by $$\eta \mapsto \eta p_0 \eqqcolon p_\eta.$$
Using this map in conjunction with $\xi$ as described above gives rise to the equivariant map
$B\to\Xbar$ given by $$b\mapsto p_{\xi(b)}\eqqcolon p_b.$$
Hence, the push-forward of $\be$ under $b\mapsto p_b$ is a $\mu$-stationary
probability measure on $\Xbar$. Since the Furstenberg measure $\bar{\nu}_{\Xbar}$ is the unique such measure, we
deduce that $(\bar{\nu}_{\Xbar})_b=\del_{p_b}$ for all $\be$-almost every $b\in B$.
This implies the following proposition which constitutes the first step towards classifying the $\mu$-stationary
measures on $X$.
\begin{proposition}\label{prop:firstprop}
    Let $\nu\in \spm \mu X$. Then, for $\be$-almost every $b\in B$ we have $\nu_b(\pi^{-1}(p_b))=1$.
\end{proposition}
\begin{proof}
    Since $\pi$ is $H$-equivariant we have that $\pi_*\nu\in \cP_\mu(\Xbar)$. Since the Furstenberg measure 
    $\bar{\nu}_{\Xbar}$ is the unique measure in $\cP_\mu(\Xbar)$ we deduce that
    $\pi_*\nu= \bar{\nu}_{\Xbar}$. Furthermore, it follows that 
    $\pi_*\nu_b=(\bar{\nu}_{\Xbar})_b$ for $\be$-almost every $b\in B$. As we observed above,
    $(\bar{\nu}_{\Xbar})_b=\del_{p_b}$ for $\be$-almost every $b\in B$ and the statement follows.
\end{proof}
Next we define some subgroups of $G$.
Let 
\begin{align*}
G_0&\defi \left\{\mat{*&*&*\\ *&*&*\\0&0&*}\in G\right\} =\Stab_G(p_0) \\
R_0 &\defi \left\{\mat{\al&0&*\\0&\al&*\\0&0&\al^{-2}}:\al\in \bR^\times\right\}
\end{align*}
and note that $R_0$ is the solvable radical of $G_0$. 
%and
%$$D_0\defi\set{
% \mat{R&0\\0&1}:R\in\SL_2(\bR)}$$ be the Levi factor of $G_0$ with trivial projection on $R_0$.\usnote{Not sure that
%this phrasing is correct. It should be a Levi factor and what do you mean by projection on $R_0$?}
%Let $U_0\defi\Stab_{D_0}(e_1)$\usnote{we should decide if this notation is acceptable because $\gou_0$ is taken.} 
%be the 1-dimensional unipotent subgroup of $D_0$ which fixes $e_1$ pointwise.
Since $P<G_0$ and $R_0$ is a normal subgroup of $G_0$ it is clear that $R_0$ is invariant under conjugation by $P$.
Moreover, 
$$L_0\defi \left\{\mat{\al&*&*\\0&\al&*\\0&0&\al^{-2}}:\al\in\bR^\times\right\}\supset R_0$$ 
is also easily seen to be invariant under conjugation by $P$.
The $P$-invariance of these subgroups allows us to define equivariant maps from $H/P$ to the set of subgroups of $G$.
In other words, the maps
$$hP\mapsto hR_0h^{-1},\ hP\mapsto hL_0h^{-1}\ \textrm{and}\ hP\mapsto hG_0h^{-1}$$ are well defined.
As before, for $\eta\in H/P$ we let $R_\eta$, $L_\eta$ and $G_\eta$ denote the images of these maps.
Combined with the map $\xi$, these maps allow us to define equivariant maps from $B$ to the set of subgroups of $G$.
These maps are given explicitly as
$$b\mapsto R_b\defi R_{\xi(b)},\ b\mapsto L_b\defi L_{\xi(b)}\ \textrm{and}\ b\mapsto G_b\defi G_{\xi(b)}.$$
Later on we will use corresponding lower case Gothic letters to denote the Lie algebras.
As $R_0$ is normal in $L_0$ we may define the 1-parameter unipotent quotient group 
$U_0\defi L_0/R_0$. Similarly for any $\eta\in H/P$
we define the 1-parameter unipotent quotient group
$$U_\eta\defi L_\eta/R_\eta.$$
This assignment is clearly equivariant and again for $b\in B$ we will use the notation
\begin{equation}\label{eq:Ub}
    U_b\defi U_{\xi(b)}.
\end{equation}
It is straightforward to check that $G_b=\Stab_G(p_b)$.
Moreover, $R_b$ acts trivially on $p_b$. It follows that the action of $U_b$ on $p_b$ is well defined and nontrivial for $\be$-almost every $b\in B$.
The crucial point for us is that this action descends to a nontrivial action of $U_b$ on $\pi^{-1}(p_b)$.
The main step in the proof of Theorem~\ref{thm:full}\eqref{main1}
is to show that if $\nu\in\spm \mu X$ is $\mu$-ergodic
and the limit measures $\nu_b$ are non-atomic almost surely, then for $\be$-almost every $b$, $\nu_b$ is $U_b$-invariant.
In the following subsection we show that this unipotent invariance implies that $\nu$ is the natural lift of $\bar{\nu}_{\Xbar}$.
%<---
%%
%---> first reduction
\subsection{Reduction of the proof of Theorem~\ref{thm:full}\eqref{main1}}\label{ssec:reduction1}
The core of the proof of Theorem~\ref{thm:full}\eqref{main1} is an application of the exponential drift argument
of Benoist and Quint. This is an elaborate argument which takes quite a lot of apparatus. In this section we isolate the following lemma
whose statement
does not require any preparation
and relying on this lemma we prove a proposition which reduces the proof of Theorem~\ref{thm:full}\eqref{main1}
to establishing the $\be$-almost sure $U_b$-invariance of the limit measures $\nu_b$.
\begin{lemma}\label{lem:firstgrowth}
    Let $\mu\in \cP(G)$ be compactly supported and suppose~\ref{case1} or~\ref{case2} holds. Then for any $\del>0$ and $R>0$ there exists $n_0>0$ such that for any
    $v\in \bR^3\smallsetminus\set 0$, $w \in \wedge^2 \bR^3\smallsetminus\set 0$ and $n>n_0$ one has
    $$\be \pa[\bigg]{\set[\bigg]{b\in B:
            \frac{\norm{b_1^n v}^{\phantom{1/2}}}{\norm{b_1^n w}^{1/2}} < \frac{\norm{v}^{\phantom{1/2}}}{\norm{w}^{1/2}}R
    }}<\del.$$
\end{lemma}
%We hope to achieve by this two things: First,
The statement of Lemma~\ref{lem:firstgrowth} and its use in the proof of Proposition~\ref{prop:unipinv} illustrates
in a simple fashion the role played by comparison of growth rates of vectors under random products in various representations,
which is a recurring theme in the paper.
%Proposition~\ref{prop:unipinv} motivates the line of attack towards establishing
%Theorem~\ref{thm:full}\eqref{main1} and shifts the attention to establishing the $U_b$-invariance of the $\nu_b$'s.
We note that the proof of Lemma~\ref{lem:firstgrowth} will only be given in~\S\ref{ssec:Iwasawa and lyapunov} after the necessary notation and tools
regarding Lyapunov exponents will be presented.
During the proof of Proposition~\ref{prop:unipinv} we will need to use the following construction.
Let  
\begin{equation}\label{eq:bds}
    B^X \defi  B\times X, \; \be^X \defi \int_B \del_b\otimes \nu_b\dv\be b
\end{equation} and $T:B^X\to B^X$ be defined by $T(b,x)\defi (Sb,b_1^{-1}x)$.
If $\nu\in \spm \mu X$ then $T$ preserves $\hbe$ and if $\nu$ is assumed to be $\mu$-ergodic then $T$ is ergodic.
Following Benoist and Quint, we will call the system $(B^X, T, \hbe)$ the \textit{backwards dynamical system}, see~\cite[\S2.5]{MR3560700}.
\begin{proposition}\label{prop:unipinv}
    Let $\nu\in\spm\mu X$ be $\mu$-ergodic.
    Suppose that for $\be$-almost every $b\in B$ the limit measure $\nu_b$ is $U_b$-invariant, where $U_b$ is as in~\eqref{eq:Ub}.
    Then $\nu$ is the natural lift of $\bar{\nu}_{\Xbar}$.
\end{proposition}
\begin{proof}
    %[Proof of Proposition~\ref{prop:unipinv} given Lemma~\ref{lem:firstgrowth}]
    Below we will show that the almost sure $U_b$-invariance of the $\nu_b$'s together with
    Lemma~\ref{lem:firstgrowth} imply
    that $\nu_b = m_{p_b}$ for $\be$-almost every $b\in B$.
    This will finish the proof because
    $$\nu=\int_B\nu_b\dv\be b =\int_B m_{p_b}\dv \be b =\int_{\Xbar} m_p\dv{\bar{\nu}_{\Xbar}}p,$$
    where the last equality follows because
    the Furstenberg measure $\bar{\nu}_{\Xbar}$ is the pushforward of $\be$ under $b\mapsto p_b$.
    % the formula $\nu=\int_B\nu_bd\be (b)$
    % translates into $\nu = \int_{\Xbar} m_p d\bar{\nu}_{\Xbar}(p)$, which means exactly that $\nu$ is the natural lift of $\bar{\nu}_{\Xbar}$ and
    % Theorem~\ref{thm:full}\eqref{main1} follows.

    Assume that $\nu_b$ is $U_b$-invariant $\be$-almost surely. By Proposition~\ref{prop:firstprop},
    $\nu_b$ is supported on $\pi^{-1}(p_b)$ for $\be$-almost every $b\in B$. The classification of $U_b$-invariant measures on
    $\pi^{-1}(p_b)$ due to Dani~\cite[Theorem A]{MR0578655} (see also~\cite{MR1135878}) implies that $\nu_b = t_bm_b+(1-t_b)\eta_b$ for some $0\le t_b\le1$, where $\eta_b$ is
    a $U_b$-invariant measure supported on
    the collection of periodic $U_b$-orbits in $\pi^{-1}(p_b)$.
    The equivariance of the $\nu_b$'s and the $m_b$'s imply the equivariance of the $\eta_b$'s
    which in turn implies that $t_b =t_{Sb}$ for $\be$-almost every $b\in B$.
    The ergodicity of the shift map implies that $t_b = t$ is $\be$-almost surely constant and then the ergodicity of $\nu$ implies that either $t=0$ or $t=1$.
    We assume that $t=0$; that is that $\be$-almost surely $\nu_b$ is supported purely on periodic $U_b$-orbits and derive a contradiction. This assumption may be restated in the backwards dynamical system as follows: Let
    $$\Sig \defi\set{(b,\br{\Lam})\in B^X:\br{\Lam}\in \pi^{-1}(p_b), \textrm{ and }U_b\br{\Lam} \textrm{ is periodic}}.$$
    Our assumption that $t=0$ implies that $\be^{\X}(\Sig) = 1$.

    For a 2-lattice $\Lambda$ we let $\av{\Lam}$ denote the covolume of $\Lam$.
    Recall that a $2$-lattice $\br{\Lam}\in\pi^{-1}(p_b)$ has a periodic $U_b$-orbit if and only if $\Lam$ intersects the eigenline
    $\ell_b$ of $U_b$ in the plane $p_b$ non-trivially.
    Therefore, the function
    $$\rho:\Sig\to (0,\infty),\ \textrm{given by}\ \rho(b,\br{\Lam})\defi\frac{\av{\Lam\cap \ell_b}}{\av{\Lam}^{1/2}}$$
    is well defined $\be^{\X}$-almost surely.
    Choose $R>0$ so that the pre-image $\Sig_{R}\defi \rho^{-1}((0,R))$ satisfies
    $\be^{\X}(\Sig_{R})>1/2.$

    Note that, if $(b,\br{\Lam})\in \Sig$, then choosing a primitive vector $v\in \Lam\cap \ell_b$
    and a basis $u_1,u_2$ of $\Lam$, if we set $w = u_1\wedge u_2 \in\wedge^2\bR^3$ then
    %basis $\set{u,w}$ of $\Lam$
    we have that $$\rho(b,\br{\Lam}) =  \frac{\norm{v}^{\phantom{1/2}}}{\norm{w}^{1/2}}.$$
    %Note also that for $\mu^{\otimes n}$-almost every $g_1^n$ and $\be$ almost any $b$ we have that
    It follows from the equivariance that for $\mu^{\otimes n}$-almost every $a\in G^n$ and $\be^X$-almost every 
    $(b,\br\Lam)\in\Sig$ we have $\ell_{ab}=a_1^n\ell_b$ and so $a_1^nv$ is a primitive vector in
    $\ell_{ab}\cap a_1^n\Lam$ and $a_1^nu_i, i=1,2$ is a basis for $a_1^n\Lam$. We conclude that
    \begin{equation}\label{eq:2125}
        \rho(ab,a_1^n\br{\Lam}) = \frac{\norm{a_1^nv}^{\phantom{1/2}}}{\norm{a_1^n w}^{1/2}}.
    \end{equation}
    %Consider the operator $Af(b,x)\defi \int_Gf(gb,gx)d\mu(g)$ defined say on non-negative measurable
    %functions on $B^X$.
    Consider the operator $\on A$ defined by $$\on Af(b,\br\Lam)\defi\int_G f(gb,g\br\Lam)\dv\mu g\quad\textrm{for}\quad f:B^X\to\left[0,\infty\right)\ \textrm{measurable}.$$
    We take $f\defi \mb{1}_{\Sig_R}$. Lemma~\ref{lem:firstgrowth} and~\eqref{eq:2125} imply that for $\be^X$-almost every $(b,\br\Lam)\in\Sig$ we have
    \begin{equation}\label{eq:2126}
        \lim_{n\to\infty}\on A^nf(b,\br\Lam)=\lim_{n\to\infty}\int_{G^n} \mb{1}_{\Sig_R}(ab,a_1^n\br\Lam) \dv{\mu^{\otimes n}} a = 0.
    \end{equation}
    It is easy to check that the operator $\on A$ preserves the measure $\be^X$. Hence for any $n\in\bN$ we have
    $$1/2\le \int_{B\times \X}  \mb{1}_{\Sig_R}\mathrm{d}{\be^{\X}}  = \int_{B\times \X} \on A^n\mb{1}_{\Sig_R} \mathrm{d}\be^{\X}.$$
    But on the other hand using~\eqref{eq:2126} and Lebesgue's dominated convergence theorem we see that the right hand side of the above equation tends to 0, which is a contradiction as required.
\end{proof}
%<---
%%
%--->Iwasawa
\subsection{The Iwasawa cocycle}\label{ssec:Iwasawa and lyapunov}
Let $H$ be as in~\ref{case1} or~\ref{case2} and let $H=KP$ be an Iwasawa decomposition of $H$ where $P$ is as in~\S\ref{ssec:boundary maps} and $K$ is
the maximal compact subgroup of $H$ corresponding to the inner product coming from the standard basis $\set{e_1,e_2,e_3}$.
%In the
%decomposition $K$ is a maximal compact subgroup of $H$ and $P$
%is a minimal parabolic subgroup as in~\eqref{eq:pc1} or~\eqref{eq:pc2} according to the case.
Let $\goz$ be the maximal abelian
subspace of the Lie algebra of $P$ and define $Z=\set{\exp\left(z\right):z\in\mathfrak{z}}$ to be the corresponding
Cartan subgroup of $H$. We denote by $\log:Z\to \goz$ the inverse of $\exp$.
Moreover, we set $N$ to be the unipotent radical of $P$ so that
$P=ZN$. See~\cite{MR1920389} for details.

Let $s:H/P\to H/N$
be a measurable section with image in $KN$.
For $h\in H$ and $\eta\in H/P$ let $\alpha:H\times H/P\to Z$
be defined so that $\alpha\left(h,\eta\right)$ is the unique element
of $Z$ such that
\begin{equation}
    hs\left(\eta\right)=s\left(h\eta\right)\alpha\left(h,\eta\right).\label{eq:section eqn}
\end{equation}
Note that since $Z$ normalises $N$ it acts on $H/N$ from the
right and moreover this action is transitive with trivial stabilisers, so equation~\eqref{eq:section eqn} makes sense
and defines $\al(h,\eta)$ uniquely. The \emph{Iwasawa cocycle}
is the map
$$\sig(h,\eta) \defi \log \al(h,\eta).$$ Indeed, it is not hard to see from~\eqref{eq:section eqn} that the
cocycle relations $\al(gh,\eta) = \al(g,h\eta)\al(h,\eta)$, $\sig(gh,\eta) = \sig(g,h\eta)+\sig(h,\eta)$ hold. See~\cite[\S8.2]{MR3560700}
for details.
%Since in the analysis exponential maps appear regularly it will be convenient to maintain
%two sets of notations such as $\al$ and $\sig$ introduced above.

Let $\on E:B\to Z$ be given by
$$\on{E}(b)\defi\alpha(b_1,\xi(Sb))$$
and for $n\in\bN$ let $\on E_n(b):B\to Z$ be
$$\on E_n(b)\defi\prod_{i=1}^n\on E(S^{i-1}b)=\alpha(b_1^n,\xi(S^nb)).$$
Additionally, we define the corresponding logarithmic versions $\on L:B\to\goz$ and $\on L_n:B\to\goz$ as
$$\on L(b)\defi\log\on E(b)\quad\textrm{and}\quad\on L_n(b)\defi\log\on E_n(b).$$
%\begin{align}
%    \on{E}\left(b\right)&=\alpha\left(b_{1},\xi\left(Sb\right)\right).\label{eq:var theta}\\
%    \nonumber \on{L}(b) &= \sig(b_1,\xi(Sb))
%\end{align}
%Additionally, for $n\in\mathbb{N}$, let
%\begin{align*}
%    \on{E}_{n}\left(b\right)&=\prod_{i=1}^{n}\on{E}\left(S^{i-1}b\right)=\alpha\left(b_1^n,\xi\left(S^{n}b\right)\right),\\
%    \on{L}_{n}(b)& = \sum_{i=1}^n\on{L}(S^{i-1}b) = \sig(b_1^n,\xi(S^nb)).
%\end{align*}
Let $V$ be a finite dimensional representation of $H$ and let $\mathcal{W}_{\mathfrak{z}}\left(V\right)$
be the set of weights of $V$ relative to $\mathfrak{z}$.
For $\omega\in\mathcal{W}_{\mathfrak{z}}\left(V\right)$ we will use
$V^{\omega}$ to denote the corresponding weight space.
Let $\mathcal{H}_{\mathfrak{z}}\left(V\right)$
be the set of highest weights of the representation $V$.
For $\omega\in\mathcal{H}_{\mathfrak{z}}\left(V\right)$
we write $V\left[\omega\right]$ for the corresponding isotypic component.
Note that $(V[\om])^\om$ is $P$-invariant and pointwise fixed by $N$.

%--->comment
%\usnote{commented def of Lyap exp}
\begin{comment}
    \begin{definition}\label{def:lyap}
        For $\omega\in\mathcal{H}_{\mathfrak{z}}\left(V\right)$ and $1\leq k\leq\dim\left(V\left[\omega\right]\right)-1$
        the $k$'th \emph{Lyapunov exponent} of $\mu$ on $V\left[\omega\right]$,
        which we denote by $\lambda_{k,\omega}\left(\mu\right)$, is defined
        recursively via the formula
        \[
            \sum_{i=1}^{k}\lambda_{i,\omega}\left(\mu\right)\defi\lim_{n\to\infty}\frac{1}{n}\int_{G}\log\left\Vert g\right\Vert _{\wedge^{k}V\left[\omega\right]}\mathrm{d}\mu^{*n}\left(g\right),
        \]
        where $\left\Vert g\right\Vert _{\wedge^{k}V\left[\omega\right]}$
        denotes the norm of $g$ when it is considered as a linear operator
        on $\wedge^{k}V\left[\omega\right]$.
    \end{definition}
    We note that the existence of
    the limit is guaranteed by Kingman's Sub-additive Ergodic Theorem
    ~\cite[Theorem 1]{MR0254907} and the values of $\lambda_{k,\omega}\left(\mu\right)$
    do not depend on the choice of norm. \usadd{Add something regarding the meaning of Lyapunov exponents and
    relate it to the growth rate of individual vectors which is discussed below}.
\end{comment}
%<---

By the discussion in~\S\ref{ssec:boundary maps}, the $P$-invariant subspace $(V[\om])^\om$ gives rise to an equivariant
map from $H/P$ to the set of subspaces of $V$ given by $$gP\mapsto g(V[\om])^\om.$$ To reduce the notational clutter, for $\eta=gP\in H/P$ we
simplify the notation to
$$    V_{\eta}\left[\omega\right]\defi g\left(V\left[\omega\right]\right)^{\omega}.$$
For $b\in B$ we will also use the notation
$$    V_{b}\left[\omega\right]\defi V_{\xi\left(b\right)}\left[\omega\right].$$
For $\omega\in\mathcal{W}_{\mathfrak{z}}\left(V\right)$, let $\chi^{\omega}:Z\to\mathbb{R}^\times$ be the
character corresponding to the weight $\om$, that is $$\chi^\om \defi \exp\comp \om \comp \log.$$
%Define
%\[
%\theta_{\omega}\left(b\right)\defi\log\left|\chi^{\omega}\left(\alpha\left(b_{1},\xi\left(Sb\right)\right)\right)\right|.
%\]
We will use the following lemma on multiple occasions.
It is the same as~\cite[Lemma 5.4]{BQAnnals} except that we replaced ``irreducible representation" by ``isotypic component".
%Its proof follows straight from the definitions and is left to the reader \usadd{(refer to
%one of the BQ papers for a similar statement)}.
\begin{lemma}
    \label{lem:expansion of leafs-1}
    Suppose $V$ is a representation
    of $G$. Then,
    for $\be$-almost every $b\in B$ and for all $n\in \bN$,
    %$a\in G^n$ such that $aS^nb\in B_0$,
    $\omega\in\mathcal{H}_{\mathfrak{z}}\left(V\right)$ and $v\in V_{S^nb}\left[\omega\right]$
    we have
    \begin{equation*}
        %\label{eq:16521}
        \chi^\om(\on E_n(b))=\norm{b_1^n v}/\norm{v}.
    \end{equation*}
    %\begin{align}
    %    \label{eq:1652}&\chi^\om(\mbe_n(b))  = \norm{b_1^nv}/\norm{v},\\
    %    \label{eq:16521}&\chi^\om(\mbe_n(a_1^nS^nb))  = \norm{a_1^nv}/\norm{v},
    %\end{align}
    %or equivalently
    %\begin{align}
    %    \label{eq:1653}&\om(\on{L}_{n}(b)) = \log(\norm{b_1^nv}/\norm{v}),\\
    %    \label{eq:16531}&\om(\on{L}_{n}(a_1^nS^nb) = \log(\norm{a_1^nv}/\norm{v}).
    %\end{align}
\end{lemma}
%%%%

%<---
%%
%--->Lyapunov
\subsection{The Lyapunov vector and pairs of highest weights}
Our choice of $P$ in~\S\ref{ssec:boundary maps} implies that $\goz$ consists of diagonal traceless matrices.
Let $\goz^+$ be the Weyl chamber associated with $P$ and $\goz^{++}$ denote its interior.
\begin{enumerate}[$\bullet$]
    \item In~\ref{case1} we have $\goz^{++} = \set{\diag{t_1,t_2,t_3}\in\goz: t_1>t_2>t_3}$.
    \item In~\ref{case2} we have $\goz^{++} = \set{\diag{t,0,-t}\in\goz: t>0}$.
        %(recall that in this case we take
        %$H=\SO(Q)$, where $Q(x,y,z) = 2xz-y^2$).
\end{enumerate}
To keep a unified treatment we will denote elements in $\goz^{++}$ by $\diag{t_1,t_2,t_3}$ and use
the inequalities $t_1>t_2>t_3$ which are valid in both cases.
The following theorem is a collection of relevant statements regarding the Lyapunov vector of $\mu$.
In~\cite{MR3560700} this collection of statements is referred to as ``the law of large numbers on $H$".
\begin{theorem}\label{thm:poslyap}
    Let
    $$\sig_\mu \defi \int_{H\times H/P}\sig\mathrm{d}\mu \mathrm{d}\bar{\nu}_{H/P},
    %%\int_B\mbl(b) d\be(b).
    $$ where $\bar{\nu}_{H/P}\in\spm\mu {H/P}$ is the unique $\mu$-stationary probability measure on $H/P$.
    Then:
    \begin{enumerate}[(1)]
        \item~\cite[Theorem 10.9(a)]{MR3560700} The Iwasawa cocycle $\sigma$ is integrable and hence $\sig_\mu\in\goz$ is well defined.
        \item\label{asconvergence}~\cite[Theorem 10.9(a)]{MR3560700}
            For $\be$-almost every $b\in B$ we have $\lim_{n\to\infty}\frac{1}{n}\on{L}_{n}(b) = \sig_\mu.$
            %\item Uniformly for $\eta\in H/P$, $\int_G\norm{\frac{1}{n}\sig(g,\eta) - \sig_\mu} d\mu^{*n}(g)\to 0$.
        \item\label{cor:poslyap}~\cite[Theorem 4.28(b) + Corollary 10.12]{MR3560700}
            If $V$ is an irreducible representation of $H$ with highest weight $\om$, then for
            all $v\in V\smallsetminus \set 0$ and for $\be$-almost every $b\in B$ one has
            $$\lim_{n\to\infty}\frac{1}{n}\log\frac{\norm{b_n^1v}}{\norm{v}}=\om(\sig_\mu).$$
            This sequence also converges in $L^1(B,\beta)$ uniformly over $v\in V\smallsetminus \set 0$.
        \item\label{simplelyap}~\cite[Theorem 10.9(f)]{MR3560700} One has $\sig_\mu\in \goz^{++}$. In particular, if $\om$ is a positive weight, then $\om(\sig_\mu)>0$.
    \end{enumerate}
\end{theorem}
%%%%%%%%%%%%%%%%%%%%%%%%%%%%%%%%%%
The vector $\sig_\mu$ defined in Theorem~\ref{thm:poslyap} is called the \textit{Lyapunov vector} of $\mu$.

There are two pairs of highest weights which play a prominent role in our discussion.
The first pair consists of the highest weights of the irreducible representations of $H$ on
$\bR^3$ and on $\wedge^2\bR^3$ which we denote by $\omrt$ and $\omrd$ respectively.
The important fact regarding this pair is that for $t=\diag{t_1,t_2,t_3}\in\goz$ we have
$\omrt(t) = t_1$ and $\omrd(t) = t_1+t_2$ so that $(\omrt-\frac{1}{2}\omrd)(t) = \frac{1}{2}(t_1-t_2)$ and so
$\omrt-\frac{1}{2}\omrd$ is positive.
Thus, by part~\eqref{simplelyap} of Theorem~\ref{thm:poslyap} we have the following
fundamental inequality: In both of~\ref{case1} and~\ref{case2} one has
\begin{equation}\label{eq:1128}
    \omrt(\sig_\mu)-\frac{1}{2}\omrd(\sig_\mu)>0.
\end{equation}
Equipped with this inequality and with Theorem~\ref{thm:poslyap} we can now easily deduce Lemma~\ref{lem:firstgrowth} which
played an important role in the proof of Proposition~\ref{prop:unipinv}.

\begin{proof}[Proof of Lemma~\ref{lem:firstgrowth}]
    Given $\eps>0$, it follows from part~\eqref{cor:poslyap} of Theorem~\ref{thm:poslyap} that there exists
    $n_0>0$ such that for all $v\in \bR^3\smallsetminus\set 0$, $w\in\wedge^2\bR^3\smallsetminus\set 0$ and $n>n_0$ one has

    $$\mu^{*n}\pa*{\left\{g\in G :
            \begin{array}{ll}
                \norm{gv} > \norm{v}\exp(n(\omrt(\sig_\mu)-\eps))\\
                \norm{gw} < \norm{w}\exp(n(\omrd(\sig_\mu)+\eps))
            \end{array}
    \right\}}>1-\eps.$$
    By~\eqref{eq:1128}, on choosing $\eps$ so that
    $\omrt(\sig_\mu) - \frac{1}{2}\omrd(\sig_\mu) - \frac{3}{2}\eps = \eps'>0$
    we get that for all for all $v\in\bR^3\smallsetminus\set 0$, $w\in\wedge^2\bR^3\smallsetminus\set 0$ and $n>n_0$,
    $$\mu^{*n}\pa[\bigg]{\set[\bigg]{g\in G :
    \frac{\norm{gv}^{\phantom{1/2}}}{\norm{gw}^{1/2}}>\frac{\norm{v}^{\phantom{1/2}}}{\norm{w}^{1/2}}\exp(n\eps')}}
    >1-\eps.$$
    The statement of the lemma now readily follows.
\end{proof}
%---> comment

%%%%%%%%%%%%%%%%%%%%%%%%%%
\begin{comment}
    The aim of this section is to introduce various objects which in the
    end we will relate to divergence properties of a random walk on our
    space $X$. In the drift argument (\S\ref{sec:The-dynamics}) we
    will start with two points that are very close together, we will consider
    what happens to the displacement between these two points when we
    start walking at random. It will be very important that both the expansion
    of this displacement and its direction evolve in a controlled manner.
\end{comment}
%<---
We now discuss the second pair of highest weights that will concern us.
%Observe first that if $V$ is a representation of $H$ and if $\bR v\in \bP V$ is fixed by $P$, then the weight
%$\om$ by which $\goz$ acts on $\bR v$ belongs to $\cH_{\goz}(V)$.
Let $\gor_0$ and $\gol_0$ be the Lie algebras of the Lie groups $R_0$ and $L_0$ we defined in \S\ref{ssec:boundary maps}.
The Lie algebras $\gor_0$ and $\gol_0$ correspond to $P$-invariant lines in the representations $\wedge^3\gog$ and $\wedge^4\gog$ respectively.
We denote the corresponding weights by $\om_{\gor_0}\in\cH_\goz(\wedge^3\gog)$ and $\om_{\gol_0}\in\cH_\goz(\wedge^4\gog)$.
Given $t=\diag{t_1,t_2,t_3}\in \goz$ we have $\oml(t) = 2(t_1-t_3)$ and $\omr(t) = t_1+t_2-2t_3$ so that
$(\oml-\omr)(t) = t_1-t_2$ and so $\oml-\omr$ is positive.
By part~\eqref{simplelyap} of Theorem~\ref{thm:poslyap} we arrive at the following fundamental inequality: In both~\ref{case1} and~\ref{case2}
\begin{equation}\label{eq:posweight}
    \oml(\sig_\mu)-\omr(\sig_\mu)>0.
\end{equation}
We will often work with
the difference $\om_{\gol_0} - \om_{\gor_0}$ and use the notation
\begin{equation}\label{eq:0632}
    \om_{\gol_0/\gor_0}\defi\om_{\gol_0}-\om_{\gor_0}\quad\textrm{and}\quad\chi_{\gol_0/\gor_0}\defi\exp\comp\om_{\gol_0/\gor_0}\comp\log.
\end{equation}
% \begin{align}
%     \nonumber &\om_{\gol_0/\gor_0} \defi \om_{\gol_0} - \om_{\gor_0},\\
%               &\chi_{\gor_0} \defi \chi^{\om_{\gor_0}},\;\; \chi_{\gol_0} = \chi^{\om_{\gol_0}}, \label{eq:0632} \\
%     \nonumber &\chi_{\gol_0/\gor_0}  \defi \chi_{\gol_0}/\chi_{\gor_0}.
% \end{align}
%<---
%%
%--->Lemmas????
\subsection{Two lemmas about representations}
In this subsection we collect some representation theoretic results which are specific to the representations we are interested in.
Let $V$ be a representation of $H$ and let $\om\in\cH_\goz(V)$. We denote by $\tau_\om:V\to V[\om]$ the natural projection and note that it is $H$-equivariant.
We always assume that the norm we choose on a vector space $V$ is induced by an inner product with respect to which the isotypic components are orthogonal and such that the maximal compact $K<H$ from the Iwasawa decomposition acts by isometries.
For $\al>0$ we denote
\begin{equation}\label{eq:2300}
    V_{\sphericalangle \al}[\om] \defi \set{v\in V: \norm{\tau_\om(v)}\ge \al\norm{v}}.
\end{equation}
This is the complement of a projective neighbourhood of $\ker \tau_\om$.

%We will need the following lemma which classifies 4-dimensional $P$-invariant subspaces of $\gog$ in both~\ref{case1} and~\ref{case2}.
%\begin{lemma}
%    In~\ref{case1} the set of 4-dimensional $P$-invariant subspaces of $\gog$ is
%    $$\set{\spa_\bR(\set{a_1d_2+a_2d_2,e_{12},e_{13},e_{23}}):(a_1,a_2)\in\bR^2}.$$
%    In~\ref{case2} the set of 4-dimensional $P$-invariant subspaces of $\gog$ is
%    $$\set$$.

%\end{lemma}
For $p\in\Xbar$ we define $G_p\defi\Stab_G(p)$ and $R_p$ to be the radical of $G_p$.
Since $G_{p_b}=G_b$ for $\be$-almost every $b\in B$ these definitions are compatible with our previous definitions 
from~\S\ref{ssec:boundary maps}.
As usual, we denote the corresponding Lie algebras by lower-case Gothic letters.
Moreover in~\ref{case2}, also recall the notation $\cC$ for the circle of isotropic planes in $\Xbar$ that we introduced in \S\ref{ssec:pboundary}.
\begin{lemma}\label{lem:max weight}
    The following hold:
    \begin{enumerate}[(1)]
        \item\label{p:0610} The weight $\om_{\gol_0}$
            is a maximal weight in $\cH_{\goz}(\wedge^4\gog)$.
            %in the sense that for all $\oml\ne \om\in \cH_{\goz}(V)$, $\oml-\om$ is
            %a positive weight.
        \item\label{p:06111} In~\ref{case1} there exists $\al>0$ such that for all $p\in \Xbar$, $u\in\wedge^3\gor_p$ and $v\in\gog$,
            $$v\wedge u \in\Vang{(\wedge^4\gog)}.$$
        \item\label{p:0611} In~\ref{case2} there exists $\al>0$ such that for all $p\in \cC$, $u\in\wedge^3\gor_p$ and $v\in\gog_p$,
            $$v\wedge u\in \Vang{(\wedge^4\gog)}.$$
        \item\label{p:06101} For all $\eta\in H/P$ one has 
            \begin{equation}\label{eq:1731}
                \set{v\wedge u:v\in\gog,u\in\wedge^3\gor_\eta}\cap(\wedge^4\gog)_\eta\br{\oml}=\wedge^4\gol_\eta.
            \end{equation}
    \end{enumerate}
\end{lemma}
\begin{proof}
    First we prove~\eqref{p:0610}.
    Given $t=\diag{t_1,t_2,t_3}\in \goz$ the eigenvalues of $\on{ad}_t$ on $\gog$ determine its eigenvalues on $\wedge^4\gog$.
    Namely, they are all possible sums of 4 eigenvalues of $\on{ad}_t$ on $\gog$ corresponding to different eigenlines.
    It is then clear from the ordering of the weights of the adjoint representation that the maximal weight of the fourth
    wedge is $\om_{\gol_0}(t) = 2(t_1 - t_3)$.

    Next we prove~\eqref{p:06111}.
    In~\ref{case1} we have $K\cong\SO_3(\bR)$ and hence it acts transitively on $\Xbar$.
    Write $p = kp_0$ for some $k\in K$ and then $\gor_p = k\gor_0$.
    Since the set $\Vang{(\wedge^4\gog)}$ is $K$-invariant, it is enough to prove~\eqref{p:06111} for $p=p_0$.
    Let $\set{e_{ij}}_{1\leq i,j\leq 3}$ be the basis
    of unit matrices in $\Mat_3(\bR)$ and let $d_1\defi e_{11}+e_{22}-2e_{33}$ and $d_2\defi e_{11}-e_{22}$ so that
    $\set{d_1,d_2}\cup\set{e_{ij}}_{1\leq i,j\leq 3,i\neq j}$ forms a basis of $\gog$.
    Since $\wedge^3\gor_0$ is one dimensional the collection of pure wedges $\gow_0\defi\set{v\wedge u:v\in\gog,u\in\wedge^3\gor_0}$ forms a subspace
    of $\wedge^4\gog$. It follows that $\gow_0\subset\Vang{(\wedge^4\gog)}$ for some positive $\al$ provided that $\gow_0\cap\ker \tau_{\om_{\gol_0}}=0$.
    To this end,
    let $u_0 =d_1 \wedge e_{23}\wedge e_{13}\in\wedge^3\gor_0$
    and assume by way of contradiction
    that there exists $v_0 \in \gog\smallsetminus\set 0$ such that $v_0\wedge u_0\in\ker\tau_{\oml}$.
    We may assume without loss of generality that $v_0$ is orthogonal to $\gor_0$ or in other words
    $$v_0\defi v_{21}e_{21}+v_{31}e_{31}+v_{32}e_{32}+v_{22}d_2+v_{12}e_{12}.$$
    Since the condition $v_0\wedge u_0\in\ker\tau_{\oml}$ is $H$-invariant and
    $u_0$ is $P$-invariant we get that
\begin{equation}\label{posang2} {p}v_0\wedge u_0\in\ker\tau_{\oml}\ \textrm{for all}\ p\in P.\end{equation}
    For $v\in\gog$ it is easy to check that
\begin{equation}\label{posang3}v\wedge u_0\in\ker\tau_{\oml}\ \textrm{if and only if}\ \idist{v,e_{12}}=0.\end{equation}
    For $t\in\bR^2$ let $p(t)\coloneqq\exp(t_1e_{12}+t_2e_{13})\in P$.
    It is easy (but tedious) to compute
    $$\idist{ {p(t)}v_0,e_{12}}=v_{12}-t_1v_{22}-t_1^2v_{21}+t_2v_{32}-t_1t_2v_{31}.$$
    Combining this computation with~\eqref{posang2} and~\eqref{posang3} gives that $v_0=0$ which is a contradiction as required.
    %
    %        Hence we must have $v_{12}=v_{21}=d_2=0$. Moreover, let
    %    $\tau_{\om_{\gol_0}}(\Ad_{p(t)}}(v)\wedge u_0 )= 0$ if and only if
    %        $\Ad_{p}(v)\in\gor_0$ for all $p\in P$. HoweverUsing the fact that $\Ad\comp\exp v=\exp\comp\on{ad}_v$ for all $v\in\gog$ we have that
    %        the last equation holds if and only if
    %        $\exp\comp\on{ad}_{p(t)}v\in\wedge^3\gor_0$.
    %
    %        $\set{n_{x,y}} =\set{ \smallmat{1&x&y \\ 0&1&0 \\ 0&0&1}},$ it is enough to show
    %        that there are $x,y$ such that $\tau_{\om_{\gol_0}}(\Ad_{n_{x,y}}(v)\wedge u_0 )\ne 0$.
    %        Letting $z_1 = \diag{1,-1,0}$, we have that $\set{z_0,z_1, e_{ij}:1\le i\ne j\le 3}$ forms a basis of $\gog$. We then obtain a basis of $V$ by taking 4-wedges of vectors from this basis. The basis element
    %        $e_{12}\wedge u_0$ is in the weight space $V^{\oml}$ and
    %        since by~\eqref{p:0610} $\oml$ is maximal among $\cH_\goz(V)$, any vector in $V$ having a non-zero $e_{12}\wedge u_0$-coefficient
    %        in its expansion with respect to this basis cannot be in $\ker \tau_{\oml}$. Thus,
    %        it is enough to find $x,y$ such that the $e_{12}\wedge u_0$-coefficient of $\Ad_{n_{x,y}}(v)\wedge u_0$
    %        is non-zero. A quick calculation shows that this coefficient is of the form
    %        $p(x,y)e_{12}\wedge u_0$, where
    %        $p(x,y) =-cx^2-(2a+ry)x + b + sy$. Clearly if $p(x,y) = 0$ for all $x,y$ then $v=0$ contradicting our assumption.
    %
    %
    %

    The proof of~\eqref{p:0611} is very similar to the proof of part~\eqref{p:06111}.
    Since we only consider $p\in\cC$ and in this case $K\cong\SO_2(\bR)$ acts transitively on $\cC$, we can reduce to the case when $p=p_0$.
    %We first make a reduction to showing that
    %$\gog_0\wedge \gor_0\subset \Valpha$ for some $\al>0$.
    As before this will follow provided that $\set{v\wedge u:v\in\gog_0,u\in\wedge^3\gor_0}$ intersects $\ker\tau_{\oml}$ trivially.
    Suppose there exists $v_0'\in\gog_0\smallsetminus\gor_0$ such that $v_0'\wedge u_0\in\ker\tau_{\oml}$.
    Without loss of generality we may suppose that $v_0'$ is orthogonal to $\gor_0$ and hence we may write
    $$v_0'\defi v_{12}e_{12}+v_{22}d_2+v_{21}e_{21}.$$
    For $t\in\bR$, let $p'(t)\defi\exp(t(e_{12}+e_{23}))\in P$.
    Another tedious computation reveals that
    $$\idist{ {p'(t)}v_0',e_{12}}=v_{12}-tv_{22}-t^2v_{21}.$$
    Again using~\eqref{posang2} and~\eqref{posang3} we see that this implies that $v_0'=0$ which is a contradiction.

    Finally we prove~\eqref{p:06101}.
    Since $K$ acts transitively on $H/P$ we see that it is enough to prove~\eqref{eq:1731} for $\eta_0\defi P$.
    Since $\oml$ is maximal in $\cH_\goz(\wedge^4\gog)$ by part~\eqref{p:0610} we deduce that 
    $(\wedge^4\gog)_{\eta_0}[\oml] = (\wedge^4\gog)^{\oml}$. Thus, fixing $u_0\in \wedge^3\gor_0\smallsetminus\set{0}$,
    \eqref{eq:1731} becomes 
    $$\set{v\wedge u_0:v\in\gog,\; \textrm{for all}\ t\in \goz,\, \on{ad}_t(v\wedge u_0) = \oml(t)(v\wedge u_0)} = \wedge^4\gol_0.$$
    The inclusion $\supseteq$ is clear as $\oml$ was defined to be the weight by which $\goz$ acts on $\wedge^4\gol_0$. 
    We now establish the inclusion $\subseteq$. Without loss of generality $v$ is orthogonal to $\gor_0$, which means that 
    $v$ is a linear combination of $\set{d_2,e_{ij}:(i,j)\notin\set{(1,3),(2,3)}}$. In turn, $v\wedge u_0$ is a linear combination
    of $\set{d_2\wedge u_0,e_{ij}\wedge u_0:(i,j)\notin\set{(1,3),(2,3)}}$. Since all the vectors in this set are eigenvalues 
    of $\on{ad}_\goz$ and only $e_{12}\wedge u_0$ has eigenvalue given by $\oml$ we deduce that if 
    $\on{ad}_t(v\wedge u_0) = \oml(t)(v\wedge u_0)$ for all $t\in \goz$, then $v\in \bR(e_{12}\wedge u_0)=\wedge^4\gol_0$
    which completes the proof. 
\end{proof}
\begin{remark}\label{rem:important}
    Lemma~\ref{lem:max weight} lies at the heart of the discussion.
    In Lemma~\ref{lem:max weight} the crucial difference between~\ref{case1} and~\ref{case2} manifests itself.
    Parts~\eqref{p:06111} and~\eqref{p:0611} say that certain vectors in $\wedge^4\gog$ have a component
    in $(\wedge^4\gog)[\oml]$ which is of `positive proportion'.
    This will allow us to use the positivity~\eqref{eq:posweight} and control to some extent the way two nearby
    points in $X$ drift away from each other.

    Part~\eqref{p:06101} is needed to ensure that our `limiting displacement' will be pointing in the right direction in
    the case that the multiplicity of $\oml$ is larger than 1.

    In~\S\ref{ssec:drift} we will know that the two nearby points lie in the same plane
    $p$. Hence, the displacement vector between them corresponds to a pure wedge of the form $v\wedge u$ for $v\in \gog_p$ and $u\in \wedge^3\gor_p$.
    This will allow us to use Lemma~\ref{lem:max weight} in both~\ref{case1} and~\ref{case2}.

    On the other hand, in \S\ref{sec:non atomicity} we will use the same positivity to prove Theorem~\ref{thm:full}\eqref{main2} which is the statement
    that the limit measures are non-atomic.
    There, we will also need to understand how two near-by points in $X$ drift apart but will need to do so for pairs of points which
    do not necessarily lie in the same plane. This means that the displacement vector between them is of the form $v\wedge u$ where
    $v\in \gog$ and $u\in\gor_p$ for some $p\in\Xbar$. Thus, we would be able to apply Lemma~\ref{lem:max weight} only for~\ref{case1}.

    One concludes that the small technical difference between parts~\eqref{p:06111} and~\eqref{p:0611} of Lemma~\ref{lem:max weight}
    is what stands behind the phenomenon appearing in Theorem~\ref{thm:cover}.
\end{remark}
\begin{remark}\label{rem:important2}
    By analysing the proof of Lemma~\ref{lem:max weight} one can see that in~\ref{case2} the subspace
    ${\set{v\wedge u:v\in\gog,u\in\wedge^3\gor_0}}$ does intersect $\ker \tau_{\oml}$ non-trivially.
    In fact, this intersection equals $\set{v\wedge u: v\in\spa_\bR({e_{21}+2e_{32}}),u\in\wedge^3\gor_0}$. This should be compared with
    the construction given in the proof of Theorem~\ref{thm:cover} since $\Lam_t=t(e_{21}+2e_{32})\spa_\bZ(\set{e_1,e_2})$.
\end{remark}
The following lemma will be used in~\S\ref{sec:Non-escape-of-mass/contracted} and~\S\ref{sec:non atomicity} where we will replace
$\mu$ by $\mu^{*n_0}$ for some $n_0\in\bN$ in order to know that the integrals on the left hand sides of equations~\eqref{eq:22061} and~\eqref{eq:22062}
are bounded away from zero uniformly.
\begin{lemma}\label{lem:replacement}
    There exist $\lam_0>0$ and $n_0>0$ such that for all $n\ge n_0$, the following hold:
    \begin{enumerate}[(1)]
        \item In both~\ref{case1} and~\ref{case2}, for all $v\in\bR^3\smallsetminus\set 0$ and $w\in \wedge^2\bR^3\smallsetminus\set 0$ one has
            \begin{align}
                \label{eq:22061}
                \int_G \log \pa[\bigg]{\frac{\norm{gv}}{\norm{v}}\Big/\frac{\norm{gw}^{1/2}}{\norm{w}^{1/2}}}\dv{\mu^{*n}}g  > n\lam_0.
            \end{align}
        \item In case~\ref{case1}, for all $p\in \Xbar$, $u\in \wedge^3\gor_p\smallsetminus\set 0$ and $v\in \gog\smallsetminus \gor_p$ one has
            \begin{align}
                \label{eq:22062}
                \int_G \log \pa[\bigg]{\frac{\norm{g(v\wedge u)}}{\norm{v\wedge u}}\Big/\frac{\norm{gu}}{\norm{u}}}\dv{\mu^{*n}}g > n\lam_0.
            \end{align}
    \end{enumerate}
\end{lemma}
\begin{proof}
    Let $$\lam_1 \defi\frac{1}{2} \min\set{\omrt(\sig_\mu) - \frac{1}{2}\omrd(\sig_\mu),\; \oml(\sig_\mu) - \omr(\sig_\mu)}$$
    which is positive by~\eqref{eq:1128} and~\eqref{eq:posweight}. The inequality~\eqref{eq:22061} with $\lam_0 = \lam_1$ and
    $n$ large enough (independent of the vectors) follows directly from
    the uniformity of the $L^1$-convergence in part~\eqref{cor:poslyap} of Theorem~\ref{thm:poslyap} applied to the irreducible
    representations of $H$ on $\bR^3$ and $\wedge^2\bR^3$.

    Next we prove~\eqref{eq:22062}. First we show that the line $\wedge^3\gor_p$ in $\wedge^3\gog$ is
    contained in the isotypic component $(\wedge^3\gog)[\omr]$. To see this, note that since
    $\gor_p = g\gor_0$ where $g\in H$ is such that $gp_0=p$, it is clear that it is enough to show that the line
    $\wedge^3 \gor_0$ is contained in the isotypic component $(\wedge^3\gog)[\omr]$. This holds since $\goz$ acts on
    $\wedge^3 \gor_0$ by the weight $\omr$ and this line is an eigenline of $P$.

    Next, note that by part~\eqref{p:06111} of Lemma~\ref{lem:max weight} there exists $\al>0$ such that
    for all $p\in\Xbar$,$u\in\wedge^3\gor_p\smallsetminus\set 0$ and $v\in\gog\smallsetminus \gor_p$ one has
    $v\wedge u\in \Vang{(\wedge^4\gog)}$. This implies that for any $g\in H$ one has
    $$\frac{\norm{g(v\wedge u)}}{\norm{v\wedge u}}\ge \al^{-1} \frac{\norm{g\tau_{\oml}(v\wedge u)}}{\norm{\tau_{\oml}(v\wedge u)}}.$$
    Together with another
    application of the uniform $L^1$-convergence in part~\eqref{cor:poslyap} of Theorem~\ref{thm:poslyap} (this time for the irreducible representations corresponding
    to the highest weights $\oml$ and $ \omr$) this shows that~\eqref{eq:22062} holds for all
    large enough $n$,
    with $n\lam_0$ on the right hand side replaced by $n\lam_1+\log \al$.
    Since $n\lam_1+\log \al>n \lam_1/2$ for all large enough $n$, the lemma is valid with $\lam_0 = \lam_1/2$.
\end{proof}

%<---
%%
%<---
%%
%---> The drift!!!
%%
%---> horocyclic flow
\section{The drift argument - Proof of Theorem~\ref{thm:full}\eqref{main1}}\label{sec:drift}
In this section we will prove Theorem~\ref{thm:full}\eqref{main1} by adapting the exponential drift argument of
Benoist and Quint from~\cite{BQJams}. Throughout $\nu\in\cP_\mu(X)$ is a $\mu$-ergodic stationary measure
and $(B^X,\be^X, T)$ denotes the backwards dynamical system as defined in~\eqref{eq:bds}.
\subsection{\label{ssec:horocyclic}The horocyclic flow}  
It will be convenient for us to work in an extension of the backwards dynamical system having an extra coordinate
which is used for book keeping purposes. 
Recall that $Z$ is the Cartan subgroup of $H$ defined in \S\ref{ssec:Iwasawa and lyapunov}. Let $\lambda$ be a Haar measure on $Z$ and let
\begin{equation}\label{eq:2233}
    B^{X,Z}\defi B^X\times Z\quad\textrm{and}\quad\be^{X,Z}\defi\int_{B\times Z}\del_b\otimes\nu_b\otimes \del_z\dv\be b\dv\lambda z.
\end{equation}
The extension of the backwards dynamical system that we consider is given by the 
map $\widehat{T}:B^{X,Z}\to B^{X,Z}$ which clearly preserves $\be^{X,Z}$ and is defined using the Iwasawa cocycle
by
$$\widehat{T}(b,x,z)\defi(Sb,b_1^{-1}x,E(b)^{-1}z).$$

The horocyclic flow is an $\bR$-action on $B^{X,Z}$ which interacts with $\widehat{T}$ in a manner reminiscent 
to the interaction of the standard horocyclic and geodesic flows on the unit tangent bundle of the upper half plane and
hence the terminology.
Recall the notation introduced in \S\ref{ssec:boundary maps} and in particular, the groups $L_0,R_0, U_0\defi L_0/R_0$ and 
the resulting equivariant families $L_\eta,R_\eta,U_\eta$ for $\eta\in H/P$ as well as the notation
$L_b,R_b,U_b$ defined for $\be$-almost every $b\in B$.
We  denote the Lie algebras of these groups by corresponding Gothic letters and note that naturally $\gou_0 
= \gol_0/\gor_0$ and similar identifications exist 
when the subscript $0$ is replaced by $\eta \in H/P$ or $b$ in the domain of definition of 
the boundary map $\xi$.
Observe that although for $\eta=gP\in H/P$ the map $\Ad_g$ maps $\gol_0$ to 
$\gol_\eta$ and $\gor_0$ to $\gor_\eta$ and 
therefore descends to a map $\Ad_g:\gou_0\to \gou_\eta$, this map is \tb{not} well defined in the sense that it depends
on the choice of representative $g$ for the coset $\eta$. This is remedied as follows.
Recall the section $s:H/P\to H/N$ that was chosen in~\S\ref{ssec:Iwasawa and lyapunov} where $N$ is the unipotent 
radical of the minimal parabolic $P$ of $H$.
Observe that although $N$ acts via the adjoint representation non-trivially on $\gol_0, \gor_0$ respectively, these actions 		    descend to the  
trivial action on the quotient $\gou_0$. Thus, given $\eta\in H/P$ with $s(\eta)=gN\in H/N$, we do have a well defined map 
$\gou_0\to\gou_\eta$ given by 
$$\ell+\gor_0\mapsto \Ad_g(\ell + \gor_0) = \Ad_g(\ell) + \gor_\eta\in \gol_\eta/\gor_\eta =\gou_\eta.$$
By abuse of notation we denote this map $$\Ad_{s(\eta)}:\gou_0\to\gou_\eta.$$
Precomposing with the boundary map $\xi$ we obtain the isomorphisms
$\Ad_{s(\xi(b))}:\gou_0\to \gou_b$ defined for $\be$-almost every $b\in B$.
Note also that $Z$ acts on $\gou_0$ via the adjoint representation and hence the isomorphism 
$\Ad_{s(\xi(b))z}=\Ad_{s(\xi(b))}\Ad_z:\gou_0\to\gou_b$ is well defined for all $z\in Z$ and $\be$-almost every $b\in B$.
%Throughout $\nu\in \cP(X)$ is a $\mu$-stationary measure and $(B^X,T,\be^X)$ is the backwards dynamical system attached to it.
Following~\cite{BQJams}, for $u\in\gou_0$ we define the \textit{horocyclic flow} $\Phi_u:B^{X,Z}\to B^{X,Z}$
by
\begin{equation}\label{eq:horoflow}
    \Phi_u(b,x,z) \defi (b,\exp(\Ad_{s(\xi(b))z}(u))x,z).
\end{equation}
Further clarification is needed in this definition: For $\eta\in H/P$ and 
$\bar{\ell}=\ell+\gor_\eta\in \gou_\eta$, $\exp(\bar{\ell})\defi \exp(\ell)R_\eta
\in U_\eta$ is well defined. Moreover, since the action of $R_\eta$ on the plane $p_\eta$ is trivial, the group $U_\eta$ acts 
on the fibre $\pi^{-1}(p_\eta)\subset X$. By Proposition~\ref{prop:firstprop}, for $\be$-almost every $b\in B,$ $\nu_b$ is supported on
$\pi^{-1}(p_b)$ and therefore we conclude from \eqref{eq:2233} 
that for $\be^{X,Z}$-almost every $(b,x,z)\in B^{X,Z}$ we have that $x\in\pi^{-1}(p_b)$ and 
$\exp(\Ad_{s(\xi(b))z}(u))\in U_b$ so 
equation \eqref{eq:horoflow} 
makes sense.

We will utilise the joint action of $\wh{T}$ and the flow $\Phi_{\gou_0}$ on $B^{X,Z}$. 
A key point is the following lemma. 
\begin{lemma}\label{lem:commutation}
    For $\be$-almost every $b\in B$, for any $\hs=(b,x,z)\in B^{X,Z}$ one has 
    $$\Phi_u\comp\widehat T(\hs)=\widehat T\comp\Phi_u(\hs)
    \quad \textrm{for all}\ u\in\gou_0 .$$
\end{lemma}
\begin{proof}
    Recall that by~\eqref{eq:section eqn} and the definition of $\on E$ for $\be$-almost every $b\in B$ one has
\begin{equation}\label{eq:comm1}b_1^{-1}s(\xi(b))= s(\xi(Sb)) \on E(b)^{-1}.\end{equation}
    For arbitrary $u\in\gou_0$, using the definitions we have that
    $$\Phi_u\comp\widehat T(b,x,z)=(Sb,\exp(\Ad_{s(\xi(Sb))\on E(b)^{-1}z}u)b_1^{-1}x,\on E(b)^{-1}z)$$
    and
    $$\widehat T\comp\Phi_u(b,x,z)=(Sb,b_1^{-1}\exp(\Ad_{s(\xi(b))z}u)x,\on E(b)^{-1}z).$$
    Once $b$ satisfies \eqref{eq:comm1} these two expressions are equal
    and the lemma follows.
\end{proof} 
Later on it will be important for us to restrict attention to a `finite window' in the $Z$-coordinate. Let $U\subset Z$ be 
a bounded measurable set of finite positive $\lam$-measure
and define\footnote{Later on we will take $U$ to be the image under the exponential map of the unit ball in $\goz$.}
\begin{equation}\label{eq:2136a}
    \bxu\defi B^X\times U\quad\textrm{and}\quad\bexu\defi\bexz|_{\bxu}.
\end{equation}
Note that $B^{X,U}$ is $\Phi_{\gou_0}$-invariant but not $\wh{T}$-invariant. 
The following proposition (in which the role of $U$ is insignificant) 
shows why the horocyclic flow is natural from the point of view of Proposition~\ref{prop:unipinv}.
\begin{proposition}\label{prop:unipinv2}
    The measure $\bexu$ is $\Phi_u$-invariant for all $u\in\gou_0$ if and only if for $\be$-almost every $b\in B$ the measure $\nu_b$ is $U_b$-invariant.
\end{proposition}
\begin{remark}\label{rk:unipinv2}
    In particular,
    by Proposition~\ref{prop:unipinv}, we can prove Theorem~\ref{thm:full}\eqref{main1} by
    establishing the $\Phi_{\gou_0}$-invariance of $\bexu$.
\end{remark}
Proposition~\ref{prop:unipinv2} is a straightforward corollary of the following lemma.
%\osadd{check that this is true!!}
\begin{lemma}\label{lem:absdisint}
    Let $\rho:(Y,\eta) \to (Y',\eta')$ be a morphism of Borel probability spaces. Let $\eta = \int_{Y'}\eta_y \dv{\eta'} y$ be the
    disintegration of $\eta$ over $\eta'$ and $M:Y\to Y$ be a measurable map such that $\rho = \rho\comp M$.
    Then $\eta$ is $M$-invariant if and only if for $\eta'$-almost every $y\in Y'$, $\eta_y$ is $M$-invariant.
\end{lemma}
The proof of Lemma~\ref{lem:absdisint} is a direct consequence of the uniqueness of disintegration and is left to the reader.

\subsection{\label{ssec:lwms}Leafwise measures}
We begin with some general notation and measure theory.
Given a 
locally compact second countable Hausdorff space $Y$ we let $\cM(Y)$ denote the space of Radon measures on
$Y$. We equip $\cM(Y)$ with the coarsest topology so that $\theta\mapsto \theta(f)\defi \int_Yf\mathrm{d}\theta$ 
is continuous for 
any $f\in C_c(Y)$.
We let $\bP\cM(Y)$ denote
the space of equivalence classes of measures in $\cM(Y)$ under the equivalence relation of proportionality
and equip it with the quotient topology. 
For $\eta\in \cM(Y)$ we
denote by $\br{\eta}$ its equivalence class in $\bP\cM(Y)$.
For $\eta\in\cM(Y)$ and a set $V\subset Y$ of finite $\eta$-measure,
we let $\eta|_V\in\cP(V)$ be given by $(\eta|_V)(F)\defi \eta(F\cap V)/\eta(V)$ for all measurable $F\subseteq V$.
Given a countably generated
sub-$\sig$-algebra $\cA$ of the Borel $\sig$-algebra, the atom of $y$ with respect to $\cA$ is the smallest $\cA$-measurable set containing $y$ and
we denote it by $[y]_\cA$.
Given $\eta\in \cM(Y)$, the conditional measures of $\eta$ along
$\cA$ are a collection $\set{\eta_y^\cA}_{y\in E}$ of probability measures $\eta_y^\cA\in \cP(Y)$, 
where $E$ is a measurable subset
of $Y$ of full $\eta$-measure such that for any $\eta$-integrable function $f$ on $Y$,
the map $y\mapsto \int_Y f \mathrm{d}\eta_y^\cA$
is the conditional expectation $\bE(f|\cA)$. It then follows that $\eta$-almost surely $\eta_y^\cA\in\cP(\br{y}_\cA)$.
If $Y$ is a group, then for $y\in Y$ we denote by $l_y:Y\to Y$ the translation by $y$ on the left. This induces an action of $Y$ on $\cM(Y)$,
$(y,\eta)\mapsto (l_y)_*\eta$.
This action respects the equivalence relation of proportionality and hence descends to an action on $\bP\cM(Y)$ which we
denote $(y,\br{\eta})\mapsto (l_y)_*\br{\eta}$.

We will use the theory of leafwise measures as presented in~\cite[\S 6]{ELpisa}, \cite[\S4]{BQAnnals}. 
This is a measure theoretic toolbox developed by Katok-Spatzier, 
Lindenstrauss, Benoist-Quint and Einsiedler-Katok-Lindenstauss \cite{MR1619571,MR2195133,BQAnnals,MR2247967}
which captures the way a measure on a space disintegrates with respect to the action of a group.
We will follow the notation and terminology of~\cite{ELpisa}.

Let $(Y,\cY)$ be a standard Borel space and let $\Psi_t$ a measurable $\bR$-action\footnote{The theory concerns itself
with a locally compact second countable metrizable topological group but we will focus on flows.} (a flow) on $Y$. Assume 
that for any $y\in Y$ the stabiliser $\Stab_\bR(y)$ is a discrete subgroup of $\bR$ and let $\eta\in \cM(Y)$ be a finite measure.
The construction in 
\cite[\S4]{BQAnnals} (see also~\cite[\S 6]{ELpisa}) gives rise to 
a measurable map $y\mapsto (\eta)_y^\Psi$ from a measurable 
subset of full $\eta$-measure $E\subset Y$ to $\cM(Y)$ having the following properties:
\begin{enumerate}[\bfseries P1:]
    \item\label{pCh} \tb{Characterising property}. Given a measurable subset $E'\subset E$ 
        with $\eta(E')>0$ and a countably
        generated sub-$\sig$-algebras $\cA$ of the Borel $\sig$-algebra on $E'$
        whose atoms are of the form 
        $\br{y}_\cA=\set{\Psi_t(y):t\in o_y}$
        for some open and bounded $o_y\subset \bR$, for all $y\in E'$, then the 
        push-forward of $(\eta)_y^\Psi|_{o_y}$ via the orbit map $t\mapsto\Psi_t(y)$
        is the conditional measure $(\eta|_{E'})_{y}^\cA$.
    \item\label{pR} \tb{Rootedness}. For all $y\in E$, we have $0\in\supp{(\eta)_y^\Psi}$.
    \item\label{pN} \tb{Normalisation}. For any $y\in E$, ${(\eta)_y^\Psi}([-1,1]) = 1$
    \item\label{pC} \tb{Compatibility}. For all $y\in E$  and $t\in\bR$ such that $\Psi_t(y)\in E$ one has
        \begin{equation*}
            %\label{eq:1545}
            \br{ {(\eta)_y^\Psi}} = (l_t)_*\br{ {(\eta)_{\Psi_t(y)}^\Psi}}.
        \end{equation*}

\end{enumerate}
Property~\tb{P\ref{pCh}} is a characterising property in the sense that if $y\mapsto \sig(y)$ is 
a measurable map defined on a set of full $\eta$-measure  into $\cM(Y)$ such that \tb{P\ref{pCh}} is satisfied
then $\br{\sig(y)} = \br{(\eta)_y^\Psi}$ for $\eta$-almost every $y\in Y$. 
Property~\tb{P\ref{pN}} is a convenient way to choose in a measurable manner a well defined measure
in the equivalence class $\br{(\eta)_y^{\Psi}}$ which is well defined for $\eta$-almost every $y \in Y$ by \tb{P\ref{pCh}, P\ref{pR}}.

We call the map $y\mapsto {(\eta)_y^\Psi}$ satisfying properties \tb{P\ref{pCh}-P\ref{pC}} the \textit{leafwise measure-map} (LWM-map) 
of $\eta$ with respect to the flow $\Psi_\bR$ and the set $E$ is called a domain of the LWM-map.
The measure ${(\eta)_y^\Psi}$ is called the \textit{leafwise measure} (LWM) 
of $\eta$ at $y$ with respect to the flow $\Psi_\bR$.

We shall consider the LWM-map of the infinite Radon measure
$\be^{X,Z}\in \cM(B^{X,Z})$ with respect to the flow $\Phi_{\gou_0}$. The fact that 
this measure is infinite does not matter much as one can present $B^{X,Z}$ as a countable union of 
$\Phi_{\gou_0}$-invariant sets of the form $B\times X\times U_i$, where for example $U_i$ is the ball of radius $i$ in $Z$ centred at 
the identity and the restriction of $\be^{X,Z}$ to each such set has finite measure.  
\begin{comment}
    The next lemma utilises the commutation relation 
    in Lemma~\ref{lem:commutation} and shows that the LWM-map of $\be^{X,Z}$ is constant along $\wh{T}$-orbits.
    \begin{lemma}\label{lem:measureablity}
        For $\be^{X,Z}$-almost every $\hs\in B^{X,Z}$ and all $n\in \bN$ 
        \begin{equation}\label{eq:1245a}
            \lwmz{\hs} = \lwmz{\wh{T}^n(\hs)}.
        \end{equation}
    \end{lemma}
    \begin{proof}
        The statement for $n=1$ follows from the commutation relation proved in Lemma~\ref{lem:commutation}
        together with the fact that $\wh{T}$ preserves $\be^{X,Z}$ because of the uniqueness of the LWM-map. 
        This propagates to the statement of the lemma by intersecting countably many co-null sets.
    \end{proof}
\end{comment}
%%%%%%%%%%%%%%%%%%%%%%%%%%%%%%%%%Begin comment
In fact, due to the fact that the 
flow $\Phi_{\gou_0}$ respects the disintegration $\be^{X,Z} = \int_{B\times Z}\del_b\otimes \nu_b\otimes\del_z\dv\be{b}\dv\lam z$
we have the following.
\begin{lemma}\label{lem:a.s.identity1}
    For $\be\otimes \lam$-almost every $(b,z)\in B\times Z$ and $\del_b\otimes\nu_b\otimes \del_z$-almost every $\hs\in \hB$,
    $$(\del_b\otimes \nu_b\otimes \del_z)_{\hs}^\Phi = (\bexz)_{\hs}^{\Phi}.$$
\end{lemma}
\begin{proof}
    Let $E$ be a domain for the LWM-map of $\bexz$ with respect to the flow $\Phi_{\gou_0}$. 
    Let $(b,z)\in B\times Z$ be such that the slice $E_{(b,z)} = \set{\hs = (b,x,z)\in E}$ has full
    $\del_b\otimes\nu_b\otimes \del_z$-measure, which holds $\be\otimes \lam$-almost surely since 
    $E$ is of full $\bexz$-measure. It is straightforward to check that 
    for $\be$-almost every $b\in B$, the assignment $\hs\mapsto\lwm{\hs}$ satisfies the characterising property 
    of the LWM-map of $\del_b\otimes \nu_b\otimes \del_z$ and by the uniqueness of the LWM-map the statement of the lemma follows.
\end{proof}

The next lemma utilises the commutation relation 
in Lemma~\ref{lem:commutation} and shows that the LWM-map of $\be^{X,Z}$ is constant along $\wh{T}$-orbits.

\begin{lemma}\label{lem:measureablity}
    For $\be^{X,Z}$-almost every $\hs\in B^{X,Z}$ and all $n\in \bN$, 
    \begin{equation}\label{eq:1245a}
        \lwmz{\hs} = \lwmz{\wh{T}^n(\hs)}.
    \end{equation}
\end{lemma}
\begin{proof}
    For $\be\otimes \lam$-almost every $(b,z)\in B\times Z$,  
    $$\wh{T}: (B^{X,Z},\del_b\otimes\nu_b\otimes \del_z)\longrightarrow 
    (B^{X,Z},\del_{Sb}\otimes\nu_{Sb}\otimes \del_{\mbe(b)^{-1}z})$$ 
    is an isomorphism of probability spaces which by Lemma~\ref{lem:commutation} commutes with 
    the flow $\Phi_{\gou_0}$. It thus follows from the uniqueness of the LWM-map that 
    for $\be\otimes \lam$-almost every $(b,z)\in B\times Z$ and $\del_b\otimes \nu_b\otimes \del_z$-almost every 
    $\hs\in B^{X,Z}$ one has the equality 
    $(\del_b\otimes \nu_b\otimes \del_z)_{\hs}^\Phi = (\del_{Sb}\otimes \nu_{Sb}\otimes \del_{\mbe{b}^{-1}z})_{\wh{T}(\hs)}^{\Phi}$.
    Taking into account Lemma~\ref{lem:a.s.identity1} we deduce that for $\bexz$-almost every $\hs\in\bxz$
    the equality $\lwmz{\hs} = \lwmz{\wh{T}(\hs)}$ holds. This propagates to the statement of the lemma by intersecting 
    countably many sets of full measure.
\end{proof}              
Preparing the grounds for the drift argument we restrict attention to
a finite window and consider the probability space $(B^{X,U},\bexu)$ as in~\eqref{eq:2136a}.
The relevance of the LWM's to our discussion is the following statement.
See \cite[Problem 6.28]{ELpisa} and \cite[Proposition 4.3]{BQAnnals}.                        
\begin{theorem}\label{thm:ashaar}
    The measure $\bexu$ is $\Phi_{\gou_0}$-invariant if and only if $\lwm {\bd s}$
    is equal to the Haar measure on $\gou_0$ for $\bexu$-almost every $\bd s\in \bxu$.
\end{theorem}
In particular, in order to prove that $\bexu$ is $\Phi_{\gou_0}$-invariant
we need
to show that $\lwm{\mb{s}}$
%it was written \lwm {\sd s} but it did not compile 
is Haar $\bexu$-almost surely.
Note that by Remark~\ref{rk:unipinv2}, this would complete
the proof of Theorem~\ref{thm:full}\eqref{main1}.
The Haar measure is characterised as the unique $\eta\in\cM(\gou_0)$ such that $\Stab_{\gou_0}(\eta) = \gou_0$.
Thus our goal is to establish the $\bexu$-almost sure equality
$$\Stab_{\gou_0}(\lwm{\mb{s}})= \gou_0.$$
Because of property {\bfseries P\ref{pC}} of the LWM's, they interact more naturally with the action of $\gou_0$ on $\bP\cM(\gou_0)$.
In turn, the drift argument in \S\ref{ssec:drift} will produce the almost sure equality $\Stab_{\gou_0}(\br{\lwm{\bd s}}) = \gou_0$.
Hence the importance of the following proposition to our discussion.
\begin{proposition}\label{prop:equalstab}
    For $\bexu$-almost every $\bd s\in \bxu$ one has
    $$\Stab_{\gou_0}(\lwm {\bd s}) = \Stab_{\gou_0}(\br{\lwm {\bd s}}).$$
\end{proposition}
\begin{proof}
    We will prove that
    $$\Stab_{\gou_0}(\lwm {\bd s}) \supseteq \Stab_{\gou_0}(\br{\lwm {\bd s}})$$
    for $\bexu$-almost every $\bd s\in\bxu$ since the reverse inclusion is obvious. By~\cite[Theorem 6.30]{ELpisa}
    for $\bexu$-almost every $\bd s\in\bxu$ we have
\begin{equation}\label{eq:polygrowth}\lim_{T\to\infty}\frac{1}{T^4}\lwm {\bd s}([-T,T])=0.\end{equation}
    Let $\bd s\in\bxu$ be such a point. If $u\in \Stab_{\gou_0}(\br{\lwm {\bd s}})\smallsetminus\set 0$ then
    there exists $c\ne 0$ such that $(l_u)_*\lwm {\bd s} = c \lwm {\bd s}$, which implies that for all $i\in \bN$,
    $(l_{iu})_*\lwm {\bd s} = c^i\lwm {\bd s}$.
    We would like to show that $c=1$. For all $n\ge 1$ we have
    \begin{align*}
        \lwm {\bd s}([-n\av u,n\av u])&\ge
        %\lwm {\bd s}(\cup_{i=1-n}^{n-1} l_{iu}([-\av u/2,\av u/2]))
        \sum_{i =1-n}^{n-1} \lwm {\bd s}(l_{iu}[-\av u/2,\av u/2])  \\
        &=  \sum_{i =1-n}^{n-1} ((l_{iu})_*\lwm {\bd s})([-\av u/2,\av u/2]) \\
        &= \lwm {\bd s}([-\av u/2,\av u/2])\sum_{i=1-n}^{n-1} c^i.
    \end{align*}
    By property {\bfseries P\ref{pR}} of the LWM's, $\lwm {\bd s}([-\av u/2,\av u/2])>0$ and hence we see that
    unless $c=1$ the volume $\lwm {\bd s}([-n\av u,n\av u])$ is growing exponentially in $n$ which
    contradicts~\eqref{eq:polygrowth} as desired.
    %hus by the aforementioned result
    %\cite[Theorem 6.30]{ELpisa} \usadd{check reference} we must have $c=1$ and so $v\in \Stab_{\gou_0}\lwm$
\end{proof}
\subsection{Zooming in on the atoms}\label{ssec:loa}

Let $\bcxz$ be the Borel $\sig$-algebra of $\bxz$ and define
$$\cQ_n\defi \widehat T^{-n}(\bcxz)\quad\textrm{and}\quad\cQ_\infty\defi\cap_{i=0}^\infty\cQ_i.$$
For $\bd s\in \bxz$ the atom of $\hs$ with respect to $\cQ_n$ is given by
$$\br{\bd s}_{\cQ_n}=\set{\bd s':\wT^n(\bd s)=\wT^n(\bd s')}$$
and the atom of $\hs$ with respect to $\cQ_\infty$ is given by
$$\br{\bd s}_{\cQ_\infty}=\set{\bd s':\textrm{there exists }n\in\bN\textrm{ with }\wT^n(\bd s)=\wT^n(\bd s')}.$$
Recall that $A\defi\supp\mu$.
For $a\in A^n$ we let
\begin{equation}\label{eq:0855}
    \bd s(a) \defi (aS^nb, a_1^n(b_1^n)^{-1}x,\on E_n(aS^nb)\mbe_n(b)^{-1}z).
\end{equation}
We then have an identification $A^n\cong\br{\bd s}_{\cQ_n}$ via the map $a\mapsto \bd s(a)$.
It is easy to see that via this identification the probability measure $\mu^{\otimes n}$ on $A^n$ corresponds 
to the conditional measure $(\be^{X,Z})_{\hs}^{\cQ_n}$ (cf.\ \cite[Lemma 3.3]{BQJams}). 

We will need to consider $\sig$-algebras whose atoms are
tiny parts of the above atoms. This is done as follows.
From now on we fix 
\begin{equation}\label{eq:1436a}
    U\defi\exp(\set{z\in\goz:\norm z <1})\subset Z
\end{equation} 
and recall the notation and definition in~\eqref{eq:2136a}.
We define
$\cQ_n^U$ (resp. $\cQ_\infty^U$) to be the restriction of $\cQ_n$ (resp.\ $\cQ_\infty$) to $\bxu$.
For $\bd{s}\in \hBU$ the atom of $\hs$ with respect to $\cQ_n^U$ is given by
$$\br{\bd s}_{\cQ_n^U} = \set{\bd s(a) : a\in A^n\ \textrm{and}\ \mbe_n(aS^nb) \mbe_n(b)^{-1}z\in U}.$$
We therefore let
\begin{equation}\label{eq:0841}
    A_{\bd s,U}^n \defi A_{(b,z),U}^n \defi \set{a\in A^n : \mbe_n(aS^nb)\mbe_n(b)^{-1}z\in U}
\end{equation}
be the subset of $A^n$ corresponding to the subset $\br{s}_{\cQ_n^U}$ of $\br{s}_{\cQ_n}$.

If $\mu^{\otimes n}(A_{\bd s,U}^n) > 0$, then we denote by $\mu^{\otimes n}_{\bd s,U}$
the normalised restriction of $\mu^{\otimes n}$ to $A_{\bd s,U}$. That
is
\begin{equation}\label{eq:0700}
    %\hmu(E) \defi \frac{\mu^{\otimes n}(E\cap A_{\bd s,U}^n)}{\mu^{\otimes n}(A_{\bd s,U}^n)}\ \textrm{for all}\ E\subseteq A^n.
    \hmu\defi\mu^{\otimes n}|_{A_{{\bd s},U}}.
\end{equation}
Note that $\hmu$ only depends on the $B$ and $Z$ co-ordinates of $\bd s$.
By~\cite[Lemma 3.6 + Equation (3.5)]{BQJams}, under the identification $a\mapsto \hs(a)$ of 
$A^n_{\hs,U}$ and $\br{\hs}_{\cQ_n^U}$ we have
that for $\bexu$-almost every ${\bd s}\in B^{X,U}$,
%$(\be^{X,U})_{\bd s}^{\cQ^U_n}$ are given by
\begin{equation}\label{eq:1039}
    \hmu = (\bexu)_{\bd s}^{\cQ^U_n}.
\end{equation}
%Via the identification $a\mapsto \bd s(a)$ of $A^n$ and $\br{\bd s}_{\cQ_n}$, the measure
%$\mu^{\otimes n}$ is mapped to the conditional measure $(\be^{X,Z})_{\bd s}^{\cQ_n}$
%(See~\cite[Lemma 3.6]{BQJams}.).

%The results of Benoist and Quint that we are going to cite require another randomization. We denote by $\lam$ a
%choice of a Haar measure on $Z$ and fix from now on
%\begin{equation}\label{eq:U}
%    U = \exp (B_1^{\goz}).
%\end{equation}
%Instead of getting statements of the form ``for $\be^X$-almost every $\hp$ some statement holds involving $\mu^{\otimes n}_{\hp,U}$", we would get statements of the form ``for $\be^X$-almost every $\hp$ for $\lam$ almost any $z\in U$ some statement holds
%involving $\mu^{\otimes n}_{\hp,Uz^{-1}}$". This will not cause too much trouble because since $z\in U$, $Uz^{-1}$ is still bounded
%and the time-change factor will still be controlled.

We are now ready to cite the essential technical results from~\cite{BQJams} that will allow us to analyse in detail the growth and directions of
sequences of vectors corresponding to displacements between points in $X$. 
These results are stated in terms of the conditional measures $\mu^{\otimes n}_{\hs, U}$.
For $\delta>0$ we use the notation
$x\asymp_\delta y$ to mean there exists a constant $c_\delta\geq 1$ depending on $\delta$ such that that $c_\del^{-1}x<y<c_\del x$ for all $x,y\in\bR$.

The following lemma is used in order to control the growth of displacements.

\begin{lemma}\label{law of norms 0}
    Let $V$ be a finite dimensional representation of $H$.
    For $\bexu$-almost every $\bd s\in\bxu$
    and all $\del>0$
    there exists  $n_0>0$ such that for all $n>n_0$,
    $\om\in \cH_{\goz}(V)$ and $v\in V[\om]\smallsetminus \set 0$ one has
    \begin{align}
        \label{eq:19140} \hmu(\set{a\in A^n:\norm{a_1^n v}\asymp_{\del} \norm{a_1^n}\norm{v}}) >1-\del.
    \end{align}
\end{lemma}
\begin{proof}
    This is the first part of~\cite[Proposition 4.21]{BQJams} where the conditional measures $\be^U_{n,c}$ (in the notation
    of Benoist and Quint) equal $\hmu$~\cite[Lemma 3.6 + Equation (3.5)]{BQJams}.
\end{proof}
We will use Lemma~\ref{law of norms 0} in the following form.
\begin{corollary}\label{law of norms}
    Let $V$ be a finite dimensional representation of $H$.
    Then for $\bexu$-almost every $\bd s=(b,x,z)\in\bxu$
    and all $\del>0$
    there exists  $n_0>0$ such that for all $n>n_0$,
    $\om\in \cH_{\goz}(V)$ and $v\in V[\om]\smallsetminus\set 0$ one has
    \begin{align}
        \label{eq:1914} \hmu (\set{a\in A^n:\norm{a_1^n v}\asymp_{\del} \chi^\om(\mbe_n(b))\norm{v}}) >1-\del.
    \end{align}
\end{corollary}
\begin{proof}
    Let $\bd s=(b,x,z)\in \bxu$ be such that the conclusion of Lemma~\ref{law of norms 0}
    %and
    %\ref{lem:expansion of leafs-1}
    holds for $\bd s$.
    Given $\del>0$
    we get the existence of $n_0$ such that~\eqref{eq:19140} holds for all $n>n_0$, $\om\in \cH_{\goz}(V)$
    and all $v\in V[\om]\mz$.

    Let $\om\in \cH_{\goz}(V)$ and $v_0\in V_{S^n b}[\om]$ a unit vector.
    By Lemma~\ref{lem:expansion of leafs-1}, if $\mb{s}$ is outside a $\bexu$-null set,
    for $\mu^{\otimes n}$-almost every $a\in A^n$ one has
    $$\norm{a_1^nv_0}/\norm{v_0} = \chi^\om(\mbe_n(aS^nb)).$$
    Applying equation~\eqref{eq:19140} to $v_0$ we get that for all $n>n_0$ and $\om\in \cH_{\goz}(V)$,
    \begin{equation}\label{eq:1146}
        \hmu(\set{a\in A^n:\norm{a_1^n}\asymp_{\del} \chi^\om(\mbe_n(aS^nb))}) >1-\del.
    \end{equation}
    Taking into account that we are conditioning on the fact that $$\mbe_n(aS^nb)\mbe_n(b)^{-1}z\in U$$
    and that  $z\in U$, we may replace $\mbe_n(aS^nb)$ with $\mbe_n(b)$ in~\eqref{eq:1146} by
    modifying the implied constant if necessary.
    This gives us that for all $n>n_0$ and $\om\in \cH_{\goz}(V)$,
    $$\hmu(\set{a\in A^n:\norm{a_1^n}\asymp_{\del} \chi^\om(\mbe_n(b))}) >1-\del.$$
    Applying again~\eqref{eq:19140} we get that for all $n>n_0$, $\om\in \cH_{\goz}(V)$ and $v\in V[\om]\mz$ one has
    $$\hmu(\set{a\in A^n:\norm{a_1^n v}\asymp_{\del} \chi^\om(\mbe_n(b))\norm{v}}) >1-2\del,$$
    which finishes the proof up to replacing $\del$ by $\del/2$.
\end{proof}
The following lemma will allow us to control the direction of displacements.
For any vector space $V$ we use the distance $\on d_{\bP V}$ on $\bP V$ defined so
that for all $v\in V\mz$ and $W\subseteq V$ one has $$\on{d}_{\bP V}(\bR v, W) \defi \min_{w\in W}\frac{\norm{v\wedge w}}{\norm{v}\norm{w}}.$$
Note that $\on d_{\bP V}(\bR v,W)=0$ if and only if $\bR v\subseteq W$.
\begin{lemma}\label{law of angles 2}
    Let $V$ be a finite dimensional representation of $H$.
    Then for $\bexu$-almost every $\bd s\in \bxu$ and
    for all $\rho>0$ and $\del>0$
    there exists  $n_0>0$ such that for all $n>n_0$,
    $\om\in \cH_{\goz}(V)$, $v\in V[\om]\mz$ and
    $\eta\in H/P$ one has
    \begin{align}
        %\label{eq:1914} \mu^{\otimes n}&\set{a_1^n\in A^n_{p,U}:\norm{a_1^n v}\asymp_U \exp(\om(\on{L}_{n}(b)))\norm{v}} >1-\del \\
        \label{eq:1915}\hmu(\set{a\in A^n:\on{d}_{\bP V}(a_1^n\bR v, a_1^n V_\eta[\om])<\rho}) >1-\del.
    \end{align}
\end{lemma}
\begin{proof}
    This is exactly the second part of~\cite[Proposition 4.21]{BQJams} where the conditional measures $\be^U_{n,c}$ (in the notation
    of Benoist and Quint) equal $\hmu$ by~\cite[Lemma 3.4 + Equation (3.5)]{BQJams}.
\end{proof}
In our application of Lemma~\ref{law of angles 2} we will not know that the vector $v$ belongs to a single isotypic component.
The following lemma will allow us to obtain similar statements for vectors which do not lie in a single isotypic component.
\begin{lemma}\label{max weight pos angle}
    Let $V$ be a representation of $H$ and assume that $\cH_{\goz}(V)$ contains a maximal weight $\om_\mb{m}$.
    Let $\al>0$ and $V_{\sphericalangle \al}[\om_{\mb{m}}]$ be as in~\eqref{eq:2300}.
    %$$V_{\sphericalangle \al}[\om_{\mb{m}}] = \set{v\in V: \al \norm{v}\le \norm{\tau_{\om_\mb{m}}(v)}}.$$
    Then for $\bexu$-almost every $\bd s\in \bxu$ and
    for all $\rho>0$ and $\del>0$ there exists $n_0>0$ such that for all $n>n_0$
    and $v\in V_{\sphericalangle \al}[\om_\mb{m}]\mz$ one has
    \begin{equation}\label{eq:1025}
        \hmu(\set{a\in A^n: \on{d}_{\bP V}(a_1^n \bR v, a_1^n \bR \tau_{\om_\mb{m}}(v)) \le \rho})>1-\del.
    \end{equation}
\end{lemma}
\begin{proof}
    Let $v\in V_{\sphericalangle \al}[\om_\mb{m}]\mz$.
    %be a non-zero vector and write
    %$v = \sum_{\om\in\cH_{\goz}(V)} \tau_\om(v).$
    Then for any $g\in H$ one has
    %\usnote{we use here the fact that the norm is induced by an inner product such that the isotypic components are orthogonal}
    \begin{align}
        \nonumber \on{d}_{\bP V}(g \bR v, g\bR \tau_{\om_{\mb{m}}}(v)) &= \frac{\norm{(\sum_{\om\in \cH_{\goz}(V)} g\tau_\om(v))\wedge g\tau_{\om_{\mb{m}}}(v)}}{\norm{gv} \norm{g\tau_{\om_{\mb{m}}}(v)}}\\
        \label{eq:1803} & = \frac{\norm{\sum_{\om\in\cH_{\goz}(V)\smallsetminus\set {\om_{\mb{m}}}} g\tau_\om(v)}\norm{g\tau_{\om_{\mb{m}}}(v)}}
        {\norm{gv} \norm{g\tau_{\om_{\mb{m}}}(v)}}\\
        %\nonumber& \ll \frac{\max_{\om\ne \om_{\mb{m}}} \norm{g\tau_\om(v)}\norm{g\tau_{\om_{\mb{m}}}(v)}}{\norm{g\tau_{\om_{\mb{m}}}(v)}^2}\\
        \nonumber& \ll \frac{\max_{\om\in\cH_{\goz}(V) \smallsetminus\set {\om_{\mb{m}}}} \norm{g\tau_\om(v)}}{\norm{g\tau_{\om_{\mb{m}}}(v)}}.
    \end{align}
    Now given $\rho>0$, by parts~\eqref{asconvergence} and~\eqref{simplelyap} of Theorem~\ref{thm:poslyap}, for $\be$-almost every $b\in B$
    there exists $n_0>0$ so that for all $n>n_0$ and
    $\om \in \cH_{\goz}(V)\smallsetminus\set {\om_{\mb{m}}}$ one has
    \begin{equation}\label{eq:1821}
        \exp({(\om-\om_{\mb{m}})(\on{L}_{n}(b))}) = \frac{\chi^\om(\mbe_n(b))}{\chi^{\om_{\mb{m}}}(\mbe_n(b))}\le \rho.
        %\chi_{\om/{\om_{\mb{m}}}}(\mbe_n(b))<\rho.
    \end{equation}
    By enlarging $n_0$ if necessary and using Corollary~\ref{law of norms} we get that for $\bexu$-almost every $\bd s\in\bxu$ and for
    all $\del>0$ and $n>n_0$ there is $F\subset A^n$ with
    $\hmu(F)>1-\del$ such that for all $a\in F$, $\om\in \cH_{\goz}(V)$ and $v\in V$ such that
    $\tau_{\om}(v)\ne 0$ we have $\norm{a_1^n\tau_\om(v)}\asymp_\del \chi^\om(\mbe_n(b))\norm{\tau_\om(v)}$.
    Thus, using~\eqref{eq:1803},~\eqref{eq:1821} and the assumption that $v\in V_{\sphericalangle \al}[\om_\mb{m}]\mz$ we get
    %$$\on{d}_{\bP V}(a_1^n \bR v, a_1^n\bR \tau_{\om_{\mb{m}}}(v))\ll_\del
    %\frac{\max_{\om\ne \om_{\mb{m}}} \chi^\om(\mbe_n(b))  \norm{\tau_\om(v)}}{\chi^{\om_{\mb{m}}}(\mbe_n(b))
    %\norm{\tau_{\om_{\mb{m}}}(v)}}.$$
    %Applying~\eqref{eq:1821} and the assumption that $\norm{\tau_{\om_\mb{m}}(v)}\le\norm{v}/\al$ this gives
    $$\on{d}_{\bP V}(a_1^n \bR v, a_1^n\bR \tau_{\om_{\mb{m}}}(v))\ll_{\del}\rho/\al$$
    for all $a\in F$ which, up to adjusting $\rho$, is the claim of the lemma.
    %That is, by adjusting $\rho$ properly we have established that for $\be$-almost every $b$, for $\lam$-almost every $z\in U$, for any $\del,\rho>0$ for all large enough $n$ there exists a set $F\subset A^n$ with $\mu^{\otimes n}_{b,Uz^{-1}}(F)>1-\del$
    %such that for any $a_1^n\in F$, $\on{d}_{\bP V}(g \bR v, g\bR \tau_{\om_{\mb{m}}}(v))\le \rho$, which is the claim saught.
\end{proof}
%\begin{definition}\label{def:Ustar}
%    We denote by $U_\star$ the subset of $U$ (which is of full $\lam$-measure) on which the statements of
%    Corollary~\ref{law of norms}, Lemma~\ref{law of angles 2} and Lemma~\ref{max weight pos angle} hold
%    when applied to the representation  $V=\wedge ^4\gog$ of $H$ and the weight $\oml\in \cH_{\goz}(V)$.
%    Note that by Lemma~\ref{lem:max weight}, $\oml$ is maximal in $\cH_{\goz}(V)$ so  Lemma~\ref{max weight pos angle} is applicable.
%\end{definition}

%%%%%%%%%%%%%

%%%%%%%%%%

%%%%%%%%%%%%%%%%%%%%%%%%%%

%<---
%%
%--->zooming 2
%Roughly speaking, this says that for most $\bd s\in\bxu$ the measure
%$(a\mapsto\bd s(a))_*\hmu$
%which is the result of
%pushing the conditional measure $\hmu$
%by $a\mapsto \bd s(a)$
%is close to $\bexu$.
The following lemma will allow us to upgrade measurability to continuity on certain compact sets of arbitrarily large measure.
%Because of~\eqref{eq:1039} these compact sets will have the property that they come from subsets of 
%$A^n$ with large $\hmu$-measure pushed to $\bxu$ with the map $a\mapsto\bd s(a)$.
During the course of the proof and in \S\ref{ssec:drift} we will use a few standard results from measure 
theory and analysis such as Lusin's theorem and the martingale
convergence theorem. A suitable reference for all of these results is~\cite{MR2267655}.
\begin{lemma}\label{lem:good U}
    Let $E\subset \bxu$ be a measurable subset such that $\bexu(E)=1$. Then, for any $0<\del<1$
    there exist compact subsets $K'\subset K\subset E$
    such that:
    \begin{enumerate}[(1)]
        %\item\label{req:1} $\be^X(K)>1-\del^2$,
        %\item\label{req:2} $\be^X(K')>1-\del$,
        \item\label{req:3} The map $\bd s\mapsto \lwm{\bd s}$ is defined and continuous on $K$.
        \item\label{req:4} The map $\bd s=(b,x,z)\mapsto \xi(b)$ is defined and continuous on $K$ (see \S\ref{ssec:boundary maps}).
            %\newcounter{counter1}
            %\setcounter{counter1}{\value{enumi}}
            %\item All the statements discussed so far which hold $\be^X$ almost surely  hold for any $\hp\in K$.
            %\end{enumerate}
            %and a compact subset $L\subset K\times U$ such that
            %\begin{enumerate}
            %\setcounter{enumi}{\value{counter1}}
        \item\label{req:2}
            %$(\be^X\otimes \lam|_U) (L)>1-2\del$,
            The volume $\bexu(K')>1-2\del$.
        \item\label{req:1}There exists $n_0>0$ such that for all $\bd s\in K'$ and $n>n_0$ one has
            \begin{equation}\label{eq:2230}
                \hmu(\set{a\in A^n : \bd s(a)\in K})>1-\del.
            \end{equation}
    \end{enumerate}
\end{lemma}
\begin{proof}
    Let $E\subset \bxu$ be a set of full $\bxu$-measure and $0<\del<1$ be given.
    %Using Lemma~\ref{lem:measureablity} we see that both of the maps from~\eqref{req:3} and~\eqref{req:4} are measurable.
    We may assume that $E$ is contained in the domain of the LWM-map and the projection to $B$ of $E$ is contained in the
    full measure set on which $\xi$ is defined and measurable.
    Hence by Lusin's theorem we may pick a compact
    set $K\subset E$ such that requirements~\eqref{req:3} and~\eqref{req:4} hold and such that
    $\bexu(K)>1-\del^2$.
    %Define
    %$\hat{K}=K\times U$ so that $\hbeu(\hat{K})>1-\del^2$.
    %Let $\vphi \defi \bE(\mb{1}_{\hK }| \cQ^U_\infty)$
    %(\hhp) = \int_{A^n}\mb{1}_{\hK }(\hhp(a_1^n)) d \mu^{\otimes n}_{b,Uz^{-1}}(a_1^n)
    %so that $\vphi$ is a measurable function on $B^{X,U}$.
    Since $0\le \bE(\mb{1}_{K }| \cQ^U_\infty)\le 1$ and $$\int_{\bxu} \bE(\mb{1}_{K }| \cQ^U_\infty) \mathrm{d}\be^{X,U}>1-\del^2,$$  by Chebyshev's inequality
    there exists a compact $L'\subset B^{X,U}$ such that $\vphi|_{L'}>1-\del$ and $\be^{X,U}(L')>1-\del$.
    Let $L = L'\cap K$ so that
    $\be^{X,U}(L)>1-\del-\del^2>1-2\del$.

    Since the conditional expectations
    $ \bE(\mb{1}_{K} | \cQ^U_n)$
    are a reversed martingale, by the martingale convergence theorem we have
    $$\lim_{n\to\infty}\bE(\mb{1}_{K} | \cQ^U_n)=\bE(\mb{1}_{K} | \cQ^U_{\infty})\quad \bexu\textrm{-almost surely}.$$
    Using Egoroff's theorem we can assume that on $L$ the convergence is uniform.
    In particular, there
    exists $n_0>0$ such that for all $n>n_0$ one has $\vphi_n|_{L}>1-\del$.
    From~\eqref{eq:1039} we see that
    $$\bE(\mb{1}_{K} | \cQ^U_n)(\bd s)= \int_{A^n}\mb{1}_{K }(\bd s(a)) \dv\hmu a.$$
    Hence, for all $\bd s\in L$ and $n>n_0$ one has
    \begin{equation}\label{eq:08061}
        \hmu(\set{a\in A^n: \bd s(a)\in {K}}) > 1-\del.
    \end{equation}
    Hence the requirements of the lemma are satisfied with the sets $L\subset K\subset E$.
    %From the bound $\be^{X,U}(L)>1-2\del$ and the fact that $U_\star$ is of full $\lam$-measure in $U$, we deduce that there exists
    %$z_\star\in U_\star$ for which
    %the slice $K' = \set{\hp\in K:(\hp,z_\star)\in L}$ satisfies $\be^X(K')>1-2\del$, which is requirement~\eqref{req:2}.
    %Finally, the fact that $\hK=K\times U$ combined with~\eqref{eq:1002} and the fact that for $a_1^n$ in the support
    %of $\mu^{\otimes n}_{b, Uz_\star^{-1}}$ we have that $\mbe_n(a_1^nS^nb) \mbe_n(b)^{-1}z_\star\in U$, we see that~\eqref{eq:08061}
    %impies~\eqref{eq:2230}.
\end{proof}
%<---
%%
%--->constructing disps
\subsection{Constructing the displacements}\label{ssec:displacements}
%For any $b\in B$ let $$\Pi_b:\gog_b\to\gog_b/\gor_b$$ be the quotient map.
We set up some notational conventions which will be used in the drift argument in the next subsection.
For $\eta\in H/P$ we consider the quotient $\gog_\eta/\gor_\eta$ of the Lie algebra 
$\gog_\eta$ of $G_\eta$ by the Lie algebra $\gor_\eta$ of the 
solvable radical $R_\eta<G_\eta$. The exponential map $\exp:\gog_\eta/\gor_\eta\to G_\eta/R_\eta$ is well defined and since
$R_\eta$ acts trivially on the plane $p_\eta$, the quotient $G_\eta/R_\eta$ acts (transitively) on the fibre $\pi^{-1}(p_\eta)$.
In particular it makes sense to write for $v\in \gog_\eta/\gor_\eta$ and $x\in \pi^{-1}(p_\eta)$, $\exp(v)x$ and in fact,  on 
letting $v$ vary in a basis of neighbourhoods of $0$ in $\gog_\eta/\gor_\eta$ one obtains a basis of neighbourhoods of 
$x$ in $\pi^{-1}(p_\eta)$. If $y=\exp(v)x$ we refer to $v$ as a \textit{displacement} between $y$ and $x$.

For $g\in G$, the adjoint action of $g$ on $\gog$ induces an isomorphism from $\gog_\eta/\gor_\eta$
to $\gog_{g\eta}/\gor_{g\eta}$. Thus, for $v\in \gog_\eta/\gor_\eta$ we let $gv$ denote the corresponding image
in $\gor_{g\eta}/\gor_{g\eta}$. If $x,y\in \pi^{-1}(p_\eta)$ and $v\in\gog_\eta/\gor_\eta$ is a displacement
between $x$ and $y$, then for any $g\in G$ we have that $gv\in \gog_{g\eta}/\gor_{g\eta}$ is a displacement 
between $gx,gy\in \pi^{-1}(p_{g\eta})$. In particular, for $\be$-almost every $b\in B$ (where $\xi$ is defined 
and equivariant) and for all $x,y\in\pi^{-1}(p_b)$, $v\in\gog_b/\gor_b$ and $n\in\bN$ one has that: 
\begin{equation}\label{eq:0930a}
    %\nonumber &\forall v\in \gog_b/\gor_b, x,y\in\pi^{-1}(p_b), n\in\bN,\\
    \textrm{if}\ \exp(v)x=y\ 
    \textrm{then}\  
    \exp((b_1^n)^{-1}v) (b_1^n)^{-1}x = (b_1^n)^{-1}y.
\end{equation}
\begin{remark}
    Note that as $\gor_\eta$ is not an ideal in $\gog$ these notions cannot be extended to define displacements
    in $\gog/\gor_\eta$ between nearby points $x,y\in X$ without the assumption that they both lie in the same 
    plane. In~\S\ref{sec:non atomicity} we will need this more general notion of displacement and develop the necessary notation and terminology.
\end{remark}
We equip $\gog_\eta/\gor_\eta$ with the quotient norm which is induced by our pre-fixed inner product 
on $\gog$. 
We choose a metric $\dx$ on $X$ in such a way that if $v\in \gog_\eta/\gor_\eta$ and $\norm v \le \eps$,
then for any $x\in \pi^{-1}(p_\eta)$ one has that $\dx(x,\exp(v)x)\le \eps$. See~\S\ref{ssec:metric} for details regarding an explicit choice 
of such a metric.
%%%%%%%%%%%%%%%%%%%%%%
%%%%%%%%%%%%%%
%%%%%%%%%%%
\begin{comment}
    so that on choosing a representative
    $\tilde{v}\in \gog_\eta$ of $v\in \gog_\eta/\gor_\eta$ and some non-zero $u\in \wedge^3\gor_\eta$, 
    we have 
    $$\norm{v} = \frac{\norm{\tilde{v}\wedge u}}{\norm{u}}.$$
    Note that the right hand side of this equation does not depend on the choice of $\tilde{v}$ and $u$. 
    Nevertheless it will be convenient for us to choose a convenient representative. To that end we denote
    by $\gom_\eta\defi \gor_\eta^\perp$ the orthogonal complement of $\gor_\eta$ in $\gog$. For $v\in \gog_\eta/\gor_\eta$
    we let $\tilde{v}\in\gom_\eta$ be the unique representative of $v$. Note that $g\tilde{v} \ne \wt{gv}$ in general 
    but since their difference lies in $\gor_{g\eta}$, 
    on choosing $u\in \wedge^3 \gor_\eta$, $g(\tilde{v}\wedge u) = \wt{gv}\wedge gu$ and so we will
    be able to keep track of the norm $\norm{gv}$ by analysing the ratio
    $\norm{g(\tilde{v}\wedge u)}/\norm{gu}$. 
\end{comment}
%%%%%%%%%%%%
%%%%%%%%%%%%
%%%%%%%%%%%

We use the assumption that the $\nu_b$'s are non-atomic $\be$-almost surely to build a sequence of displacements that will become input for the drift argument as reflected in the following lemma.
\begin{lemma}\label{lem:tnunbdd}
    Let $F\subset \bxz$ be a set of positive $\bexz$-measure and suppose that for $\be$-almost every $b\in B$ the measures $\nu_b$ are non-atomic. Then,
    for $\bexz$-almost every $(b,x,z)\in F$ there exists a sequence $\set{v_i}_{i\in\bN}\subset \gog_b/\gor_b\mz$ tending to 
    $0$ such that that for all
    $i\in\bN$ one has $(b,\exp(v_i)x,z)\in F$ and $\limsup_{n\to\infty}t_n(b,v_i)=\infty$ where for $v\in \gog_b/\gor_b$
    \begin{equation}
        t_n(b,v)\defi \chi_{\gol_0/\gor_0}(\mbe_n(b)) \norm{(b_1^n)^{-1} v}.
        %&=\chi_{\gol_0/\gor_0}(\mbe_n(b)) 
        %\frac{\norm{ {(b_1^n)^{-1}}(\tilde{v}_i\wedge u_b)}}{\norm{ {(b_1^n)^{-1}} u_b}}
        %\textrm{for some}\ u_b\in\wedge^3\gor_b\smallsetminus\set 0.
    \end{equation}
\end{lemma}
\begin{proof}
    %By neglecting a $\be^X$-null set from $F$ we may assume that for any $(b,x)\in F$ we have that $b\in B_0$,
    %$\supp\nu_b\subseteq\pi^{-1}(p_b)$ and $x\in\pi^{-1}(p_b)$.
    % Note also that the sets $\exp(B_\eps^{\gom_b})x$ describe
    %a basis of neighbourhoods of $x$ in $\pi^{-1}(p_b)$ as $\eps>0$ varies.
    We fix a measurable set $F\subset\bxz$ such that $\bexz(F)>0$.
    %By neglecting a $\bexz$-null subset
    Since the statement we are trying to prove is an almost sure statement, it is safe to neglect $\bexz$-null sets.
    It follows from Proposition~\ref{prop:firstprop} that we may assume 
    $\supp\nu_b\subseteq\pi^{-1}(p_b)$ and $x\in\pi^{-1}(p_b)$ for all $(b,x,z)\in F$.
    %It follows from the definition of $\bexz$ and the assumption that $\bexz(F)>0$ that $\be$-almost surely
    %$\nu_b(F_{b,z})>0$ where $F_{b,z}\defi\set{x\in X:(b,x,z)\in F}$ \usadd{the last sentence is not true, maybe you meant
    % for $\be^{X,Z}$-almost every $\mb{s}\in F$?}.
    Furthermore, using the definition of $\bexz$, we may assume that for all $(b,x,z)\in F$, $x$ belongs to the support
    of $\nu_b$. In other words, if
    for $i\in\bN$ we let $\cN_i^b$ denote a basis of neighbourhoods of 0 in $\gog_b /\gor_b$ then
    for all $(b,x,z)\in F$ and $i\in \bN$,
    %Since $\be^X=\int_B\del_b\otimes \nu_bd\be(b)$, $\be^X$-almost every $(b,x)\in F$ belongs to the support
    %of the restriction of $\del_b\otimes \nu_b$ to $F$. In other words, for $\be^X$-almost every $(b,x)\in F$, for any $\eps>0$
\begin{equation}\label{eq:disp1}\nu_b(\set{\exp(v) x:v\in\cN_i^b}) >0.\end{equation}
    For $b\in B$ let $$\gos_b \defi \set{v\in \gog_b/\gor_b : \limsup\nolimits_{n\to\infty}t_n(b,v)<\infty }.$$
    In light of~\eqref{eq:disp1} and the definition of the measure $\bexz$ in order to prove the lemma,
    it is enough to establish
    \begin{equation}\label{eq:disp2}
        \nu_b(\exp(\gos_b) x) = 0\ \textrm{for}\ \be^X\textrm{-almost every}\ (b,x)\in B^X.
    \end{equation}

    Let $\dx$ denote a distance function on $X$ as discussed before the lemma. For
    $(b,x)\in B^X$ let $$W_b(x) \defi \set{y\in X : \lim\nolimits_{n\to\infty}\dx((b_1^n)^{-1} y, (b_1^n)^{-1} x)\to 0}.$$
    It is shown in~\cite[Proposition 6.18]{BQJams} that $\be^X$-almost surely $\nu_b(W_b(x)\smallsetminus\set{x}) = 0$. Due
    to our non-atomicity assumption we deduce that $\be^X$-almost surely $\nu_b(W_b(x))=0$. Hence we can verify~\eqref{eq:disp2}
    by showing that
    \begin{equation}\label{eq:1050} 
        \exp(\gos_b)x\subset W_b(x) \textrm{ for $\be^X$-almost every $(b,x)\in B^X$}.
    \end{equation}
    To this end, let $(b,x)\in B^X$ and
    $v\in \gos_b$ so that $t_n(b,v)$ is bounded and let $y=\exp(v) x$. We will finish
    by showing that if $(b,x)$ is outside a $\be^X$-null set, then $y\in W_b(x)$.
    By part~\eqref{asconvergence} of Theorem~\ref{thm:poslyap} and~\eqref{eq:posweight} one has
    $$\lim_{n\to\infty}\om_{\gol_0/\gor_0}(\on{L}_{n}(b)/n)=\om_{\gol_0/\gor_0}(\sig_\mu)>0\quad \be\textrm{-almost surely}$$
    and hence
    %$(\chi_{\gol_0/\gor_0}(\mbe_n(b)))^{1/n}\to
    %\exp(\om_{\gol_0/\gor_0}(\sig_\mu))>1$ $\be$-almost surely which in turn implies  that
    \begin{equation}\label{eq:1427}
        \limsup_{n\to\infty}\chi_{\gol_0/\gor_0}(\mbe_n(b))=\infty\quad \be\textrm{-almost surely}.
    \end{equation}
    Therefore, once $(b,x)$ is such that~\eqref{eq:1427} holds then taking into account the definition
    of $t_n(b,v)$ and its boundedness we conclude that $\lim_{n\to\infty}\norm{(b_1^n)^{-1}v}= 0$. In particular,
    on denoting $x_n = (b_1^n)^{-1}x$ we get that  
    \begin{align*}
        \dx((b_1^n)^{-1}x, (b_1^n)^{-1}y) &= \dx(x_n, (b_1^n)^{-1}\exp(v)x) \\
                                          &= 
        \dx(x_n, \exp((b_1^n)^{-1}v)x_n).
    \end{align*}
    It follows that 
    $$\lim_{n\to\infty}\dx((b_1^n)^{-1}x, (b_1^n)^{-1}y)=0$$
    where we use~\eqref{eq:0930a} which holds $\be$-almost surely.
    This shows that $y\in W_b(x)$ and finishes the proof of the lemma.
\end{proof}
%<---
%%
%--->exp drift 1

\subsection{The exponential drift - Proof of Theorem~\ref{thm:full}\eqref{main1}}\label{ssec:drift}
We now prove Theorem~\ref{thm:full}\eqref{main1} which we restate for convenience.
\begin{theorem}
    \label{thm:the drift-3}
    Let $\mu\in\cP(G)$ be a compactly supported measure and suppose we are either in \ref{case1} or \ref{case2}.
    Let $\nu\in \cP_\mu(X)$ be an ergodic $\mu$-stationary measure on $X$ and assume that for $\be$-almost every $b\in B$ the limit measures
    $\nu_b$ are non-atomic, then $\nu$ is the natural lift of the Furstenberg measure of $\mu$ on $\Xbar$. 
\end{theorem}

\begin{proof}
    Let $U$ be as in~\eqref{eq:1436a}.
    By Proposition~\ref{prop:unipinv} and Proposition~\ref{prop:unipinv2} it is enough to establish that $\bexu$ is invariant under the horocyclic flow $\Phi_{\gou_0}$. 
    By Theorem~\ref{thm:ashaar} we are reduced to establishing that for $\bexu$-almost every $\mb{s}\in B^{X,U}$ the LWM, $(\bexu)_{\mb{s}}^\Phi$ is equal to the Haar measure on $\gou_0$. 
    Said differently, we are reduced to establishing the equality $\Stab_{\gou_0}(\lwm{\mb{s}}) = \gou_0,$ $\bexu$-almost surely. By Proposition~\ref{prop:equalstab} it is enough to establish
    the following claim: 
    \begin{claim}\label{the claim}
        The equality $\Stab_{\gou_0}\pa{\br{\lwm{\mb{s}}}} = \gou_0$ holds $\bexu$-almost surely. 
    \end{claim}
    The rest of the proof is devoted to proving this claim.
    There exists a measurable $S$-invariant set of 
    full measure $B_0\subset B$ such that for all $b\in B_0$, the boundary map $\xi$ is defined and equivariant at $b$ and 
    Lemma~\ref{lem:expansion of leafs-1} is applicable to $b$ with respect to the exterior powers 
    of the adjoint representation of $H$ on $\gog$.

    Let $E\subseteq B_0\times X\times U$ be a measurable subset of full $\bexu$-measure such that
    the LWM-map is defined on $E$, Lemma~\ref{lem:measureablity} is applicable for any point in $E$ 
    in the sense that for all $\hs\in E$ and $n\in \bN$, 
    \begin{equation}\label{eq:1455a}
        \lwm{\hs} = \lwm{\wh{T}^n(\hs)}.
    \end{equation}
    Additionally, using Proposition~\ref{prop:firstprop}, we assume that 
    for all $(b,x,z)\in E$ one has
    $\nu_b(\pi^{-1}(p_b))=1$ and $x\in\pi^{-1}(p_b)$. 
    For $\mb{s}=(b,x,z)\in E$ and $v\in \gog_b/\gor_b$
    we denote 
\begin{equation}\label{eq:weirdnotation}\exp(v)\mb{s} \defi (b,\exp(v)x,z).\end{equation}

    Let $0<\del<1/10$ be arbitrarily small and let $K'\subset K\subset E$ be compact subsets 
    be as guaranteed 
    by Lemma~\ref{lem:good U}. 
    \begin{definition*}
        Given a point $\mb{s} = (b,x,z)\in K'$  we say that a sequence $\set{ v_i}_{i\in\bN}\subset \gog_b/\gor_b$ of non-zero vectors
        converging to $0$ is \textit{unstable} for $\hp$ if  
        \begin{equation}\label{eq:0742}
            \hp_i \defi \exp(v_i) \hp\in K'\qfa i\in\bN
        \end{equation}
        and for any fixed $i\in\bN$ the sequence  
        \begin{equation}\label{eq:incseq}
            t_n(b,v_i)\defi \chi_{\gol_0/\gor_0}(\mbe_n(b))\norm{(b_1^n)^{-1}v_i}
        \end{equation}
        in the variable $n$ is unbounded. 
        Although we do not record in
        this terminology the set $K'$, it should cause no confusion because $K'$ will remain fixed until the last step of the proof.
    \end{definition*}
    By Lemma~\ref{lem:tnunbdd}, $\bexu$-almost every $\hp\in K'$ has an unstable sequence. We note that 
    this is the part of the proof where the non-atomicity of the $\nu_b$'s is being used.

    Let $\hp\in K'$ and $\set{v_i}_{i\in\bN}$ be an unstable sequence for $\hp$. For all $i$, $n$ and $a\in A^n$ 
    such that $aS^nb\in B_0$
    the relation~\eqref{eq:0742} between $\hp$ and $\hp_i$ propagates to a similar relation between 
    $\hp(a)$ and $\hp_i(a)$. Namely, using the notations from~\eqref{eq:0855} and~\eqref{eq:weirdnotation}
    \begin{equation}\label{eq:2133}
        \hp_i(a) = \exp(a_1^n(b_1^n)^{-1} v_i) \hp(a).
    \end{equation}
    The assumption that $aS^nb\in B_0$ is used in order for~\eqref{eq:0930a} to apply.

    The proof of Claim~\ref{the claim} 
    relies on showing that for arbitrarily large $i\in\bN$ one can choose carefully $n_i\in\bN$ and  $a_i\in A^{n_i}$
    in such a way that equation~\eqref{eq:2133} limits to an equation giving rise to the fact that $\br{\lwm{\mb{s}}}$ is invariant under 
    an arbitrarily small element of $\gou_0$. It is quite long and so we try to break it into steps and introduce 
    auxiliary notation and terminology to ease the complications.
    \begin{definition*}
        We say that a point $\hp\in K'$ satisfies \textit{hypothesis $\on{ED}$} if there exists 
        an unstable sequence $\set{v_i}_{i\in\bN}$ for $\hp$ such that for all $\eps>0$,
        for all $i\in\bN$ 
        %sequence of $i$'s going to $\infty$
        there exists choices $n_i\in\bN$ and $a_i \in A^{n_i}_{\hs,U}$
        such that $n_i\to\infty$ as $i\to\infty$ and the following hold:
        \begin{enumerate}[\bfseries{ED}1:]
            \item\label{claim:p3} For all $i\in\bN$ one has that the points $\hp(a_i), \hp_i(a_i)\in K$.
            \item\label{claim:p1} $\norm{ (a_i)_1^{n_i}(b_1^{n_i})^{-1} v_i}\asymp \eps$.
            \item\label{claim:p2} $\lim_{i\to\infty}\on{d}_{\bP\wedge^4\gog}((a_i)_1^{n_i}(b_1^{n_i})^{-1}( \bR \tilde{v}_i\wedge(\wedge^3 \gor_b)), 
                (\wedge^4\gog)_{a_iS^{n_i}b}[\om_{\gol_0}]
                %\wedge^4\gol_{(a_i)_1^{n_i}S^{n_i} b}
                )
                =0$, where $\tilde{v}_i\in\gog_b$ is a representative of $v_i$.
        \end{enumerate}
    \end{definition*}
    We note that property {\bfseries ED\ref{claim:p3}} implies 
    that the $B$-coordinate of the points
    $\hp(a_i)$ and  $\hp_i(a_i)$, which is $a_iS^{n_i}b$, belongs to $B_0$ by our assumption on $E$. 
    As explained above, this implies that~\eqref{eq:2133} holds and moreover,
    from the definition
    of $B_0$ we have the equality 
    \begin{equation}\label{eq:2135}
        (a_i)_1^{n_i}(b_1^{n_i})^{-1}\xi_b = \xi_{a_iS^{n_i}b}.
    \end{equation}
    We complete the proof of Claim \ref{the claim} in two steps by proving: 
    \begin{enumerate}[(Step 1)]
        \item\label{step2} $\bexu$-almost every $\hp\in K'$ satisfies hypothesis ED. 
        \item\label{step1} If $\hp\in K'$ satisfies hypothesis ED then $\Stab_{\gou_0}\br{\lwm{\mb{s}}}=\gou_0.$
    \end{enumerate}
    Indeed, by part~\eqref{req:2} of Lemma~\ref{lem:good U}, $\bexu(K')\ge 1-2\del,$ and since $\del$ is arbitrary the claim follows.

    \vspace{3mm}
    \noindent  \tb{Proof of Step~\ref{step2}}.
    As mentioned before, Lemma~\ref{lem:tnunbdd} implies that for $\bexu$-almost every $\hp\in K'$ there exists an
    unstable sequence $\set{v_i}_{i\in\bN}$. Therefore, there is no problem fixing $\hp\in\ K'$ and $\set{v_i}_{i\in\bN}$ a corresponding 
    unstable sequence. Let $\eps>0$.

    Fix $i\in\bN$ and consider the sequence $t_n=t_n(b,v_i)$ from~\eqref{eq:incseq}.
    Note that since the support of $\mu$ is compact, the ratios $t_{n+1}/t_n$ are bounded by a constant
    depending on $\mu$. By the definition of the instability of $\set{v_i}_{i\in\bN}$ for $\hp$, the sequence
    $t_n$ is unbounded and since $t_1$ is arbitrarily small for all large $i$, 
    we conclude that for all large $i$ the number $n_i\defi\min\set{n: t_n>\eps}$ is well defined and in that case 
    \begin{equation}\label{eq:0040}
        t_{n_i}\asymp \eps.
    \end{equation} 
    Note  that since $v_i\to 0$, we must have that $n_i\to \infty$ as $i\to \infty$. 
    The existence of $a_i\in A^{n_i}$ 
    for which properties {\bfseries ED\ref{claim:p3}}-{\bfseries ED\ref{claim:p2}} will hold
    will be established by probabilistic means using the conditional probability measure
    $\mu^{\otimes n_i}_{\hs, U}$ discussed in~\S\ref{ssec:loa}.

    First we demonstrate that property {\bfseries ED\ref{claim:p3}} holds for a set of large $\mu^{\otimes n_i}_{\hs, U}$-measure.  
    Since $n_i\to \infty$ 
    and both of $\hs$ and $\hs_i$ are elements of $K'$, 
    by Lemma~\ref{lem:good U}, which we used to obtain $K'$ and $K$, we have that: 
    For $\bexu$-almost every $\hs\in K'$,
    \begin{equation}\label{eq:2101}\tag{$\star$}
        %\boxed{
        \mu^{\otimes n_i}_{\hs,U}(\set{a\in A^{n_i}:\hp(a), \hp_i(a)\in K})>1-2\del\qfa i\gg 1.
        %}
    \end{equation}

    We turn to property {\bfseries ED\ref{claim:p1}}. 
    Observe that for $c\in B_0$ and $v\in\gog_c/\gor_c$, 
    \begin{equation}\label{eq:1045a}
        \norm{v} = \frac{\norm{\tilde{v}\wedge u_c}}{\norm{u_c}}
    \end{equation} 
    where we will use $\tilde{v}\in\gog_c$ to denote a choice of a representative for $v$ and $u_c$ 
    to denote a non-zero element of $\wedge^3\gor_c$ (note that the quantity in~\eqref{eq:1045a} does not depend
    on our choices). This will allow us to obtain {\bfseries ED\ref{claim:p1}} 
    by considering the representations of $H$ on $\wedge^4\gog$ and $\wedge^3\gog$.

    For $i\in\bN$ we use the notation
    \begin{equation}
        \label{eq:vwedge} \mb{v}_i \defi (b_1^{n_i})^{-1}( \tilde{v}_i\wedge u_b)\quad\textrm{and}\quad
        \mb{v}_i' \defi (b_1^{n_i})^{-1}u_b. 
    \end{equation}
    For $a\in A^{n_i}$ our goal is to understand the norm in {\bfseries ED\ref{claim:p1}}. 
    We will compare the quantities
    \begin{equation}\label{eq:incseq1}
        \norm{a_1^{n_i}(b_1^{n_i})^{-1}v_i}=\frac{\norm{a_1^{n_i}\mb{v}_i}}{\norm{a_1^{n_i}\mb{v}_i'}}\quad \textrm{ and }\quad 
        t_{n_i}(b,v_i) =  \frac{\chi_{\gol_0}(\mbe_n(b))\norm{\mb{v}_i}}
        {\chi_{\gor_0}(\mbe_n(b))\norm{\mb{v}_i'}}
    \end{equation}
    and show that they are of the same  order of magnitude.
    %As a result of~\eqref{eq:0040} we will establish that $\norm{a_1^{n_i}(b_1^{n_i})^{-1}v_i}\asymp \eps$ as needed. 
    We start by relating the numerators of the ratios in~\eqref{eq:incseq1} 
    and then consider the corresponding denominators.

    We apply Corollary~\ref{law of norms} to the representation of $H$ on $V=\wedge^4 \gog$ and for 
    the weight $\om_{\gol_0}\in \cH_{\goz}(V)$ and the vector $\tau_{\om_{\gol_0}}(\mb{v}_i)\in V[\om_{\gol_0}]$ and conclude 
    that for $\bexu$-almost every $\hs\in K'$,
    \begin{equation}\label{eq:0653}
        \mu^{\otimes n_i}_{\hs,U}(\set{a\in A^{n_i}: \norm{a_1^{n_i}\tau_{\om_{\gol_0}}(\mb{v}_i)}\asymp 
        \chi_{\gol_0}(\mbe_{n_i}(b)) \norm{\tau_{\om_{\gol_0}}(\mb{v}_i)}})>1-\del\qfa i\gg 1.
    \end{equation}
    We wish to replace in \eqref{eq:0653} the term 
    $\norm{\tau_{\om_{\gol_0}}(\mb{v}_i)}$ by
    $\norm{\mb{v}_i}$ and  $\norm{a_1^{n_i}\tau_{\om_{\gol_0}}(\mb{v}_i)} = \norm{\tau_{\om_{\gol_0}}(a_1^{n_i}\mb{v}_i)} $
    by $\norm{a_1^{n_i}\mb{v}_i}$. For this we use parts~\eqref{p:06111} and~\eqref{p:0611} of  
    Lemma~\ref{lem:max weight}. In order for Lemma~\ref{lem:max weight} to be applicable
    we need that $\mb{v}_i \in \gog_{S^{n_i}b}\wedge(\wedge^3\gor_{S^{n_i}b})$ and 
    $a_1^{n_i}\mb{v}_i\in \gog_{aS^{n_i}b}\wedge (\wedge^3\gor_{aS^{n_i}b})$.
    The first containment holds since $b\in B_0$ and the relevant spaces vary equivariantly. 
    For the second containment, if we require $a$ to be also
    in the set measured in~\eqref{eq:2101} then $aS^{n_i}b\in B_0$ as well and the relevant 
    equivariance applies. This leads us to conclude from \eqref{eq:2101} and \eqref{eq:0653} 
    that for $\bexu$-almost every $\hs\in K'$,
    \begin{equation}\label{eq:0654}
        \mu^{\otimes n_i}_{\hs,U}(\set{a\in A^{n_i}: \norm{a_1^{n_i} \mb{v}_i}\asymp 
        \chi_{\gol_0}(\mbe_{n_i}(b))  \norm{\mb{v}_i}})>1-3\del\qfa \gg 1.
    \end{equation}

    Regarding the denominators in~\eqref{eq:incseq1}, we claim that for $a$'s which are measured in~\eqref{eq:2101},
    \begin{equation}\label{eq:0717}
        \norm{a_1^{n_i}\mb{v}_i'} = \chi_{\gor_0}(\mbe_{n_i}(aS^{n_i}b))\norm{\mb{v}_i'}\asymp 
        \chi_{\gor_0}(\mbe_{n_i}(b))\norm{\mb{v}_i'}\qfa i \gg 1.
    \end{equation} 
    The first equality follows from an application of Lemma~\ref{lem:expansion of leafs-1} to the vector $\mb{v}_i'= (b_1^{n_i})^{-1}u_b\in \gor_{S^{n_i}b}$ together with the observation that $aS^{n_i}b\in B_0$ which uses our assumption that $a$ belongs 
    to the set measured in~\eqref{eq:2101}.
    The approximation part in~\eqref{eq:0717} comes from the fact that 
    $a\in \supp \mu^{\otimes n_i}_{\hs,U} = A^{n_i}_{\hs,U}$ satisfies 
    $\mbe_n(aS^{n_i}b)\mbe_{n_i}^{-1}(b)\in U$.

    We thus conclude from~\eqref{eq:incseq1},~\eqref{eq:0654} and~\eqref{eq:0717} that for $\bexu$-almost every $\hs\in K'$,
    $$\mu^{\otimes n_i}_{\hs, U}(\set{a\in A^{n_i}: t_{n_i}(b,v_i)\asymp \norm{ a_1^{n_i}(b_1^{n_i})^{-1} v_i}})\ge 1-3\del\qfa i\gg 1.$$
    Taking into account \eqref{eq:0040} we see that for $\bexu$-almost every $\hs\in K'$,
    \begin{equation}\label{eq:0655}\tag{$\star\star$}
        %\boxed{
        \mu^{\otimes n_i}_{\hs,U}
        \mset{a\in A^{n_i}:\norm{a_1^{n_i}(b_1^{n_i})^{-1} v_i)}\asymp\eps }>1-3\del\qfa i\gg 1.
        %}
    \end{equation}
    Equation~\eqref{eq:0655} will take care of {\bfseries ED\ref{claim:p1}}.

    We now turn to {\bfseries ED\ref{claim:p2}}. 
    Fix $i\gg 1$ and let $k\in\bN$. We apply Lemma~\ref{law of angles 2} 
    to the representation $V = \wedge^4\gog$ with $\rho = 1/k$, 
    the weight $\om_{\gol_0}\in \cH_{\goz}(V)$,
    the vector $\tau_{\om_{\gol_0}}(\mb{v}_i)$ where $\mb{v}_i$ is as in~\eqref{eq:vwedge} 
    and the flag $\eta = \xi(S^{n_i}b)\in H/P$. 
    The statement of Lemma~\ref{law of angles 2} in this case and in particular equation~\eqref{eq:1915}, implies that: 
    For $\bexu$-almost  any $\hs\in K'$,
    \begin{align}
        \label{eq:22281} 
        {
            \mu^{\otimes n_i}_{\hs,U}\mset{a\in A^{n_i}: 
            \on{d}_{\bP V}(a_1^{n_i}\bR\tau_{\om_{\gol_0}}(\mb{v}_i), a_1^{n_i}V_{S^{n_i}b}[\om_{\gol_0}])<1/k}>1-\del\qfa i\gg 1.
        }
    \end{align}

    If we also know that if $aS^{n_i}b\in B_0$, which happens whenever $a$ is in the set measured in \eqref{eq:2101}, then 
    we have the equality $a_1^{n_i}V_{S^{n_i}b}[\om_{\gol_0}] = V_{aS^{n_i}b}[\om_{\gol_0}]$ and thus we conclude from \eqref{eq:22281} that
    for $\bexu$-almost every $\hs\in K'$,
    \begin{align}
        \label{eq:2228} 
        {
            \mu^{\otimes n_i}_{\hs,U}\mset{a\in A^{n_i}: 
            \on{d}_{\bP V}(a_1^{n_i}\bR\tau_{\om_{\gol_0}}(\mb{v}_i), V_{aS^{n_i}b}[\om_{\gol_0}])<1/k}>1-3\del\qfa i\gg 1.
        }
    \end{align}
    Next we replace in 
    \eqref{eq:2228} the vector  $\tau_{\om_{\gol_0}}(\mb{v}_i)$ by $\mb{v}_i$. 
    To justify this passage we apply 
    Lemma \ref{max weight pos angle} to the representation $V$ with $\rho = 1/k$, $\om_{\mb{m}}=\om_{\gol_0}$ and the vector 
    $\mb{v}_i$. Note that Lemma~\ref{max weight pos angle} is applicable in light of Lemma~\ref{lem:max weight}.
    The statement of Lemma~\ref{max weight pos angle}, in particular equation~\eqref{eq:1025}, implies that: 
    For $\bexu$-almost every $\hs\in K'$,
    \begin{equation}
        \label{eq:1032} 
        {
            \mu^{\otimes n_i}_{\hs,U}  \mset{a\in A^{n_i}:
            \on{d}_{\bP V}(a_1^{n_i}\bR\mb{v}_i, a_1^{n_i}\bR\tau_{\om_{\gol_0}}(\mb{v}_i))<1/k}
            >1-\del\qfa i \gg 1.
        }
    \end{equation}
    Equations~\eqref{eq:2228} and~\eqref{eq:1032} and the triangle inequality imply that: 
    For $\bexu$-almost every $\hs\in K'$, 
    for any positive integer $k$,
    \begin{equation}\label{eq:1037}\tag{$\star\star\star$}
        %\boxed{
        {
            %\scriptstyle
            \mu^{\otimes n_i}_{\hs,U}\mset{a\in A^{n_i}: 
            \on{d}_{\bP V}(a_1^{n_i}\bR \mb{v}_i, V_{aS^{n_i}b}[\om_{\gol_0}])<2/k}>1-4\del\qfa i \gg 1.
            %}
        }
    \end{equation}
    We then choose $\set{ i_k}_{k\in\bN}$ with $i_k\to\infty$, such that~\eqref{eq:1037} holds.

    To tie things up and finish this part of the proof we note that for $k\gg 1$, 
    equations~\eqref{eq:2101}, \eqref{eq:0655} and \eqref{eq:1037} hold for $i=i_k$. 
    Since $\del<1/10$, we  
    deduce that for $\bexu$-almost every $\hs\in K'$ 
    there must 
    exist $n_i\in\bN$ and $a_i\in A^{n_i}_{\hs,U}$ such that properties {\bfseries ED\ref{claim:p3}}-{\bfseries ED\ref{claim:p2}}
    are satisfied and so $\hs$ satisfies hypothesis ED. This concludes the proof of Step \ref{step2}. 

    \vspace{3mm}
    \noindent \tb{Proof of Step~\ref{step1}}. 
    Let $\hp\in K'$ satisfy hypothesis ED with respect to the unstable sequence $\set{v_i}_{i\in\bN}$ 
    and let $\eps>0$ be arbitrarily small. 
    %Consider a sequence of $i$'s going to $\infty$
    Let $\set{n_i}_{i\in\bN}$ and $a_i\in A^{n_i}_{\hs,U}$ be such that $n_i\to\infty$ as $i\to\infty$ and properties 
    {\bfseries ED\ref{claim:p3}}-{\bfseries ED\ref{claim:p2}} hold.

    By taking a subsequence if necessary and using {\bfseries ED\ref{claim:p3}} we may assume that
    $$\lim_{i\to\infty}\bd{s}(a_i)\eqqcolon \bd{r}_1\in K\quad\textrm{and}\quad\lim_{i\to\infty}\bd{s}_i(a_i)\eqqcolon \bd{r}_2\in K.$$
    Note that the $U$-coordinate of $\bd{r}_1$ and $\bd{r}_2$ are the same.
    We claim that the relation~\eqref{eq:2133} between $\hs(a_i)$ and $\hs_i(a_i)$ limits to the fact that
    \begin{equation}\label{eq:1536a}
        \textrm{there exists}\ w\in\gou_0\ \textrm{such that}\ \norm{w}\asymp\eps\ \textrm{and}\ \Phi_w(\mb{r}_1) = \mb{r}_2.
        %\textrm{ for some $w\in \gou_0$ of norm $\asymp \eps$.}
    \end{equation}

    We prove~\eqref{eq:1536a}: 
    The $B$-coordinate of $\hs(a_i)$ and $\hs_i(a_i)$ equals to $a_iS^{n_i}b$ and converges to the 
    $B$-coordinate of $\mb{r}_1$ and $\mb{r}_2$ which we denote by $a\in B_0$ (note that this time $a$ is an infinite sequence).
    Let us denote for $c\in B_0$ by $\gom_c$ the orthogonal complement of 
    $\gor_c$ in $\gog$ and by $\Pi_c:\gog_c\to \gom_c$ the orthogonal projection. 
    Recall that $\tilde{v}_i\in \gog_b$ denotes a representative of $v_i$. With this notation
    the relation~\eqref{eq:2133} may be rewritten as
    \begin{equation}\label{eq:2133b}
        \hs_i(a_i) = \exp(\Pi_{a_iS^{n_i}b}((a_i)_1^{n_i}(b_1^{n_i})^{-1}\tilde{v}_i)) \hs(a_i).
    \end{equation}
    Property {\bfseries ED\ref{claim:p1}} says that 
    $$\norm{ \Pi_{a_iS^{n_i}b}((a_i)_1^{n_i}(b_1^{n_i})^{-1}\tilde{v}_i)} = \norm{(a_i)_1^{n_i}(b_1^{n_i})^{-1}v_i} \asymp \eps.$$
    Note that $\Pi_{a_iS^nb}( (a_i)_1^{n_i}(b_1^{n_i})^{-1} \tilde{v}_i)\in \gom_{a_iS^{n_i}b}$ 
    and that projectively $\gom_{a_iS^{n_i}b}\to \gom_a$ because 
    of the
    continuity of $\hp'\mapsto \gom_{b'}$ on $K$, which follows from the continuity of $\hp'\mapsto \xi_{b'}$ on $K$ guaranteed 
    by Lemma~\ref{lem:good U}.
    Thus, after taking a subsequence if necessary we get  
    $$\lim_{i\to\infty}\Pi_{a_iS^{n_i}b}( (a_i)_1^{n_i}(b_1^{n_i})^{-1} \tilde{v}_i)= \tilde{v}\in \gom_a\quad \textrm{where}\ \norm{\tilde{v}}\asymp\eps.$$ 
    Equation~\eqref{eq:2133b} thus limits to the fact that 
    \begin{equation}\label{eq:1740a}
        \mb{r}_2 = \exp(\tilde{v}) \mb{r}_1.
    \end{equation}
    In fact, due to {\bfseries ED\ref{claim:p2}}, the aforementioned continuity and \eqref{eq:2135}
    we have that 
    \begin{align*}
        \bR \tilde{v}\wedge \gor_a &=\lim_{i\to\infty} \bR \Pi_{a_iS^{n_i}b}((a_i)_1^{n_i}(b_1^{n_i})^{-1}\tilde{v}_i) \wedge \gor_{a_iS^{n_i}b}\\
                                   &= \lim_{i\to\infty} \bR(a_i)_1^{n_i}(b_1^{n_i})^{-1}\tilde{v}_i \wedge \gor_{a_iS^nb}\\
                                   &=\lim_{i\to\infty} (a_i)_1^{n_i}(b_1^{n_i})^{-1}(\tilde{v}_i\wedge \gor_b) \in  \lim_{i\to\infty} (\wedge^4\gog)_{a_iS^nb}[\oml] = (\wedge^4\gog)_a[\oml].
    \end{align*} 
    By part~\eqref{p:06101} of Lemma~\ref{lem:max weight} we deduce that  
    $\tilde{v}\in \gol_a$. Since the map $\gom_a\cap \gol_a\to \gou_a = \gol_a/\gor_a$ is an isometry and since the image of $s$ in $H/N$ is compact 
    the map $\Ad_{s(\xi(a))}:\gou_0\to \gou_a$ is an isomorphism of bounded norm. 
    It follows that if we denote by $w\in\gou_0$ the image of $\tilde{v}$ then $\norm{w}\asymp \eps$ and 
    by the definition of the horocyclic flow given in~\eqref{eq:horoflow}, equation~\eqref{eq:1740a} transforms into
    \eqref{eq:1536a}.

    After establishing the alignment \eqref{eq:1536a} we arrive at the endgame. 
    By~\eqref{eq:1455a} for all $i\in\bN$,
    $$\lwm{\bd{s}(a_i)}=\lwm{\bd{s}}\quad\textrm{and}\quad \lwm{\bd{s}_i(a_i)}=\lwm{\bd{s}_i}$$
    because $\wh{T}^{n_i}(\hs) = \wh{T}^{n_i}(\hs(a_i))$ and $\wh{T}^{n_i}(\hs_i) = \wh{T}^{n_i}(\hs_i(a_i))$.
    Since the LWM-map is continuous on $K$ as it is the output of Lemma~\ref{lem:good U}, 
    we can take limits in the above and get that
    \begin{equation}\label{eq:limits}
        \lwm{\bd{s}}=\lwm{\bd{r}_1}=\lwm{\bd{r}_2}.
    \end{equation}
    Property \tb{P\ref{pC}} of the LWM-map and equations~\eqref{eq:limits} and~\eqref{eq:1536a} imply 
    $w\in\Stab_{\gou_0}(\br{\lwm{\hs}})$. Since the latter is a closed subgroup of $\gou_0\simeq \bR$ and $\eps$ is arbitrarily small
    we deduce that $\Stab_{\gou_0}(\br{\lwm{\hs}}) = \gou_0$ which concludes the proof of Step \ref{step1} and by that the proof of Claim~\ref{the claim}.        
\end{proof}
%%%%%%%%%%%%%%%%%%%%%%%%%%%%%%%

\section{Proof of Theorem~\ref{thm:full}\eqref{main3}}\label{sec:cover case}
\begin{proof}[Proof of Theorem~\ref{thm:full}\eqref{main3}]

    %We start by analysing what the assumptions of the theorem says about the limit measures $\nu_b$.
    Assume that we are in~\ref{case2} and that $\nu\in\spm\mu X$ is $\mu$-ergodic and
    not the natural lift.
    Then, according to Theorem~\ref{thm:full}\eqref{main1}
    $$\be (\set{b\in B : \nu_b \textrm{ has atoms}})>0.$$
    The equivariance of the $\nu_b$'s  and the ergodicity of the shift map imply that the above set has measure 1.
    Similarly, if $w(b)$ denotes the maximal weight of an atom of $\nu_b$ then the equivariance implies that $w= w(b)$ is constant
    $\be$-almost surely.
    The same equivariance implies that $\set{(b,x)\in B^X: \nu_b(\set{x}) = w}$ is $T$-invariant and since it is of positive
    $\be^X$-measure, it must be of measure 1 by ergodicity of $T$.
    That is to say, for $\be$-almost every $b\in B$ the limit measure
    $\nu_b$ is purely atomic and gives the same mass $w$ to each of its atoms.
    Since $\nu_b$ is a probability measure we deduce that
    there exists $k\in\bN$ such that $w= 1/k$ and $\nu_b$ has exactly $k$ atoms.
    By Proposition~\ref{prop:firstprop} we also know that $\be$-almost surely
    $\nu_b\in\cP(\pi^{-1}(p_b))$.

    Under the assumption that $\Ga$ is discrete and Zariski dense in $\SO(Q)(\bR)$ we have
    by~\cite[Theorem 2.21]{FurmanHandbook} (see also~\cite{LedrappierIMJ, KaimAnnals, KaimDokl})
    that the Furstenberg measure $\bar{\nu}_{\Xbar}$ on $\Xbar$ is the Poisson boundary of $(\Ga,\mu)$.
    Moreover, if $\mu$ is absolutely continuous with respect to the Haar measure on $\SO(Q)(\bR)$ and contains the identity in the interior of 
    its support then the same conclusion follows from~\cite[Theorem 2.17]{FurmanHandbook} (see also ~\cite[Theorem 5.3]{MR0146298}).
    By combining~\cite[Proposition 2.25, Theorem 2.31 parts (a) and (b)]{FurmanHandbook} this implies that any
    extension of the Furstenberg measure is a measure preserving extension.
    %In particular, $\nu$ is a measure preserving extension of
    %$\bar{\nu}_{\Xbar}$, which means that
    %when writing the disintegration $\nu=\int_{\Xbar} \nu_p d\bar{\nu}_{\Xbar}(p)$
    %then the collection $\nu_p$ is equivariant in the sense that
    %$g\nu_p=\nu_{gp}$ for $\mu\otimes \bar{\nu}_{\Xbar}$-almost every $(g,p)$.
    We disintegrate $\nu$
    into a collection of measures $\set{\nu_p}_{p\in\Xbar}$ with respect to the map $\pi$ as in Definition~\ref{def:naturallift}.
    Since we have established that $\nu$ is a measure preserving extension of $\bar{\nu}_{\Xbar}$ the composition map $b\mapsto p_b\mapsto \nu_{p_b}$ is equivariant.
    Using the fact that $(b\mapsto p_b)_*\beta=\bar{\nu}_{\Xbar}$ we have
    %(since $\bar{\nu}_{\Xbar}$ is the image of $\be$ under $b\mapsto p_b$),
\begin{equation}\label{eq:disinitgration}\nu=\int_{\Xbar} \nu_p \dv{\bar{\nu}_{\Xbar}}p = \int_B \nu_{p_b}\dv\be b,\end{equation}
    Since the collection $\set{\nu_{p_b}}_{b\in B}$ is equivariant and the measures $\set{\nu_b}_{b\in B}$ are the
    unique equivariant collection satisfying $\nu = \int_B\nu_b\dv\be b$ we deduce
    that $\nu_b = \nu_{p_b}$ for $\be$-almost every $b\in B$.
    Thus we have shown that $\nu$ is a measure preserving  $k$-extension of
    $\bar{\nu}_{\Xbar}$.
\end{proof}
As mentioned in Remark~\ref{rem:amplification} the statement of Theorem~\ref{thm:case2} is amplified 
in the case $\mu$ satisfies assumption~\eqref{case:ac} to the fact that the natural lift is the unique $\mu$-stationary measure. 
To see this, note that when a $\Ga_\mu=\SO(Q)(\bR)$-orbit intersects a fibre of $\pi$ above a plane in 
the circle of isotropic planes $\cC=\supp \bar{\nu}_{\Xbar}$
it intersects it in infinitely many points. Thus the possibility of the existence of an ergodic finite extension is excluded.
%<---
%%
%--->Non escape of mass
%%
%%--->replacing mu by mu^n
\section{\label{sec:Non-escape-of-mass/contracted}Non-escape of mass}
In this section we construct a proper function on $X$ which can be thought of like a height function.
We will show that this function is contracted by the averaging operator induced by $\mu$, where
$\mu$ is as in~\ref{case1} or~\ref{case2}.
The existence of such a function is important
in two ways.
First, it implies that almost surely the random walks of $\mu$ on $X$ are recurrent in a strong sense.
This recurrence will imply that the limiting distribution of almost every random walk is a probability measure or in other words that mass does not escape.
In turn this will allow us to conclude Theorem~\ref{thm:full}\eqref{main4} at the end of this section.
Second,
this function will also play an important role in the proof of Theorem~\ref{thm:full}\eqref{main2}, given in
\S\ref{sec:non atomicity}.
\subsection{Replacing $\mu$ by $\mu^{*n_0}$}\label{ssec:replacing}
Before starting the construction of the contracted function we note that
the statements in Theorem~\ref{thm:full} are not affected by replacing $\mu$ by $\mu^{*n_0}$.
%\usnote{make sure this is correct}
Using Lemma~\ref{lem:replacement}
we choose $n_0>0$ and make the replacement $$\mu\defi\mu^{*n_0}$$ so that for some $L_0>0$ the following holds:
\begin{enumerate}[(1)]
    \item In both~\ref{case1} and~\ref{case2}, for all $v\in\bR^3\mz$ and $w\in \wedge^2\bR^3\mz$ one has
        \begin{align}
            \label{eq:22063}
            \int_G \log \pa[\bigg]{\frac{\norm{gv}}{\norm{v}}\Big/\frac{\norm{g w}^{1/2}}{\norm{w}^{1/2}}}\dv{\mu} g > L_0.
        \end{align}
    \item In case~\ref{case1}, for all $p\in\Gr_2(\bR^3)$, $u\in\wedge^3\gor_p\mz$
        %(where
        %recall that $\gor_p = g\gor_0$ for $g\in H$ such that $gp_0=p$),
        and $v\in \gog\smallsetminus \gor_p$ one has
        \begin{align}
            \label{eq:22064}
            \int_G \log \pa[\bigg]{\frac{\norm{g(v\wedge u)}}{\norm{v\wedge u}}\Big/\frac{\norm{gu}}{\norm{u}}}\dv{\mu} g > L_0.
        \end{align}
\end{enumerate}
%<---
%%
%--->contraction hyp
\subsection{The contraction hypothesis}\label{subsec:Construction-of-the}

Suppose that $G$ acts continuously on a locally compact metric space
$Y$ and $\eta\in\cP(G)$, then for a measurable  $f:Y\to[0,\infty)$ define
\[
    \on A_\eta f\left(x\right)\defi\int_{G}f\left(gx\right)\dv\eta g.
\]
Recall that a function $f:Y\to[0,\infty)$ is said to be
\emph{proper} if $f^{-1}\left(C\right)$ is pre-compact for all compact
subsets $C\subset [0,\infty)$. It is said to be \emph{lower semi-continuous} if the sublevel sets
$f^{-1}([0,M])$ are closed for any $M\ge0$.
\begin{definition}\label{def:CH}
    A function $f:Y\to[0,\infty)$ satisfies the \emph{contraction hypothesis} with respect to $\eta$ on $Y$
    if 
    there exist constants $c<1$ and $b>1$ such that
    \[
        \on A_{\eta} f\left(y\right)\leq cf\left(y\right)+b\quad\textrm{for all }y\in Y
    \]
    We use the notation $\on{CH}_\eta(Y)$ for the set of all such functions.
\end{definition}
%--->comment
%\usnote{Commented lemma regarding $CH_\mu$ vs $CH_{\mu^*n}$}
%%%%%%%%%%%%%%%%%%%%%%%%%%%%%%%%%

%%%%%%%%%%%%%%%%%%%%%%%%%%%%%%%%%%%
\begin{comment}
    \begin{lemma}
        \label{lem:Lturning n_o to n}Let $Y$ be any locally compact metric
        space with a continuous $G$ action, then there exists a proper function
        in $\mathrm{CH}_{\mu}\left(Y\right)$ if and only if for all $n\in\mathbb{N}$
        there exists a proper function in $\mathrm{CH}_{\mu^{*n}}\left(Y\right)$.
    \end{lemma}
    \begin{proof}
        Clearly if $f\in\mathrm{CH}_{\mu}\left(Y\right)$ then $f\in\mathrm{CH}_{\mu^{*n}}\left(Y\right)$
        for any $n\in\mathbb{N}$. For the converse, let $n\in\mathbb{N}$
        be arbitrary and suppose that $f\in\mathrm{CH}_{\mu^{*n}}\left(Y\right)$
        is a proper function and thus there exists $c<1,$ $b>0$ such that
        \[
            A_{\mu}^{n}\left(f\right)\leq cf+b.
        \]
        Now one can check that the function $f'\defi\sum_{k=0}^{n-1}c^{1-\frac{k+1}{n}}A_{\mu}^{k}\left(f\right)$
        is proper and is such that
        \[
            A_{\mu}\left(f'\right)\leq c^{1/n}f'+b
        \]
        which proves that $f'\in\mathrm{CH}_{\mu}\left(Y\right)$.
    \end{proof}
\end{comment}
%%%%%%%%%%%%%%%%%%%%%%%
%<---

%%%%%%%%%%%%%%%%%%%%%%%%%

Our next goal is to construct a function $f\in\mathrm{CH}_{\mu}\left(X\right)$.
The idea constructing contracted functions in
order to establish some kind of recurrence can be traced back to the
paper~\cite{MR2087794} of A. Eskin and G. Margulis. These ideas were
later taken up and used by Benoist and Quint in~\cite{BQInventiones},
~\cite{BQAnnals} and~\cite{BQJams}. The following lemma is an
extension of~\cite[Lemma 6.12]{BQJams} which in turn is an extension
of~\cite[Lemma 4.2]{MR2087794}. But first we introduce a definition.
\begin{definition}\label{def:uniform mu exp}
    Let $\mathcal{F}$
    be a family of positive functions on $G$ and $\eta\in \cP(G)$ such that:
    \begin{enumerate}[(1)]
        \item \label{enu:1st deriv cond 1}%
            There exist $\delta_{0}>0$ and $0\leq I_0<\infty$ such that
            \[
                \int_{G}\sup_{f\in\mathcal{F}}f\left(g\right)^{\del_0}\dv\eta g\leq I_{0}.
            \]
        \item \label{enu:1st deriv cond 2}There exists $L_{0}>0$ such that
            \[
                \inf_{f\in\mathcal{F}}\int_{G}\log f\left(g\right)\dv\eta g \geq L_{0}.
            \]
    \end{enumerate}
    Then, we say that $\cF$ is \emph{uniformly $(\del_0,I_0,L_0)$-expanded} by $\eta$.
\end{definition}
Using this definition we prove a very mild generalisation of~\cite[Lemma 6.12]{BQJams}.
The proof is identical to the one given there.
\begin{lemma}\label{lem:bounded compact sets/first derivative positive}
    Let $\eta\in\cP(G)$ and
    $(\del_0,I_0,L_0)$ be positive parameters.
    Let $\mathcal{F}$ be a family of positive functions on $G$ uniformly $(\del_0,I_0,L_0)$-expanded by $\eta$.
    Then there exists $\delta_{1} = \del_1(\del_0,I_0,L_0)>0$ such that for all $0<\delta\leq\delta_{1}$
    there exists $0<c = c(\del,L_0)<1$ such that for all
    %any $\eta\in \cP(G)$ with respect to which $\cF$
    %is uniformly $\eta$-expanded with parameters $(\del_0,I_0,L_0)$ and for any
    $f\in\mathcal{F}$ one has
    \[
        \int_{G}f\left(g\right)^{-\delta}\dv\eta g\leq c.
    \]
\end{lemma}
\begin{proof}
    Set $\delta_{1}\defi\min\left\{ \frac{\delta_{0}}{2},\frac{L_0\del_0^2}{4I_0}\right\} $
    and let $\eta\in\cP(G)$ be such that $\cF$ is uniformly $(\del_0,I_0,L_0)$-expanded according to Definition~\ref{def:uniform mu exp}.
    We will use the facts that
    \[
        \exp\left(x\right)\leq1+x+\frac{x^{2}}{2}\exp\left(\left|x\right|\right)\quad\textrm{and}\quad x^{2}\leq\exp\left(\left|x\right|\right)
    \]
    for all $x\in\mathbb{R}$. Then for any $f\in\mathcal{F}$ and $\del\in\bR$
    we have
    \begin{alignat*}{1}
        \int_{G}f\left(g\right)^{-\delta}\dv\eta g & =\int_{G}\exp\left(-\delta\log f \left(g\right)\right)\dv\eta g\\
                                                   & \leq1-\delta\int_{G}\log f\left(g\right)\dv\eta g+\frac{\delta^{2}}{2}\int_{G}\left(\log f\left(g\right)\right)^{2}f\left(g\right)^{\delta}\dv\eta g
        %& = \int_{G}\exp\left(-\delta\log\left(f\left(g\right)\right)\right)d\eta\left(g\right)\\
        %                                                            & \leq1-\delta\int_{G}\log\left(f\left(g\right)\right)d\eta\left(g\right)+\frac{2\delta^{2}}{\del_0^2}
        %     \int_{G}\left(\frac{\del_0}{2}\log\left(f\left(g\right)\right)\right)^{2}f\left(g\right)^{\delta_0/2}d\eta\left(g\right)\\
        %     & \leq1-\delta\int_{G}\log\left(f\left(g\right)\right)d\eta\left(g\right)+\frac{2\delta^{2}}{\del_0^2}
        %      \int_{G}f\left(g\right)^{\delta_0}d\eta\left(g\right).
    \end{alignat*}
    and
    $$\left(\log f(g)\right)^2\leq \frac{4}{\del_0^2}f(g)^{\del_0/2}.$$
    Using these inequalities together with conditions
    (\ref{enu:1st deriv cond 1}) and (\ref{enu:1st deriv cond 2}) of Definition~\ref{def:uniform mu exp} and our choice of $\del_1$,
    we see that for
    all $f\in\mathcal{F}$ and $0<\del<\delta_{1}$ one has
    \[
        \int_{G}f\left(g\right)^{-\delta}\dv\eta g \leq1-\delta L_{0}+2\frac{\del^2}{\delta_{0}^{2}}I_{0}\leq1-\frac{\delta}{2}L_{0}<1,
    \]
    so the statement holds with $c =1-\frac{\delta}{2}L_{0}$ as required.
\end{proof}
\begin{remark}
    It will be important for us
    %in \S\ref{sec:non atomicity}
    that given $\eta\in\cP(G)$ and a family of positive functions $\cF$ uniformly $(\del_0,I_0,L_0)$-expand by $\eta$ the constants $\del_1$ and $c$ whose existence is assured by
    Lemma~\ref{lem:bounded compact sets/first derivative positive} are uniform
    over all measures in the set 
    $$\set{\eta\in\cP(G): \textrm{$\cF$ is uniformly $(\del_0,I_0,L_0)$-expanded by $\eta$}}.$$
    %as long as $\cF$ is uniformly $\eta$-expanded with fixed parameters $(\del_0,I_0,L_0)$.
\end{remark}
%<---

Let $\Lam$
denote a 2-lattice in $\bR^3$ and let $\br{\Lam}\in X$ denote
the corresponding homothety class. We denote by $\av{\Lam}$ the co-volume of
$\Lam$ in the plane it spans.  For any
$v\in \Lam$ we define the \textit{normalised length} of $v$ with respect to $\Lam$ to be
$$N_\Lam(v) \defi \frac{\norm{v}}{\av{\Lam}^{1/2}}.$$
This quantity already appeared implicitly in the proof of Proposition~\ref{prop:unipinv}. We let,
\[
    f_{\Lam,v}\left(g\right)\defi \frac{N_{g\Lam}(gv)}{N_\Lam(v)}
    %\frac{\norm{gv}}{\norm{g(x_1\wedge x_2)}^{1/2}}\Big/ \frac{\norm{v}}{\norm{x_1\wedge x_2}^{1/2}}
\]
That is, $f_{\Lam,v}(g)$ is the cocycle that measures by which factor $v$ is stretched under the action of $g$ taking into account the normalisation
factors which make $\Lam$ and $g\Lam$ of co-volume 1 in their respective planes.

Let $$\mathcal{F} \defi \left\{ f_{\Lam,v}\right\} _{\br{\Lam}\in X,v\in \Lam}.$$
The main step towards constructing a function in $\on{CH}_\mu(X)$ is the following.
\begin{proposition}
    \label{prop:main prop - families satisfy conditions}
    Let $\mu$ be as in~\ref{case1} or~\ref{case2} and suppose that~\eqref{eq:22063} 
    %and~\eqref{eq:22064} 
    hold.
    Then $\mathcal{F}$ is uniformly $(\del_0,I_0,L_0)$-expanded by $\mu$ for some  $L_0$ as in~\eqref{eq:22063}
    and some positive $\del_0$ and $I_0$.
\end{proposition}
\begin{proof}
    We verify conditions~\eqref{enu:1st deriv cond 1},~\eqref{enu:1st deriv cond 2} of Definition~\ref{def:uniform mu exp}. The validity
    of condition~\eqref{enu:1st deriv cond 1} is immediate with say $\del_0=1$ and some $I_0<\infty$
    from the assumption that $\mu$ is compactly supported.  For condition~\eqref{enu:1st deriv cond 2} we note that
    if $\Lam=\spa_\bZ\set{u,w}$, then
    $$f_{\Lam,v}(g) = \frac{\norm{gv}}{\norm{g(u\wedge w)}^{1/2}}\Big/ \frac{\norm{v}}{\norm{u\wedge w}^{1/2}}$$
    for all $v\in\Lam$. It follows that condition~\eqref{enu:1st deriv cond 2} is implied by
    equation~\eqref{eq:22063} which holds for $\mu$ as indicated in \S\ref{ssec:replacing}.
\end{proof}
%\usnote{commented text defining $C_{E}(V)$ (some sort of norm)}
%--->comment
%The obvious corollary of Proposition~\ref{prop:main prop - families satisfy conditions}
%(obtained via Lemma~\ref{lem:bounded compact sets/first derivative positive})
%is the following.
%\begin{corollary}
%\label{lem:uniform contraction-2}Then for all $\delta$ sufficiently
%small there exist $n_{0}>0$ and $0<c<1$ such that for all $n\geq n_{0}$
%and $x=\left[\left\langle x_{1},x_{2}\right\rangle _{\mathbb{Z}}\right]\in X$
%one has
%\[
%\int_{G}\frac{\left\Vert g\left(x_{1}\wedge x_{2}\right)\right\Vert ^{\delta/2}}{\left\Vert gx_{1}\right\Vert ^{\delta}}\mathrm{d}\mu^{*n}%\left(g\right)\leq c\frac{\left\Vert x_{1}\wedge x_{2}\right\Vert ^{\delta/2}}{\left\Vert x_{1}\right\Vert ^{\delta}}.
%\]
%\end{corollary}
%Here and in the next section we will need to refer to various constants
%depending on a compact set $E\subseteq G$ and a representation $V$
%of $H$. These constants are defined by
%\begin{equation}
%C_{E}\left(V\right)\defi\max_{g\in E}\left\Vert g\right\Vert _{V}^{\pm1}.\label{eq:def of constants}
%\end{equation}
%<---
For $\br{\Lam}\in X$ we set
\[
    u_X(\br{\Lam})\defi\pa[\Big]{\min\nolimits_{v\in \Lam\mz}N_\Lam(v)}^{-1}.
\]
It is clear from the definition of $N_\Lam$ that  $u_X$ is well defined in the sense that its value 
does not depend on the choice of $\Lam$ from $[\Lam]$.
Moreover, by Mahler's compactness criterion that $u_X:X\to [0,\infty)$ is a continuous proper function.
The following proposition establishes the existence of a function which satisfies the contraction hypothesis of Definition~\ref{def:CH}.
\begin{proposition}
    \label{prop:contracted func on X}
    Let $\mu$ be as in~\ref{case1} or~\ref{case2} and suppose that~\eqref{eq:22063} and~\eqref{eq:22064} hold.
    %(satisfying the assumption of \S\ref{ssec:replacing}).
    Then, for all $\delta$ sufficiently small $u_{X}^{\delta}\in\mathrm{CH}_{\mu}\left(X\right)$.
\end{proposition}
\begin{proof}
    Given $M>0$ we split $X$ into $X^{\le M} = u_X^{-1}([0,M])$ and $X^{>M} = u_X^{-1}((M,\infty)).$
    %When $M\ge1$, because we are considering 2-lattices, if $\br{\Lam}\in X^{>M}$ then
    %$\min_{v\in\Lam\mz}N_\Lam(v)$ is attained by a unique (up to sign) vector $v_{\textrm{min}}(\Lam)\in\Lam$.
    %Any other vector $w\in \Lam$ which is not co-linear to $v_{\textrm{min}}(\Lam)$ is of normalised length $N_\Lam(w)\ge1$.
    %It follows from the fact that $\mu$ is compactly supported that if $M$ is large enough and $\br{\Lam}\in X^{>M}$
    %then for any $g\in\supp \mu$,  $$v_{\textrm{min}}(g\Lam) = g v_{\textrm{min}}(\Lam).$$
    We claim that there exists $M>0$ such that if $[\Lam]\in X^{>M}$ then there exists a unique (up to sign) vector $v_\mathrm{min}(\Lam)\in\Lam$ such that
    $u_X([\Lam])=N_\Lam(v_\mathrm{min}(\Lam))$ and $u_X(g[\Lam])=N_{g\Lam}(gv_\mathrm{min}(\Lam))$ for all $g\in\supp\mu$.
    First we note that because we are dealing with 2-lattices for any $M\geq 1$ and $[\Lam]\in X^{>M}$ the vector $v_\mathrm{min}(\Lam)$ is well defined up to sign.
    Now set $M\defi\sup_{g\in\supp\mu}\norm g^2$ and suppose there exists $v'\in g\Lam$ with $v'\neq gv_\mathrm{min}(\Lam)$ and
    such that $u_X([\Lam])=N_{g\Lam}(v')$. Then note that
    $$\frac{\norm {g^{-1}v'}}{\norm g \av {g\Lam}^{1/2}}\leq N_{g\Lam}(v')\leq N_{g\Lam}(gv_\mathrm{min}(\Lam))\leq \frac{\norm{g}\norm{v_\mathrm{min}(\Lam)}}{\av{g\Lam}^{1/2}}.$$
    Since $g^{-1}v'$ and $v_\mathrm{min}$ cannot be colinear we see that this is a contradiction if $\norm{v_\mathrm{min}(\Lam)}<1/M$ since in this case we would get that $\Lam$ contains two non colinear vectors with norm less than 1.

    Next suppose that $[\Lam]\in X^{>M}$ and write
    $\Lam = \spa_{\bZ}\set{u,w}$ so that for $\del>0$ one has
    $$\on A_\mu u_X^\del(\br{\Lam}) = \int_G u_X^\del(g\br{\Lam})\dv\mu g =
    \int_G \pa[\bigg]{\frac{\norm{g(u\wedge w)}^{1/2}}{\norm{gv_{\mathrm{min}}(\Lam)}}}^\del \dv\mu g.
    %\le  c\frac{\norm{v_{\on{min}}(\Lam)}^{-\del}}{\norm{(u\wedge w)}^{-\del/2}} = u_X^\del(\br{\Lam}).
    $$
    By Proposition~\ref{prop:main prop - families satisfy conditions} and Lemma~\ref{lem:bounded compact sets/first derivative positive}, if $\del$ is small enough, there exists $0<c<1$ such that
    $$\int_G \pa[\bigg]{\frac{\norm{g(u\wedge w)}^{1/2}}{\norm{gv_{\mathrm{min}}(\Lam)}}}^\del \dv\mu g
    \le c \pa[\bigg]{\frac{\norm{(u\wedge w)}^{1/2}}{\norm{v_{\mathrm{min}}(\Lam)}}}^\del  = u_X^\del(\br{\Lam}).$$
    Fix $\del$ and let $0<c<1$ be such a number.
    If $\br{\Lam}\in X^{\le M}$ then from the compactness of $\supp \mu$ and the properness
    $u_X^\del$ we conclude that $\on A_\mu u_X^\del(\br{\Lam})\le b(M,\del,\mu)= b $. In any case we have
    $$\on A_\mu u_X^\del(\br{\Lam})\le c u_X^\del(\br{\Lam}) + b,$$
    that is, $u_X^\del\in \on{CH}_\mu(X)$ as desired.
\end{proof}
Using the existence of a proper function in $\on{CH}_\mu(X)$ we can give a proof of Theorem~\ref{thm:full}\eqref{main4} by
citing Benoist and Quint.
\begin{proof}[Proof of Theorem~\ref{thm:full}\eqref{main4}]
    Let $x\in X$ be given. By~\cite[Corollary 2.2]{BQInventiones}, any weak-* accumulation point of the sequence
    $\frac{1}{n}\sum_{k=1}^n \mu^{*k}*\del_x$ is a probability measure on $X$. It is also evidently $\mu$-stationary.

    %--->comment
    %we need the following lemma.
    %\begin{lemma}
    %For almost any $b$, for any $x\in X$, the sequence of empirical measures
    %$$\del_{b,x,N} \defi \frac{1}{N}\sum_{n=1}^N \del_{b_n^1x}$$
    %is tight; i.e.\ any accumulation point of it is a probability measure.
    %\end{lemma}
    %\begin{proof}
    %The obvious relation between
    %$\del_{b,x,N}$ and $\del_{Sb,b_1x,N-1}$ implies that the set
    % $$\set{b: \exists x\in X \textit{ s.t.\ } \del_{b,x,N} \textit{ is not tight}}$$
    %is $S$-invariant. By the ergodicity of the shift, if the statement of the lemma is false, the above set is of full $\be$-measure.
    %\usadd{something is wrong in the way I am trying to prove this...}
    %\end{proof}
    %<---
    Moreover, it follows from~\cite[Corollary 3.3]{BQAnnals2} that for all $x\in X$, for $\be$-almost every $b\in B$,
    any weak-* accumulation point of the sequence $\frac{1}{n}\sum_{k=1}^n\del_{b_k^1x}$ is $\mu$-stationary. Thus, we are only
    left to establish that $\be$-almost surely, such an accumulation point is a probability measure. This is again a consequence
    of the existence of a function in $\on{CH}_\mu(X)$. Indeed,~\cite[Example 3.1, Proposition 3.9]{BQAnnals2} implies this exact
    statement since the contracted function $u_X$ is proper.
\end{proof}
%<---
%%
%--->non-atomic
\section{The limit measures are non-atomic}\label{sec:non atomicity}
In this section we assume $\mu$ is as in~\ref{case1} and also assume the validity of~\eqref{eq:22063},~\eqref{eq:22064} as in \S\ref{ssec:replacing} which is ensured by replacing $\mu$ by $\mu^{*n_0}$ if necessary.
The main goal of this section is to prove Theorem~\ref{thm:full}\eqref{main2}.
\subsection{Metric considerations}\label{ssec:metric}
We will need to have some understanding of a convenient metric on $X$.
In order to do this we study the local structure of $X$.
For $p\in\Xbar$ let
$$\Pi_p:\gog\to(\gor_p)^\perp\defi\gom_p,$$ be the orthogonal
projection 
where the inner product in the above definition is supposed to be $K$-invariant.
%The crucial difference is that $\Pi_p$ is no longer equivariant.
It is important to note that $\Pi_p$ is not equivariant.
We use the convention that for any representation $V$ of $H$ the notation
$\norm{g}_V$ stands for the operator norm of $g$ on $V$.

%We will use the fact that a small enough neighbourhood of  of a point $x\in X$ looks like $\gom_{\pi(x)}$.
Let $\dx$ denote a metric on $X$ induced by a Riemannian metric obtained in the following manner: For
a point $x\in X$ the derivative at the identity $d_e \al_x$ of the orbit map $\al_x:G\to X$, $g\mapsto gx$ satisfies
$\ker d_e\al_x = \gor_p$ where $p$ is the plane of $x$. Since $d_e \al_x$ is of full rank, it restricts to a linear isomorphism
$d_e\al_x:\gom_p =\gor_p^\perp\to T_xX$. We use this isomorphism to transport the inner product structure that
$\gom_p$ inherits from $\gog$ to $T_xX$ thus inducing a Riemannian metric on $X$.

If we denote by $c_g$ conjugation by $g$ then for any $x\in X$
we have the following commutative diagram and its derivative:
$$\xymatrix{
    G\ar[d]_{c_g}\ar[r]^{\al_x} &X\ar[d]^g \\
    G\ar[r]_{\al_{gx}} & X
}
\quad\quad
\xymatrix{
    \gog\ar[d]_{\Ad_g}\ar[r]^{d_e\al_x} & T_xX\ar[d]^{d_xg}\\
    \gog\ar[r]_{d_e\al_{gx}} &T_xX
}$$
The fact that the horizontal maps on the right diagram are linear surjections of norm at most 1 implies that the norm
of $d_xg:T_xX\to T_{gx}X$ is bounded by $\norm{g}_\gog$. 
%which we abbreviate to $\norm{g}_\gog$.
In particular, since the metric $\dx$ is defined in terms of length of paths, it satisfies the important
inequality
\begin{equation}\label{eq:1525}
    \dx(gx,gy)\le \norm{g}_\gog \dx(x,y)\quad\textrm{for all}\ g\in G\ \textrm{and}\ x,y\in X.
\end{equation}
Consider $X\times \gog$ as a Riemannian manifold and consider $X^*\defi \set{(x,v):v\in \gom_{\pi(x)}}$ as a submanifold. The map
$\psi: X^*\to X\times X$ defined by $\psi(x,v) = (x,\exp(v) x)$ is smooth and has the property that on the submanifold
$X^*_0 = \set{(x,0)\in X^*}$ the derivative
$$d_{(x,0)}\psi:T_xX\oplus \gom_{\pi(x)} \to T_xX\oplus T_xX$$
is an isometry. In fact, it equals the identity after identifying
$T_xX$ with $\gom_{\pi(x)}$ as described earlier. Since $X^*$ and $X\times X$ are of the same dimension we conclude that
there is an open neighbourhood $X^*_0\subset \cV_0\subset X^*$ that is mapped by $\psi$ diffeomorphically onto an
open neighbourhood $\Del_X = \psi(X_0^*)\subset \cU_0\subset X\times X$.
Given $(x,y)\in \cU_0$ we define the \textit{orthogonal displacement vector}  $\oxy$ between $x$ and $y$ to
be the unique vector $v\in \gom_{\pi(x)}$ such that $(x,v)\in \cV_0$ and $\psi(x,v) = (x,y)$, or in other words $y=\exp(v)x$.
We prove the following.
\begin{lemma}\label{lem:nbd}
    For any compact set $E\subset G$ and all $0<c<1$ there exists a neighbourhood of the diagonal $\cU\subset X\times X$
    such that for all $(x,y)\in \cU$ and $g\in E\cup E^{-1}\cup\set{e}$ one has:
    \begin{enumerate}[(1)]
        \item\label{p:0931} The orthogonal displacement $\ogy$ is well defined.
        \item\label{p:0932} The following inequality holds $c \norm{\ogy}\le \dx(gx,gy)\le c^{-1} \norm{\ogy}$.
        \item\label{p:0933} For all $u\in \wedge^3 \gor_{\pi(x)}\mz$ one has
            $c \norm{\ogy }\le \norm{g(\oxy\wedge u)}/ \norm{ gu}\le c^{-1} \norm{\ogy}$.
            %$$c\frac{\norm{g(\oxy\wedge u)}}{\norm{ gu}} \le \dx(gx,gy)\le c^{-1} \frac{\norm{g(\oxy\wedge u)}}{\norm{ gu}}.$$
    \end{enumerate}
\end{lemma}
\begin{proof}
    Throughout the proof we may assume that $E=E\cup E^{-1}\cup \set{e}$ by enlarging it if necessary.
    First we prove~\eqref{p:0931}.
    Let $\cU_0$ be the neighbourhood of $\Del_X$ on which the orthogonal displacement is defined.
    %Assume that $E=E\cup E^{-1}\cup \set{e}$ by enlarging it if necessary.
    We first show that $\cap_{g\in E}g\cU_0$ contains a neighbourhood of
    $\Del_X$. This is done
    by showing that for any compact $K\subset X$ we have
\begin{equation}\label{eq:nbhdofdiag}\on{d}_{X\times X}(K\times K\smallsetminus \cap_{g\in E}g\cU_0,\Del_X)>0.\end{equation}
    To this end, let $K\subset X$ be a compact set.
    %and let $\hat{K} = EK$.
    By~\eqref{eq:1525} and the compactness of $E$,
    we deduce from the fact that
    $$\on{d}_{X\times X}(({K}\times {K}\smallsetminus \cU_0),\Del_X)>0,$$ that
    $$\on{d}_{X\times X}(\cup_{g\in E}g({K}\times {K}\smallsetminus \cU_0),\Del_X)>0.$$
    However, since $K\times K\smallsetminus \cap_{g\in E}g\cU_0\subset \cup_{g\in E} g({K}\times {K}\smallsetminus \cU_0)$
    the previous equation implies~\eqref{eq:nbhdofdiag} as claimed.

    We conclude that $\cap_{g\in E}g\cU_0$ contains a neighbourhood $\cU_1$ of $\Del_X$
    and deduce that for any $g\in E$ and $(x,y)\in \cU_1$, $(gx,gy)\in \cU_0$ so that the orthogonal displacement
    $\ogy$ is well defined.

    Now we will prove~\eqref{p:0932}.
    Let $0<c<1$ be given and for $A,B\in\bR$ write $A\sim_c B$ to denote $cA< B< c^{-1}A$.
    Consider the map $\psi:X^*\to X\times X$ and the neighbourhood $\cV_0$ as defined before the statement of the lemma.
    Let $\cV_1\defi \psi^{-1}(\cU_1)\subset \cV_0$.
    Since the differential
    $d\psi$ is an isometry on the submanifold $X^*_0$ and $E$ is compact there is a neighbourhood 
    $\cV_2\subset \cV_1$ of $X^*_0$
    such that for all $(x,v)\in \cV_2$ and $g\in E$ one has $\norm{d_{(gx,gv)}\psi^{\pm1}}\sim_c1$.
    The image $\cU_2\defi\psi(\cV_2)\subset \cU_1$ is then a
    neighbourhood of $\Del_X$.
    Next we replace $\cV_2$ and $\cU_2$ by even smaller neighbourhoods $\cV_3$ and $\cU_3\defi\psi(\cV_3)$ 
    of $X_0^*$ and $\Del_X$ respectively,
    %such that $\psi(\cV_2)=\cU_2$,
    so that for all $(x,v)\in \cV_3$ the whole interval $\set{(x,tv):t\in[0,1]}$ is contained in $\cV_2$.
    Similarly,
    for any $(x,y)\in \cU_3$, the geodesic path between $(x,x)$ and $(x,y)$ is contained in  $\cU_2$.
    %\usnote{Expand here a bit and explain why $\cU_3,\cV_3$ exist}

    Given $(x,y)\in \cU_3$ 
    let $\oxy$ denote the corresponding orthogonal displacement so that $\psi(x,\oxy) = (x,y)$. 
    Since the path $\zeta(t) = (x,t\oxy)$
    is the geodesic in $X^*$ from $(x,0)$ to $(x,\oxy)$ and is of length $\norm{\oxy}$ and since it is contained in $\cV_2$ 
    on which
    $\norm{d\psi}\sim_c 1$, we conclude that the image path $\psi(\zeta(t))$ connecting $(x,x)$ to $(x,y)$ 
    has length $\sim_c \norm{\oxy}$.
    But, the distance in $X\times X$ from $(x,x)$ to $(x,y)$ is exactly $\dx(x,y)$ and so we obtain the inequality
    $\dx(x,y)< c^{-1}\norm{\oxy}$ for all $(x,y)\in \cU_3$ and $g\in E$.

    For the other inequality, let $(x,y)\in \cU_3$, $g\in E$ and let $\zeta(t)$ denote the geodesic path between $(x,x)$ 
    to $(x,y)$ which is
    of length $\dx(x,y)$ as mentioned earlier.
    By the choice of $\cU_3$, $\zeta(t)\in \cU_2$ for all $t$ and therefore, on applying
    $\psi|_{\cU_2}^{-1}$ we obtain a path connecting $(x,0)$ and $(x,\oxy)$ whose length is $<c^{-1}\dx(x,y)$.
    Since the distance between $(x,0)$ and $(x,\oxy)$ is $\norm{\oxy}$ we obtain 
    the inequality $\norm{\oxy}<c^{-1}\dx(x,y)$.
    In total we showed that for all $(x,y)\in \cU_3$ and $g\in E$ one has $\dx(x,y)\sim_c\norm{\oxy}$.
    To finish, we replace $\cU_3$ by $\cU_4$ a neighbourhood of $\Del_X$ contained in $\cap_{g\in E} g\cU_3$
    (in a similar fashion to the proof of part~\eqref{p:0931}) and conclude that for all $(x,y)\in \cU_4$ and all $g\in E$
    we have that $(gx,gy)\in \cU_3$ and therefore $\dx(gx,gy)\sim_c\norm{\ogy}$ as desired 

    Finally we prove~\eqref{p:0933}. Let $\cU=\cU_4$ be as in the proof of part~\eqref{p:0932} and let $(x,y)\in \cU$.
    %By shrinking $\cU$ if necessary and using the compactness of $E$ we may suppose that $\cup_{g\in E}g\cU$ is "in the domain of $\log$".
    %For each $x\in X$ there exists a neighbourhood $\cL_{0,x}$ of the identity in $G$ so that for all $y\in\cL_{0,x}x$ the element $\log y\in\gom_x$ is well defined by being the unique
    %element such that $\exp\log y =y$. A subset $\cW$ being in the domain of the log means that $\cW\subset\cup_{x\inX}\cL_{0,x}x$.
    Note that $\norm{g(\oxy\wedge u)}/ \norm{ gu} = \norm{\Pi_{gx}(g\oxy)}$ for all $u\in\wedge^3\gor_{\pi(x)}\mz$
    and that both of $\exp(\ogy)$ and  $\exp(g\oxy)$ take $gx$ to $gy$ so
    $$\exp(-g\oxy)\exp(\ogy)\in\Stab_G(gx).$$
    There is a neighbourhood of the identity $\cL_0$ in $G$ such that $\log:\cL_0\to\gog$ is well defined. By shrinking $\cU$ if
    necessary we may suppose that the above product is in $\cL_0$ for $g\in E$ and $(x,y)\in \cU$.
    %\usnote{Amend the presentation and address the fact that this product is in the domain of $log$}
    Therefore, we may apply the logarithm and see that the result lies in $\ker \Pi_{gx} = \gor_{\pi(gx)}$ (which is 
    the Lie algebra of $\Stab_G(gx)$).
    On the other hand~\cite[\S2]{MR3237440} we have
    $$\log (\exp(-g\oxy)\exp(\ogy)) = \ogy -g\oxy + O(\norm{\ogy} \norm{\oxy}\norm{g}_\gog).$$
    Applying the projection $\Pi_{gx}$ we see that
    \begin{equation}\label{eq:1147a}
        \ogy = \Pi_{gx}(g\oxy) + O(\norm{\ogy} \norm{\oxy}\norm{g}_\gog).
    \end{equation}
    Equation~\eqref{eq:1147a} together with part~\eqref{p:0932} imply that on shrinking $\cU$ if necessary,
    the ratio 
    $\norm{\Pi_{gx}(g\oxy)}/\norm{\ogy}$ is bounded away from zero for $(x,y)\in \cU$ and $g\in E$.
    In equation~\eqref{eq:1147a} we take norms, use the triangle inequality and divide by $\norm{\Pi_{gx}(g\oxy) }$ 
    to arrive at 
    $$\norm{\ogy}\Big/ \frac{\norm{g(\oxy\wedge u)}}{\norm{ gu}}= 
    \frac{\norm{\ogy}}{\norm{\Pi_{gx}(g\oxy)}}  = 1 + O(\norm{\oxy}\norm{g}_\gog),$$
    where the cancellation in the big-$O$ is justified by the aforementioned boundedness away from zero of
    $\norm{\Pi_{gx}(g\oxy)}/\norm{\ogy}$.
    Now it is clear that since $g\in E$ and $E$ is compact, if $\cU$ is chosen small enough then the big-$O$ in the
    above equality is as small as we wish yielding part~\eqref{p:0933} of the proposition.
\end{proof}
%<---
%%%%%%%%%%%%
%--->criterion
\subsection{\label{subsec:Criterion-for-positive}A criterion for non-atomicity of the limit measures.}
In this section we leave for a moment the space $X$ and work in an abstract setting. We follow closely~\cite[\S6]{BQJams}.
We assume throughout that we are
working in the following setting:
\begin{enumerate}[{\bfseries S1:}]
    \item\label{setting1} $Y$ is a locally compact metric space on which $G$ acts.
    \item\label{setting2} There exists a proper lower semi-continuous contracted function $u_{Y}\in\mathrm{CH}_{\mu}\left(Y\right)$.
\end{enumerate}
\begin{remark}
    In \S\ref{ssec:proof of non atomicity} we will apply the results of this section to $Y=X\times X$ with the function
    $u_{X\times X}(x,y) \defi u_X^{\del}(x) + u_X^\del(y) \in \on{CH}_\mu(X\times X)$,
    where $\del>0$ is small enough so that  $u_X^\del\in \on{CH}_\mu(X)$ by Proposition~\ref{prop:contracted func on X}.
\end{remark}
For $M>0$ we consider the compact set
\[
    Y_{M}\defi\left\{ y\in Y:u_{Y}\left(y\right)\leq M\right\}.
\]
For $y\in Y_M$ and
$b\in B$ we define the stopping time
\[
    \rho_{M,y}\left(b\right)\defi\inf\left\{ n\ge 1:b_n^1y\in Y_{M}\right\}
\]
and refer to it as the \emph{first return time} to $Y_{M}$. It is a stopping time
in the sense that $\set{\rho_{M,y}(b)\le n}$ is independent from $b_{j}$ for any $j>n$.
The sets $Y_M$ have remarkable recurrence properties as reflected by the following proposition.
\begin{proposition}\label{prop:exponential moment}
    Suppose that {\bfseries S\ref{setting1}} and {\bfseries S\ref{setting2}} hold.
    Then for any $M$ large enough there exists $c>1$ such that
    $$\sup_{y\in Y_M} \int_B c^{\rho_{M,y}(b)} \dv\be b<\infty.$$
    In particular the first return time is integrable and $\be$-almost surely finite.
\end{proposition}
\begin{proof}
    This is ~\cite[Definition 6.1 and Proposition 6.3]{BQJams}.
\end{proof}
The first return time
naturally defines a map $\widehat{\rho}_{M}:B\times Y\to G$ called
the \emph{first return cocycle} where
\[
    \widehat{\rho}_{M,y}\left(b\right)\defi b_{\rho_{M,y}\left(b\right)}^{1}.
\]
In turn, the first return cocycle induces a collection of  transition probability
measures $\mu_{M,y}\in\mathcal{P}\left(G\right)$
which are the images of $\beta$ by the first return cocycle. In other
words for $y\in Y_M$ and $f\ge 0$ a measurable function on $G$,
\[
    \int_G f(g) \dv{\mu_{M,y}}g \defi\int_{B}f(\wh{\rho}_{M,y}(b))\dv\be b = \int_B f(b_{\rho_{M,y}(b)}^1) \dv\be b.
\]
Finally the transition probability measures $\mu_{M,y}$ induce  a Markov operator called
\emph{the first return Markov operator} which is denoted by $\on A_{M,\mu}$  and is defined
as follows. For any $f\ge 0$ measurable function on $Y_M$
\[
\on A_{M,\mu} f(y)\defi\int_{G}f\left(gy\right)\mathrm{d}\mu_{M,y}\left(g\right). \]
Extending Definition~\ref{def:CH} we say that $f\in\mathrm{CH}_{\on A_{M,\mu}}\left(Y_{M}\right)$
if $f:Y_M\to[0,\infty]$ is
%a proper \usadd{(not sure we need the properness)}, lower-semi-continuous function
such that there exists constants $c<1,$ $b>0$ satisfying
\[
    \on A_{M,\mu} f(y)\leq cf\left(y\right)+b\quad\textrm{for all }y\in Y_{M}.
\]
We will prove Theorem~\ref{thm:full}\eqref{main2} by an application of the following criterion.
It shows that one can deduce the non-atomicity of the limit measures of 
$\nu\in \spm\mu Y$, for $Y$ as above, if one can build functions satisfying the contraction hypothesis
on bounded parts of $Y\times Y\smallsetminus \Del_Y$, where $\Del_Y$ denotes the diagonal copy
of $Y$ in $Y\times Y$.
%The following proposition is a bit confusing as we are using the above terminology
%for $Y\times Y$ rather than for $Y$ itself.
\begin{proposition}
    \label{prop:Criterion}
    Suppose that {\bfseries S\ref{setting1}} and {\bfseries S\ref{setting2}} hold.
    Consider the product space $Y\times Y$ on which $G$ acts diagonally and consider the function
    $u_{Y\times Y}(y_1,y_2) \defi u_Y(y_1)+u_Y(y_2)$ so that $u_{Y\times Y}\in \on{CH}_\mu(Y\times Y)$. Let
    $(Y\times Y)_M$ and $\mu_{M,(y_1,y_2)}$ and $\on A_{M,\mu}$ be the sublevel sets, transition probability measures and Markov operator
    associated
    to the action of $G$ on $Y\times Y$ with respect to $u_{Y\times Y}$.

    If for every large enough $M$
    %and for any compact subset
    %$Z\subset (Y\times Y)_{M}\smallsetminus \Del_Y$,
    there exists a
    %lower semicontinuous \usadd{(the lsc assumption is currently built in the definition of $\on{CH}$)}
    proper continuous function
    $v_{M}:(Y\times Y)_{M}\smallsetminus \Del_Y\to[0,\infty)$
    %which is bounded on compact subsets of $(Y\times) and infinite on $\Del_Y\cap (Y\times Y)_{M}$
    such that
    $v_{M}\in\mathrm{CH}_{\on A_{M,\mu}}( (Y\times Y)_{M}\smallsetminus \Del_Y),$
    then for any atom-free $\nu\in \spm\mu Y$, the limit measures are $\be$-almost surely non-atomic
\end{proposition}
\begin{proof}
    This follows from~\cite[Proposition 6.16 and Proposition 6.17]{BQJams}.
\end{proof}
We continue to collect results from~\cite{BQJams} that will allow us to construct the
functions $v_M$ in Proposition~\ref{prop:Criterion}.
\begin{proposition}
    \label{lem:prop 6.7}
    Suppose that {\bfseries S\ref{setting1}} and {\bfseries S\ref{setting2}} hold.
    Then, if $N:G\to\left[0,\infty\right)$ is a continuous submultiplicative
    function,
    for any $M$ large enough there exists $\delta>0$ such that
    \[
        \sup_{y\in Y_{M}}\int_{G}N\left(g\right)^{\delta}\mathrm{d}\mu_{M,y}\left(g\right)<\infty.
    \]
\end{proposition}
\begin{proof}
    This is~\cite[Definition 6.1, Proposition 6.3 and Proposition 6.7]{BQJams}. Notice that
    we are assuming $\mu$ is compactly supported so it has finite exponential moments with respect
    to $N$ in the terminology of~\cite[Definition 6.6]{BQJams}.
\end{proof}
%%%%%%%%%%%%
For the following proposition we give a full proof. This is a slight upgrade of~\cite[Lemma 6.10]{BQJams}.
\begin{proposition}\label{prop:martingales}
    Suppose that {\bfseries S\ref{setting1}} and {\bfseries S\ref{setting2}} hold.
    Let $M$ be large enough so that Proposition~\ref{prop:exponential moment} is applicable and in particular, 
    the stopping times
    $\set{\rho_{M,y}}_{y\in Y_M}$ are integrable.

    Let $G$ act on a space
    $W$ and assume that $f:G\times W\to \bR$ is an additive cocycle in the sense
    that $f(gh,w)=f(g,hw)+f(h,w)$ for all $g,h\in G$ and $w\in W$.
    Assume that:
    \begin{enumerate}[(1)]
        \item\label{ass:101}% there exists $J_0$ such that for any $w\in\cD$, $\mu$-almost surely $\av{f(g,w)} < J_0$.
            %$ \sup_{w\in\cD} \int \av{f(g,w)}d\mu(g) < J_0$ for some $J_0$.
            There exists $J_0>0$ such that $\sup_{w\in W} \norm{f(\ph,w)}_{L^\infty(G,\mu)} < J_0$.
        \item\label{ass:102} There exists $L_0>0$ such that $\inf_{w\in W} \int_G f(g,w) \dv\mu g > L_0$.
    \end{enumerate}
    Then, for all $w\in W$ and $y\in Y_M$ we have that
    $f(\ph,w)\in L^1(G,\mu_{M,y})$ and moreover
    \begin{equation}\label{eq:1053}
        \inf_{w\in W, y\in Y_M} \int_G f(g,w)\dv{\mu_{M,y}} g \ge L_0.
    \end{equation}
\end{proposition}
\begin{proof}
    Let $\cB$ denote the Borel $\sig$-algebra of $B$ and
    $\rho:B\to \bN$ be an integrable stopping time. That is, $\int_B \rho \mathrm{d}\be <\infty$ and
    for all $n\in\bN$ one has $\set{b: \rho(b)\le n}$ is measurable with respect to
    the sub-$\sig$-algebra $\cB_n$ of $\cB$ generated by the cylinder sets obtained by specifying the first $n$ co-ordinates.
    Let $\mu_\rho\defi (b\mapsto b_{\rho(b)}^1)_*\be\in \cP(G)$ be the push-forward of $\be$ under the almost surely defined product map $b\mapsto b_{\rho(b)}^1$.
    Then, we will prove that
    under the assumptions~\eqref{ass:101} and~\eqref{ass:102}, for all $w\in W$ one has
    \begin{equation}\label{eq:1130}
        \int_G \av{f(g,w)}\dv{\mu_\rho} g \le J_0\int_B \rho \mathrm d\be\quad\textrm{and}\quad
        \int_G f(g,w)\dv{\mu_\rho} g \ge L_0\int_B \rho \mathrm d\be.
    \end{equation}
    If  $M$ is as in the statement, the stopping times
    $\set{\rho_{M,y}}_{y\in Y_M}$ are integrable so the left
    inequality of equation~\eqref{eq:1130} applied with $\rho=\rho_{M,y}$ proves that
    $f(\ph,w)\in L^1(G,\mu_{M,y})$ for all $w\in W$ and $y\in Y_M$. Moreover, the right inequality of~\eqref{eq:1130} applied with the same choice of $\rho$
    implies equation~\eqref{eq:1053} because the integral of a stopping time is at least 1.

    Fix $w\in W$ and for $i\in \bN$ let $X_i(b) \defi f(b_i,b_{i-1}^1w)$,  where $b_0^1$ denotes
    the empty product, so that by the cocycle equation
    $$f(b_{\rho(b)}^1 ,w) = \sum_{i=1}^{\rho(b)} X_i(b).$$
    With this notation, we can rewrite equation~\eqref{eq:1130} as follows
    \begin{equation}\label{eq:10052}
        \int_B \av[\bigg]{\sum_{i=1}^{\rho(b)} X_i(b)}\dv\be b\le J_0\int_B {\rho} \mathrm d\be \quad \textrm{and}\quad \int_G \sum_{i=1}^{\rho(b)} X_i(b) \dv\be b\ge L_0 \int_B \rho \mathrm d\be.
    \end{equation}
    Note that
    \begin{equation*}
        \av[\bigg]{\sum_{i=1}^{\rho(b)}X_i(b)}\le \sum_{i=1}^\infty \mb{1}_{\set{b\in B:\rho(b)\ge i}}(b)\av{X_i(b)}\quad\textrm{for all }\ b\in B.
    \end{equation*}
    Hence, using~\eqref{ass:101} which implies that   $\av{X_i}\le J_0$ and the monotone
    convergence theorem, we obtain
    \begin{align*}
        %\label{eq:10051}
        \int_B \av[\bigg]{\sum_{i=1}^{\rho(b)} X_i(b)}\dv\be b&\leq
        %\int \sum_1^\infty \mb{1}_{\set{\rho\ge i}}\av{X_i} d\be
        \sum_{i=1}^\infty \int_{\set{b\in B:\rho(b) \ge i}}  \av{X_i} \mathrm d\be\\
        \nonumber
        &\le \sum_{i=1}^\infty \be(\set{b\in B:\rho(b) \ge i}) J_0 = J_0\int \rho \mathrm d\be.
    \end{align*}
    This is the left inequality of~\eqref{eq:10052}.
    We now turn to the proof of the right inequality of~\eqref{eq:10052}.
    Consider the sequence of random variables
    $Z_n \defi \sum_{i=1}^n(X_i - L_0)$.
    Since $X_i$ is $\cB_i$-measurable for all $i\in\bN$, one has $\bE(Z_n|\cB_{n-1}) = Z_{n-1}+\bE(X_n|\cB_{n-1})-L_0$.
    Hence, provided that
\begin{equation}\label{eq:submgle}\bE(X_n|\cB_{n-1})\ge L_0\quad \be\textrm{-almost surely,}\end{equation}
    $Z_n$ is a submartingale with respect to the filtration $\cB_n$.
    Recall that the definition of conditional expectation
    is given by integration with respect to the conditional measures.
    For all $b\in B$ the conditional measure
    $\be^{\cB_{n-1}}_b$ of $\be$ with respect to $\cB_{n-1}$ at $b$ is the measure on $B$ given by
    $$\be^{\cB_{n-1}}_b = \del_{b_1}\otimes\dots\otimes  \del_{b_{n-1}}\otimes \mu^{\otimes \bN}$$
    It follows that for $\be$-almost every $b\in B$ we have
    $$\bE(X_n|\cB_{n-1})(b) = \int_B f(c_n, c_{n-1}^1w)\dv{\be^{\cB_{n-1}}_b} c=\int_G f(c_n, b_{n-1}^1 w)\dv\mu {c_n}>L_0,$$
    where the last inequality follows from assumption~\eqref{ass:102}. Hence~\eqref{eq:submgle} holds and $Z_n$ is a
    submartingale as claimed.

    It is a classical fact (see~\cite[section 10.9]{Williams}) that the process $Z_{\min\set{n,\rho}}$
    is also a submartingale with respect to $\cB_n$.  Hence, it satisfies the inequality
\begin{equation}\label{eq:limsubmtgle}\int_B Z_1 \mathrm d\be \le \liminf_{n\to\infty}\int_B Z_{\min\set{n,\rho(b)}}(b)\dv\be b.\end{equation}
    Since $\rho$ is almost surely finite $\lim_{n\to\infty}Z_{\min\set{n,\rho}}=Z_\rho$ almost surely.
    Next we claim that $Z_{\min\set{n,\rho}}$ is bounded by an integrable function. Note that
    $$\av{Z_{\min{\set{n,\rho}}}}\le \sum_{i=1}^\infty\mb{1}_{\set{b\in B: \rho(b)\ge i}} \av{X_i-L_0} $$ and hence using~\eqref{ass:101}
    and the monotone convergence theorem
    \begin{align*}
        %\int \sum_1^\infty\mb{1}_{\set{\rho\ge i}} \av{X_i-L_0}d\be
        \av{Z_{\min{\set{n,\rho}}}}
        & \leq \sum_{i=1}^\infty \int_{\set{b\in B:\rho(b)\ge i}}\av{X_i-L_0}\mathrm d\be\\
        &\le
        \sum_{i=1}^\infty (J_0+L_0)\be\mset{b\in B:\rho(b)\ge i} = (J_0+L_0)\int_B \rho \mathrm d\be.
    \end{align*}
    %We will show in a moment that the
    %submartingale $Z_{\min\set{n,\rho}}$ is dominated by an $L^1$-function and so the above $\liminf$ is actually a limit and could
    %commute with the integral sign.
    Thus, using assumption~\eqref{ass:101},~\eqref{eq:limsubmtgle} and the dominated convergence theorem we obtain
    \begin{equation*}
        0< \int_Bf(b_1,w)\dv\be b -L_0  = \int_B Z_1\mathrm d\be  \le  \lim_{n\to\infty} \int_B Z_{\min\set{n,\rho(b)}}(b) \dv\be b = \int_B Z_\rho \mathrm d\be.
        %&=\int_B \lim_n Z_{\min\set{n,\rho}} d\be \\
    \end{equation*}
    But from the definition of $Z_\rho$ we have
    \begin{equation*}
        \int_B Z_\rho \mathrm d\be =  \int_B \sum_{i=1}^{\rho(b)} (X_i(b)-L_0) \dv\be b =\int_B\sum_{i=1}^{\rho(b)} X_i(b) \dv\be b - L_0\int_B \rho \mathrm d\be.
    \end{equation*}
    Putting the last two inequalities together yields the right inequality in~\eqref{eq:10052} which finishes the proof.
    %
    %
    %    Indeed, there is an $L^1$-function dominating $\av{Z_{\min\set{n,\rho}}}$ for any $n$ and
    %    the proof  is very similar to
    %    the integrability of $\av{\sum_1^\rho X_i}$ established in~\eqref{eq:1005} and~\eqref{eq:10051}.
\end{proof}

Similarly to the scheme leading to the construction of the contracted function $u_X^\ka \in \on{CH}_\mu(X)$  in
Proposition~\ref{prop:contracted func on X},
a key point in building the functions $v_M$ which will participate in an application of Proposition~\ref{prop:Criterion} is showing
that a certain family of functions $\cF$ is uniformly $(\del_0,I_0,L_0)$-expanded by $\eta$ according to Definition~\ref{def:uniform mu exp} for a certain choice of $\eta$.
% $\eta = \mu_{M,y}$ with parameters
%uniform over $y\in Y_M$.  We begin by defining $\cF$.
Let
\begin{equation}\label{eq:1510}
    \mathcal{D}\defi\set{ (p,v):p\in\Xbar, v\in \gog \smallsetminus \gor_p}.
    %\mathfrak{m}_{p}\smallsetminus\left\{ 0\right\} \right\},
\end{equation}
%where for $p\in \on{Gr}_2(\bR^3)$, $\gor_p = g\gor_0$ is
%the 3-dimensional sub-Lie-algebra of $\gog$, where $g\in G$ satisfies $gp_0=p$ (similarly to the
%discussion in~\S\ref{ssec:boundary maps}).
For $(p,v)\in\cD$ we choose $u_p\in\wedge^3 \gor_p\mz$
and define the multiplicative cocycle
\begin{equation}\label{eq:15101}
    f_{p,v}(g) \defi \frac{\norm{g(v\wedge u_p)}}{\norm{v\wedge u_p}}\Big/ \frac{\norm{g u_p}}{\norm{u_p}}.
\end{equation}
Note that the definition of $f_{p,v}$ is independent of the choice of $u_p$. We set $$\cF' \defi \set{f_{p,v}}_{(p,v)\in \cD}.$$
\begin{proposition}\label{prop:unimuexp}
    %Let $\ka>0$ be small enough such that $u_X^\ka\in \on{CH}_\mu(X)$ (see Proposition~\ref{prop:contracted func on X}),
    %and for $M>0$
    %let $\mu_{M,y}$ be the Markov measures
    %associated to the first return cocycle to $X_M$ as defined above.
    Suppose that {\bfseries S\ref{setting1}} and {\bfseries S\ref{setting2}} hold.
    Then,
    for all large enough $M$ and  $y\in Y_M$, the family $\cF'$ is uniformly $(\del_0,I_0,L_0)$-expanded by $\mu_{M,y}$. The parameters
    $(\del_0,I_0,L_0)$ may depend on $M$ but not on $y$.
\end{proposition}
\begin{proof}
    Take $M$ large enough so that Proposition~\ref{lem:prop 6.7} holds for the submultiplicative function
    $N(g) \defi \norm{g}_{\wedge^4\gog}\norm{g^{-1}}_{\wedge^3\gog}$.
    Since $\sup_{f\in\cF'}f(g)\leq N(g)$ it follows that there exists $I_0>0$ and $\del_0>0$ such that
    %$\int_G N(g)^{\del_0}\mathrm d\mu_{M,y} \le I_0$
    %for all $y\in Y_M$. We conclude that if  $f = f_{p,v}\in\cF$ then on choosing $0\ne u\in \wedge^3\gor_p$ we have
    %$f_{p,v}(g) = \frac{\norm{g(v\wedge u)}}{\norm{v\wedge u}}\Big/\frac{\norm{g u}}{\norm{u}} \le N(g)$. It follows that
    $$\int_G \sup_{f\in \cF'} f^{\del_0}(g) \dv{\mu_{M,y}} g\le I_0$$ for all $y\in Y_M$. This verifies condition~\eqref{enu:1st deriv cond 1} of
    Definition~\ref{def:uniform mu exp}.

    To verify condition~\eqref{enu:1st deriv cond 2} of Definition~\ref{def:uniform mu exp} we argue as follows. Assume
    $M$ is large enough
    so that Proposition~\ref{prop:martingales} is applicable and consider the additive cocycle $G\times \cD\to \bR$ given by
    $(g,p,v)\mapsto \log f_{p,v}(g)$. Condition~\eqref{ass:101} of Proposition~\ref{prop:martingales} is satisfied with some $J_0$
    because $\supp \mu$ is compact and condition~\eqref{ass:102} of Proposition~\ref{prop:martingales} is satisfied with $L_0$ as in~\eqref{eq:22064}
    by the discussion in \S\ref{ssec:replacing}.
    As an outcome we deduce equation~\eqref{eq:1053} which reads as
    $$\inf_{y\in Y_M}\inf_{(p,v)\in \cD}\int \log f_{p,v}(g)\dv{\mu_{M,y}}g\ge L_0.$$
    This is exactly condition~\eqref{enu:1st deriv cond 2} of Definition~\ref{def:uniform mu exp}.
\end{proof}
As an immediate corollary of Proposition~\ref{prop:unimuexp} and Lemma~\ref{lem:bounded compact sets/first derivative positive} we have the following.
\begin{corollary}\label{cor:expansion2}
    Suppose that {\bfseries S\ref{setting1}} and {\bfseries S\ref{setting2}} hold.
    Then, for all large enough $M$ there exists $\del_1 >0$
    such that for all $0<\del<\del_1$ there exists $0<c=c(\del)<1$ such that for all $(p,v)\in \cD$, $y\in Y_M$
    and $u_p\in \wedge^3 \gor_p\mz$ one has
    \begin{equation}\label{eq:1008}
        \int_G \pa[\bigg]{\frac{\norm{g(v\wedge u_p)}}{\norm{gu_p}}}^{-\del}   \dv{\mu_{M,y}} g \le c
        \pa[\bigg]{\frac{\norm{(v\wedge u_p)}}{\norm{u_p}}}^{-\del}.
    \end{equation}
    %such that the family $\cF = \set{f_{p,v}:(p,v)\in \cD}$ is
    %uniformly $\mu_{M,y}$-expanded with parameters $(\del_0,I_0,L_0)$ for any $y\in X_M$.
\end{corollary}
%%%%%%%%%%%%%%%%%%%%%%%%%%%%%%

%<---
%%
%--->proof of no atoms

\subsection{Proof of non-atomicity of the limit measures}\label{ssec:proof of non atomicity}
%%%%%%%%%%%%%%%%%%%%%%%%%%%%%%%%%
We now apply the results of \S\ref{subsec:Criterion-for-positive} to the space $Y = X\times X$.
We fix once and for all $\ka>0$ such that $u_X^\ka\in \on{CH}_\mu(X)$ as in Proposition~\ref{prop:contracted func on X}.
We then take 
$$u_{X\times X}(x,y) \defi u_X^\ka(x) + u_X^\ka(y),$$ so that $u_{X\times X}\in \on{CH}_\mu(X\times X)$.
We use the notions and notation of \S\ref{subsec:Criterion-for-positive} for the $G$-action on $X\times X$ and the contracted
function $u_{X\times X}$. In particular, we have the sublevel sets $(X\times X)_M$,
the first return times $\rho_{M,(x,y)}$, the transition probability measures $\mu_{M,(x,y)}$ and the Markov operator $\on A_{M,\mu}$.

%%%%%%%%%%%%%%%%%%%%%%%%%%%%%%%%%%%%%%%%%%%%%%%%

Equipped with the above theory we are finally in a position to prove the following proposition which shows that
one can apply Proposition~\ref{prop:Criterion} in our setting.
%%%%%%%%%%%%%%%%%%%%%%
\begin{proposition}\label{prop:cf2}
    For all large enough $M>0$
    %and all compact subsets $Z\subset\left(X\times X\right)_{M}\smallsetminus\Delta_{X}$,
    there exists a continuous proper function $v_{M}:\left(X\times X\right)_{M}\smallsetminus \Del_X\to [0,\infty)$
    %which is bounded on compact subsets of $(X\times X)_M\smallsetminus \Del_X$
    %and infinite on $\Delta_{X}\cap\left(X\times X\right)_{M}$
    such that 
    $$v_{M}\in\mathrm{CH}_{\on A_{M,\mu}}\pa{\left(X\times X\right)_{M}\smallsetminus \Del_X}.$$
\end{proposition}
\begin{proof}
    %[Proof of Proposition~\ref{prop:cf2}]
    Consider the submultiplicative function
    $$R(g) \defi \max\set{\norm{g^{\pm 1}}_{\gog}, \norm{g^{\pm 1}}_{\wedge^4\gog} \norm{g^{\mp1}}_{\wedge^3\gog}}.$$
    We choose $M$ large enough and $\del$ small enough
    so that Corollary~\ref{cor:expansion2} holds and
    Proposition~\ref{lem:prop 6.7} yields
    \begin{equation}\label{eq:1048}
        \sup_{(x,y)\in (X\times X)_M}\int_G R(g)^{2\del} \dv{\mu_{M,(x,y)}} g \eqqcolon  J<\infty.
    \end{equation}
    We let $c<1$ be
    the constant satisfying~\eqref{eq:1008}.
    Define
    $$v_M(x,y) \defi \dx(x,y)^{-\del}$$
    and view it as a function on $(X\times X)_M\smallsetminus \Del_X$. Here $\dx$ is the metric discussed in \S\ref{ssec:metric}.
    It is then clear that it is continuous and proper. We are
    left to prove the existence of constants $c'<1$ and $b'$ such that for all $(x,y)\in (X\times X)_M$,
    $$\on A_{M,\mu}v_M(x,y) = \int_G v_M(gx,gy)\dv{\mu_{M,(x,y)}} g\le c'v_M(x,y) + b'.$$
    We will estimate the integral by splitting the domain of integration $G$ into a compact piece and its complement and
    treat each piece separately. This splitting is done using the submultiplicative function $R$ introduced above.

    For any $T>0$ let $G^{\leq T}\defi\set{g\in G:R(g)\leq T}$ and $G^{>T}\defi G\smallsetminus G^{\leq T}$. Then we have
    \begin{equation}\label{eq:0924}
        \on A_{M,\mu}v_M(x,y)= I_1 + I_2,
    \end{equation}
    where,
    \begin{equation*}
        I_1 \defi  \int_{G^{\leq T}} v_M(gx,gy)\dv{\mu_{M,(x,y)}}g\quad\textrm{and}\quad        I_2 \defi  \int_{G^{>T}}v_M(gx,gy)\dv{\mu_{M,(x,y)}}g.
    \end{equation*}
    %and
    %\begin{equation*}
    %    I_2 \defi  \int_{G^{>T}}v_M(gx,gy)\dv{\mu_{M,(x,y)}}g.
    %\end{equation*}
    We first estimate $I_1$.
    Fix a large $T$ and a constant $c_1<1$ and apply Lemma~\ref{lem:nbd} to the compact set $G^{\leq T}$
    and the constant $c_1$ to
    obtain a neighbourhood of the diagonal $\cU = \cU_{T,c_1}$ satisfying the conclusion of Lemma~\ref{lem:nbd}.
    The integral $I_1$ is bounded separately according to whether $(x,y)\in \cU$ or not:

    \noindent \tb{Case 1}. Assume $(x,y) \in (X\times X)_M\smallsetminus \cU$.
    By compactness
    $$\theta\defi\inf\set{\dx(gx',gy'): (x',y')\in (X\times X)_M\smallsetminus \cU,\; g \in G^{\leq T}} >0.$$
    Hence in this case we have
    $$I_1 = \int_{G^{\leq T}} \dx(gx,gy)^{-\del} \dv{\mu_{M,(x,y)}} g \le \theta^{-\del}.$$
    \tb{Case 2}. Assume $(x,y)\in (X\times X)_M\cap \cU$.  Let
    %$p\in \Xbar$ be the plane of $x$ and choose
    $u\in \wedge^3\gor_{\pi(x)}\mz$. Then, by applying parts~\eqref{p:0932} and~\eqref{p:0933} of Lemma~\ref{lem:nbd}
    %$\frac{\norm{g(o_{x,y}\wedge u)}}{\norm{gu}}$,
    and using Corollary~\ref{cor:expansion2} we get
    \begin{align*}
        I_1 &= \int_{G^{\leq T}} \dx(gx,gy)^{-\del} \dv{\mu_{M,(x,y)}} g\\
            &\le  c_1^{-2\del}\int_{G^{\leq T}}\pa[\bigg]{ \frac{\norm{g(o_{x,y}\wedge u)}}{\norm{gu}}}^{-\del} \dv{\mu_{M,(x,y)}} g\\
            & \le c_1^{-2\del} c \pa[\bigg]{ \frac{\norm{(o_{x,y}\wedge u)}}{\norm{u}}}^{-\del}\le c_1^{-4\del} c \dx(x,y)^{-\del}.
    \end{align*}
    In total, combining the two cases we arrive at the bound $$I_1\le c_1^{-4\del} c v_M(x,y) + \theta^{-\del}.$$

    We now bound $I_2$.
    Notice that $R(g)=R(g^{-1})$ and by ~\eqref{eq:1525}, for all
    $x,y\in X$ and $g\in G$ one has $\dx(gx,gy)\le R(g)\dx(x,y)$.
    It follows that $\dx(gx,gy)^{-\del}\le R(g)^\del \dx(x,y)^{-\del}$.
    Therefore by~\eqref{eq:1048} we have that
    \begin{align*}
        I_2& = \int_{G^{>T}} v_M(gx,gy)\dv{\mu_{M,(x,y)}} g\\
           &\le v_M(x,y)\int_{G^{>T}} R(g)^\del \dv{\mu_{M,(x,y)}} g\\
           &\le v_M(x,y)T^{-\del} \int_{G^{>T}} R(g)^{2\del} \dv{\mu_{M,(x,y)}} g\le T^{-\del} J v_M(x,y).
    \end{align*}
    Combining the estimates for $I_1$ and $I_2$ we conclude that for all $(x,y)\in (X\times X)_M$ one has
    $$\on A_{M,\mu}v_M(x,y)\le (c_1^{-4\del}c + T^{-\del}J) v_M(x,y)  + \theta^{-\del}.$$
    If we choose $c_1$ close enough to 1 and $T$ large enough then the constant
    $c' \defi c_1^{-4\del}c + T^{-\del}J $ is strictly less than 1, which completes the proof. Note that $\theta$ depends on
    $T$ and $c_1$ (which determine $\cU$) but that does not matter.
\end{proof}

Finally we conclude the proof
%of Theorem~\ref{thm:Let--be non atomic} which
%is a restatement
of Theorem~\ref{thm:full}\eqref{main2}.
\begin{proof}[Proof of Theorem~\ref{thm:full}\eqref{main2}]
    Notice that since $\pi_*\nu = \bar{\nu}_{\Xbar}$ is the Furstenberg measure of $\mu$ on $\Xbar$ which 
    is non-atomic (see Remark~\ref{rem:non-atomic}), $\nu$ is non-atomic as well.
    By Proposition~\ref{prop:cf2} we may apply the criterion for the non-atomicity of the limit measures given
    in Proposition~\ref{prop:Criterion} which concludes the proof.
\end{proof}

%%%%%%%%%%%%%%%%%%%%

%<---
%%
%<---
%%
%--->appendix
\appendix
\section{Diophantine approximation and the geometry of numbers}\label{appendix}
In the past decades Furstenberg has been promoting an approach to attack a famous
open problem in Diophantine approximation: \textit{Are cubic numbers well approximable?} Recall that a number $\al\in \bR$
is well approximable if the coefficients $a_i$ in its continued fraction expansion $\al = \br{a_0;a_1,a_2,\dots}$ form an unbounded
sequence of integers. Lagrange's theorem asserts that $\al$ is a quadratic irrational number if and only if its continued fraction
expansion is eventually periodic and hence clearly not well approximable.
This could be proved by translating the problem into a dynamical problem about the action of
the diagonal group $a_t = \diag{e^t,e^{-t}}$ acting on the space of lattices in the plane $\PGL_2(\bR)/\PGL_2(\bZ)$. Furstenberg's
approach says that the dynamical system $a_t\curvearrowright  \PGL_2(\bR)/\PGL_2(\bZ)$ which is
tailored to detect quadratics, should be replaced with a dynamical system which is tailored to detect cubic irrationals. He then suggests the following characterisation of well approximability in terms of the dynamics on the space of 2-lattices $\X$ discussed
in this paper.
\begin{theorem}[Furstenberg - unpublished]\label{Furstenberg's criterion}
    Let $A$
    %\set{\diag{e^t,e^s,e^{-(t+s)}}\in G:t,s\in\bR}$
    denote the
    %connected component of the identity of the
    group of diagonal matrices in $G$ and let $x = \br{\Lam}\in X$ be the homothety class of the 2-lattice $\Lam = \spa_\bZ\set{v,w}$ spanned by $v,w\in\bR^3\mz$. Assume that $\Lam\cap p = \set{0}$ for $p = \spa_\bR\set{e_i,e_j}$ any one of the three planes fixed by $A$. Then the orbit $Ax$ is unbounded in $X$ if and only if one of the ratios
    $v_i/w_i$ is well approximable, for $i=1,2,3$.
\end{theorem}

Before we proceed to show that the dynamical system $A\curvearrowright \X$ can detect cubic numbers in a certain sense, we introduce a notion in \textit{geometry of numbers} that will be important for our discussion.
Consider a lattice $L\subset\bR^3$ and an $L$-rational line $W$ (where $W$ is $L$-rational if $L\cap W \ne\set{0}$).
We define the
\textit{directional $2$-lattice} $$L_W \defi \pi^{W}(L)$$
where $\pi^W$ denotes the orthogonal projection onto $W^\perp$. The term comes from visualising $L_W$ as
representing what $L$ looks like when one is looking in the direction of $W$.
We set
$$\goD(L) = \set{\br{L_W}:W\in\bP\bR^3\textrm{ is $L$-rational}}\subset X.$$
We now wish to describe subcollections of $\goD(L)$ which are obtained by conditioning on $W$. Note that an $L$-rational line
$W\in\bP\bR^3$ is characterised by the generator $v_W$ (well defined up to sign) of $L\cap W$. Given a subset
$S\subset \bR^3$ we define the set of conditioned directional lattices defined by $L$ and $S$ to be
$$\goD_S(L) = \set{\br{L_W}:W\in\bP\bR^3\textrm{ is $L$-rational and }v_W\in S}\subset X.$$
Sometimes, instead of writing $L_W$ we write $L_v$, where $v=v_W$.

Consider a lattice $L\subset \bR^3$ which is obtained in the following manner.
Let $\bK$ be a totally real number field of degree 3 over $\bQ$ and for $i=1,2,3$ let $\sig_i$ be its distinct embeddings into the reals. Let $\vphi:\bK\to \bR^3$
be the so-called geometric embedding given by $\vphi(\al) \defi (\sig_i(\al))_1^3$.
It is well known that if $\cO_\bK$ denotes the ring of integers in $\bK$ then $L\defi \vphi(\cO_\bK)$ is a lattice in $\bR^3$.
Let $\on{N}:\bR^3\to \bR$ denote the cubic form given by $\on{N}(v)=v_1 v_2 v_3$, so that for $\al\in\bK$ one has $N_{\bK/\bQ}(\al) = \on{N}(\vphi(\al))$.

The lattice $L$ has a very special
relationship with  the surface 
$$S\defi\set{v\in\bR^3 :\on{N}(v)=\pm1}.$$ Namely,
$$S\cap L = S\cap L_{\on{prim}} = \set{\vphi(\al): \al\in \cO_{\bK}^\times}
\footnote{$L_{\on{prim}}$ denotes the collection
of primitive vectors in $L$.}.$$
In particular, the group
\begin{equation}\label{eq:1712}
    A_L\defi \set{\diag{\sig_1(\al),\sig_2(\al),\sig_3(\al)}:\al\in \cO_\bK^\times},
\end{equation}
which has a finite index subgroup which is a lattice 
in $A\simeq \bR^2$ by Dirichlet's unit theorem, acts transitively and simply on $S\cap L$.
This furnishes the link between Furstenberg's criterion for well approximability and cubic numbers.
\begin{corollary}\label{cor:ap}
    Let $\bK$, $L=\vphi(\cO_\bK)$ and $S=\set{v\in \bR^3: \on{N}(v) =\pm 1}$ be as above. Then the collection of conditioned directional lattices
    $\goD_S(L)$ is unbounded in $X$ if and only if for some $1\le i\le 3$, $\sig_i(\bK)\smallsetminus \bQ$ is composed of well approximable numbers.
\end{corollary}
\begin{proof}
    For $v\in\bR^3$ we denote by $\pi^v:\bR^3\to v^\perp$ the orthogonal projection with kernel $\bR v$.
    Let $\mb{1} = (1,1,1)\in L$ and denote
    $\Lam = \pi^\mb{1}(L)$. Since the group in~\eqref{eq:1712} contains a finite index subgroup which is 
    cocompact in $A$,
    the unboundedness of $A\br{\Lam}$ is equivalent to the unboundedness of $A_L\br{\Lam}$. But
    \begin{align*}
        A_L\br{\Lam} &= \set{\diag{\sig_1(\al),\sig_2(\al),\sig_3(\al)}([\pi^{\mb{1}}(L)]) :\al\in \cO_\bK^\times}\\
                     & = \set{[\pi^{\vphi(\al)}(L)]:\al\in\cO_\bK^\times} =\set{\br{\pi^v(L)}: v\in L\cap S} = \goD_S(L),
    \end{align*}
    where we have used the fact that for $g\in G$ and $v\in L$, $g\pi^v(L) = \pi^{gv}(g^{-t}L)$ and the fact that for $a\in A_L$, we have that $a^{-t}L=L$.

    %Let $\al',\be' \in \cO_\bK$ be such that
    Let $\al,\be\in\cO_\bK$ be such that $\set{1,\al,\be}$ forms a basis of $\cO_\bK$ over $\bQ$.
    %$\set{\mb 1,\vphi(\al),\vphi(\be)}$ is a basis of $L$.
    Denote $\al' = \al - \frac{1}{3}\on{Tr}(\al)$ and similarly denote $\be' = \be- \frac{1}{3}\on{Tr}(\be)$.
    It follows that
    $\Lam$ is spanned by $\vphi(\al')$ and  $\vphi(\be')$.
    Hence, by Furstenberg's criterion (Theorem~\ref{Furstenberg's criterion}) we deduce that $\goD_S(L)$ is unbounded
    if and only if there exists $1\le i\le 3$ such that the ratio $\sig_i(\al'/\be')$ is well approximable.
    Now it is not hard to see that for a given cubic real field $\bF$,
    since $\bF\smallsetminus \bQ$ is a single $\GL_2(\bQ)$-orbit 
    (under the action by M\"obius transformations),
    and since this action preserves well approximability,
    then either all elements of $\bF\smallsetminus \bQ$ are well approximable or non of them are.
    This concludes the proof of the Corollary.
\end{proof}
Let $$\Del\defi\set{p\in \Xbar: p \textrm{ contains one of the axis $\bR e_i$, $i=1,2,3$}}$$
be the projective triangle defined by the axis.
Recasting in terms of conditioned directional lattices the common belief that real cubic numbers should be well approximable and even generic for the Gauss map we propose the following.
\begin{conjecture}\label{conj:2}
    Let $L$ and $S$ be as in Corollary~\ref{cor:ap}. Then the closure in $X$ of $\goD_S(L)$ contains $\pi^{-1}(\Del)$.
\end{conjecture}
In the introduction we stated Conjecture~\ref{conj:1} which follows from the following conjecture. We view it as an analogue
to Conjecture~\ref{conj:2}.
We use the notation $V_Q^1 =  \set{v:Q(v)=1}$ for $Q(v)\defi2v_1v_3-v_2^2$ from the introduction
and also $\cC\subset \Xbar$ for the circle of isotropic subspaces
defined before Theorem~\ref{thm:cover}.
\begin{conjecture}
    The closure of $\goD_{V_Q^1}(\bZ^3)$ contains $\pi^{-1}(\cC)$.
\end{conjecture}

%\section{A projective line lift invariant under an arithmetic group by Uri Bader}\label{appendixofbader}
%\pagebreak
%\input{appendix}

\bibliographystyle{amsalpha}
\bibliography{ref}

\def\cprime{$'$}
\begin{bibdiv}
\begin{biblist}

\bib{MR2405998}{article}{
      author={Abels, Herbert},
       title={Proximal linear maps},
        date={2008},
        ISSN={1558-8599},
     journal={Pure Appl. Math. Q.},
      volume={4},
      number={1, Special Issue: In honor of Grigory Margulis. Part 2},
       pages={127\ndash 145},
         url={http://dx.doi.org/10.4310/PAMQ.2008.v4.n1.a5},
      review={\MR{2405998}},
}

\bib{BFLM}{article}{
      author={Bourgain, Jean},
      author={Furman, Alex},
      author={Lindenstrauss, Elon},
      author={Mozes, Shahar},
       title={Stationary measures and equidistribution for orbits of nonabelian
  semigroups on the torus},
        date={2011},
        ISSN={0894-0347},
     journal={J. Amer. Math. Soc.},
      volume={24},
      number={1},
       pages={231\ndash 280},
         url={http://dx.doi.org/10.1090/S0894-0347-2010-00674-1},
      review={\MR{2726604}},
}

\bib{MR0147566}{article}{
      author={Borel, Armand},
      author={Harish-Chandra},
       title={Arithmetic subgroups of algebraic groups},
        date={1962},
        ISSN={0003-486X},
     journal={Ann. of Math. (2)},
      volume={75},
       pages={485\ndash 535},
         url={http://dx.doi.org/10.2307/1970210},
      review={\MR{0147566}},
}

\bib{MR2267655}{book}{
      author={Bogachev, V.~I.},
       title={Measure theory. {V}ol. {I}, {II}},
   publisher={Springer-Verlag, Berlin},
        date={2007},
        ISBN={978-3-540-34513-8; 3-540-34513-2},
         url={http://dx.doi.org/10.1007/978-3-540-34514-5},
      review={\MR{2267655 (2008g:28002)}},
}

\bib{BQAnnals}{article}{
      author={Benoist, Yves},
      author={Quint, Jean-Fran{\c{c}}ois},
       title={Mesures stationnaires et ferm\'es invariants des espaces
  homog\`enes},
        date={2011},
        ISSN={0003-486X},
     journal={Ann. of Math. (2)},
      volume={174},
      number={2},
       pages={1111\ndash 1162},
         url={http://dx.doi.org/10.4007/annals.2011.174.2.8},
      review={\MR{2831114}},
}

\bib{BQInventiones}{article}{
      author={Benoist, Yves},
      author={Quint, Jean-Fran{\c{c}}ois},
       title={Random walks on finite volume homogeneous spaces},
        date={2012},
        ISSN={0020-9910},
     journal={Invent. Math.},
      volume={187},
      number={1},
       pages={37\ndash 59},
         url={http://dx.doi.org/10.1007/s00222-011-0328-5},
      review={\MR{2874934}},
}

\bib{BQAnnals2}{article}{
      author={Benoist, Yves},
      author={Quint, Jean-Fran\c{c}ois},
       title={Stationary measures and invariant subsets of homogeneous spaces
  ({III})},
        date={2013},
        ISSN={0003-486X},
     journal={Ann. of Math. (2)},
      volume={178},
      number={3},
       pages={1017\ndash 1059},
         url={http://dx.doi.org/10.4007/annals.2013.178.3.5},
      review={\MR{3092475}},
}

\bib{BQJams}{article}{
      author={Benoist, Yves},
      author={Quint, Jean-Fran{\c{c}}ois},
       title={Stationary measures and invariant subsets of homogeneous spaces
  ({II})},
        date={2013},
        ISSN={0894-0347},
     journal={J. Amer. Math. Soc.},
      volume={26},
      number={3},
       pages={659\ndash 734},
         url={http://dx.doi.org/10.1090/S0894-0347-2013-00760-2},
      review={\MR{3037785}},
}

\bib{MR3560700}{book}{
      author={Benoist, Yves},
      author={Quint, Jean-Fran\c{c}ois},
       title={Random walks on reductive groups},
      series={Ergebnisse der Mathematik und ihrer Grenzgebiete. 3. Folge. A
  Series of Modern Surveys in Mathematics [Results in Mathematics and Related
  Areas. 3rd Series. A Series of Modern Surveys in Mathematics]},
   publisher={Springer, Cham},
        date={2016},
      volume={62},
        ISBN={978-3-319-47719-0; 978-3-319-47721-3},
      review={\MR{3560700}},
}

\bib{MR0578655}{article}{
      author={Dani, S.~G.},
       title={Invariant measures of horospherical flows on noncompact
  homogeneous spaces},
        date={1978},
        ISSN={0020-9910},
     journal={Invent. Math.},
      volume={47},
      number={2},
       pages={101\ndash 138},
      review={\MR{0578655}},
}

\bib{DaniRaghavan}{article}{
      author={Dani, S.~G.},
      author={Raghavan, S.},
       title={Orbits of {E}uclidean frames under discrete linear groups},
        date={1980},
        ISSN={0021-2172},
     journal={Israel J. Math.},
      volume={36},
      number={3-4},
       pages={300\ndash 320},
         url={http://dx.doi.org/10.1007/BF02762053},
      review={\MR{597457}},
}

\bib{MR2247967}{article}{
      author={Einsiedler, Manfred},
      author={Katok, Anatole},
      author={Lindenstrauss, Elon},
       title={Invariant measures and the set of exceptions to {L}ittlewood's
  conjecture},
        date={2006},
        ISSN={0003-486X},
     journal={Ann. of Math. (2)},
      volume={164},
      number={2},
       pages={513\ndash 560},
         url={http://dx.doi.org/10.4007/annals.2006.164.513},
      review={\MR{2247967}},
}

\bib{ELpisa}{incollection}{
      author={Einsiedler, M.},
      author={Lindenstrauss, E.},
       title={Diagonal actions on locally homogeneous spaces},
        date={2010},
   booktitle={Homogeneous flows, moduli spaces and arithmetic},
      series={Clay Math. Proc.},
      volume={10},
   publisher={Amer. Math. Soc., Providence, RI},
       pages={155\ndash 241},
      review={\MR{2648695 (2011f:22026)}},
}

\bib{EskinLindenstrauss}{article}{
      author={Eskin, Alex},
      author={Lindenstrauss, Elon},
       title={Random walks on locally homogeneous spaces},
        date={2018},
     journal={Preprint},
         url={https://www.math.uchicago.edu/~eskin/RandomWalks/paper.pdf},
}

\bib{MR2087794}{incollection}{
      author={Eskin, Alex},
      author={Margulis, Gregory},
       title={Recurrence properties of random walks on finite volume
  homogeneous manifolds},
        date={2004},
   booktitle={Random walks and geometry},
   publisher={Walter de Gruyter GmbH \& Co. KG, Berlin},
       pages={431\ndash 444},
      review={\MR{2087794 (2005m:22025)}},
}

\bib{EskinMirzahani}{article}{
      author={{Eskin}, A.},
      author={{Mirzakhani}, M.},
       title={{Invariant and stationary measures for the SL(2,R) action on
  Moduli space}},
        date={2013-02},
     journal={ArXiv e-prints},
      eprint={1302.3320},
}

\bib{MR0121828}{article}{
      author={Furstenberg, H.},
      author={Kesten, H.},
       title={Products of random matrices},
        date={1960},
        ISSN={0003-4851},
     journal={Ann. Math. Statist.},
      volume={31},
       pages={457\ndash 469},
         url={http://dx.doi.org/10.1214/aoms/1177705909},
      review={\MR{0121828}},
}

\bib{MR727020}{article}{
      author={Furstenberg, H.},
      author={Kifer, Y.},
       title={Random matrix products and measures on projective spaces},
        date={1983},
        ISSN={0021-2172},
     journal={Israel J. Math.},
      volume={46},
      number={1-2},
       pages={12\ndash 32},
         url={http://dx.doi.org/10.1007/BF02760620},
      review={\MR{727020 (85i:22010)}},
}

\bib{FurmanHandbook}{incollection}{
      author={Furman, Alex},
       title={Random walks on groups and random transformations},
        date={2002},
   booktitle={Handbook of dynamical systems, {V}ol.\ 1{A}},
   publisher={North-Holland, Amsterdam},
       pages={931\ndash 1014},
         url={http://dx.doi.org/10.1016/S1874-575X(02)80014-5},
      review={\MR{1928529}},
}

\bib{MR0163345}{article}{
      author={Furstenberg, Harry},
       title={Noncommuting random products},
        date={1963},
        ISSN={0002-9947},
     journal={Trans. Amer. Math. Soc.},
      volume={108},
       pages={377\ndash 428},
      review={\MR{0163345 (29 \#648)}},
}

\bib{MR0146298}{article}{
      author={Furstenberg, Harry},
       title={A {P}oisson formula for semi-simple {L}ie groups},
        date={1963},
        ISSN={0003-486X},
     journal={Ann. of Math. (2)},
      volume={77},
       pages={335\ndash 386},
         url={http://dx.doi.org/10.2307/1970220},
      review={\MR{0146298}},
}

\bib{MR0284569}{incollection}{
      author={Furstenberg, Harry},
       title={Random walks and discrete subgroups of {L}ie groups},
        date={1971},
   booktitle={Advances in {P}robability and {R}elated {T}opics, {V}ol. 1},
   publisher={Dekker, New York},
       pages={1\ndash 63},
      review={\MR{0284569}},
}

\bib{MR1040268}{article}{
      author={Gol{\cprime}dshe{\u\i}d, I.~Ya.},
      author={Margulis, G.~A.},
       title={Lyapunov exponents of a product of random matrices},
        date={1989},
        ISSN={0042-1316},
     journal={Uspekhi Mat. Nauk},
      volume={44},
      number={5(269)},
       pages={13\ndash 60},
         url={http://dx.doi.org/10.1070/RM1989v044n05ABEH002214},
      review={\MR{1040268}},
}

\bib{KaimAnnals}{article}{
      author={Kaimanovich, Vadim~A.},
       title={The {P}oisson formula for groups with hyperbolic properties},
        date={2000},
        ISSN={0003-486X},
     journal={Ann. of Math. (2)},
      volume={152},
      number={3},
       pages={659\ndash 692},
         url={http://dx.doi.org/10.2307/2661351},
      review={\MR{1815698}},
}

\bib{KaimDokl}{article}{
      author={Kaimanovich, V.~A.},
       title={An entropy criterion of maximality for the boundary of random
  walks on discrete groups},
        date={1985},
        ISSN={0002-3264},
     journal={Dokl. Akad. Nauk SSSR},
      volume={280},
      number={5},
       pages={1051\ndash 1054},
      review={\MR{780288}},
}

\bib{MR1920389}{book}{
      author={Knapp, Anthony~W.},
       title={Lie groups beyond an introduction},
     edition={Second},
      series={Progress in Mathematics},
   publisher={Birkh\"auser Boston, Inc., Boston, MA},
        date={2002},
      volume={140},
        ISBN={0-8176-4259-5},
      review={\MR{1920389 (2003c:22001)}},
}

\bib{MR1619571}{article}{
      author={Katok, A.},
      author={Spatzier, R.~J.},
       title={Corrections to: ``{I}nvariant measures for higher-rank hyperbolic
  abelian actions'' [{E}rgodic {T}heory {D}ynam. {S}ystems {\bf 16} (1996), no.
  4, 751--778; {MR}1406432 (97d:58116)]},
        date={1998},
        ISSN={0143-3857},
     journal={Ergodic Theory Dynam. Systems},
      volume={18},
      number={2},
       pages={503\ndash 507},
         url={http://dx.doi.org/10.1017/S0143385798110969},
      review={\MR{1619571}},
}

\bib{LedrappierIMJ}{article}{
      author={Ledrappier, Fran\c{c}ois},
       title={Poisson boundaries of discrete groups of matrices},
        date={1985},
        ISSN={0021-2172},
     journal={Israel J. Math.},
      volume={50},
      number={4},
       pages={319\ndash 336},
         url={http://dx.doi.org/10.1007/BF02759763},
      review={\MR{800190}},
}

\bib{MR2195133}{article}{
      author={Lindenstrauss, Elon},
       title={Invariant measures and arithmetic quantum unique ergodicity},
        date={2006},
        ISSN={0003-486X},
     journal={Ann. of Math. (2)},
      volume={163},
      number={1},
       pages={165\ndash 219},
         url={http://dx.doi.org/10.4007/annals.2006.163.165},
      review={\MR{2195133}},
}

\bib{MR1919409}{incollection}{
      author={Nevo, Amos},
      author={Zimmer, Robert~J.},
       title={Actions of semisimple {L}ie groups with stationary measure},
        date={2002},
   booktitle={Rigidity in dynamics and geometry ({C}ambridge, 2000)},
   publisher={Springer, Berlin},
       pages={321\ndash 343},
      review={\MR{1919409}},
}

\bib{MR1933077}{article}{
      author={Nevo, Amos},
      author={Zimmer, Robert~J.},
       title={A structure theorem for actions of semisimple {L}ie groups},
        date={2002},
        ISSN={0003-486X},
     journal={Ann. of Math. (2)},
      volume={156},
      number={2},
       pages={565\ndash 594},
         url={http://dx.doi.org/10.2307/3597198},
      review={\MR{1933077}},
}

\bib{MR1135878}{article}{
      author={Ratner, Marina},
       title={On {R}aghunathan's measure conjecture},
        date={1991},
        ISSN={0003-486X},
     journal={Ann. of Math. (2)},
      volume={134},
      number={3},
       pages={545\ndash 607},
         url={http://dx.doi.org/10.2307/2944357},
      review={\MR{1135878 (93a:22009)}},
}

\bib{2016arXiv161105899S}{article}{
      author={{Simmons}, D.},
      author={{Weiss}, B.},
       title={{Random walks on homogeneous spaces and diophantine approximation
  on fractals}},
        date={2016-11},
     journal={ArXiv e-prints},
      eprint={1611.05899},
}

\bib{MR3237440}{book}{
      author={Tao, Terence},
       title={Hilbert's fifth problem and related topics},
      series={Graduate Studies in Mathematics},
   publisher={American Mathematical Society, Providence, RI},
        date={2014},
      volume={153},
        ISBN={978-1-4704-1564-8},
      review={\MR{3237440}},
}

\bib{Williams}{book}{
      author={Williams, David},
       title={Probability with martingales},
      series={Cambridge Mathematical Textbooks},
   publisher={Cambridge University Press, Cambridge},
        date={1991},
        ISBN={0-521-40455-X; 0-521-40605-6},
         url={http://dx.doi.org/10.1017/CBO9780511813658},
      review={\MR{1155402}},
}

\end{biblist}
\end{bibdiv}
\end{document}